\documentclass[11pt, reqno]{amsart}
\usepackage{tikz-cd}
\usepackage{amsfonts}
\usepackage{amssymb}
\usepackage{latexsym}
\usepackage[final]{hyperref}
\hypersetup{colorlinks=true, linkcolor=blue, anchorcolor=blue, citecolor=red, filecolor=blue, menucolor=blue, pagecolor=blue, urlcolor=blue}
\usepackage[]{enumerate}
\usepackage{verbatim}
\usepackage{stix2}

\usepackage{graphicx}       
\usepackage{amsmath}

\usepackage{tikz}
\usepackage{cmap}
\usepackage{ulem}
\usepackage{soul}
\usepackage{comment}
\usepackage{fancybox}
\usepackage{relsize}
\usepackage{pifont}
\usepackage{accents}

\renewcommand{\theequation}{\arabic{section}.\arabic{equation}}

\renewcommand{\d}{\, \mathrm{d}}
\renewcommand{\emph}[1]{\textit{#1}}
\newcommand{\Abs}[1]{\left\vert #1 \right\vert}
\newcommand{\abs}[1]{\vert #1 \vert}
\newcommand{\bigabs}[1]{\bigl\vert #1 \bigr\vert}
\newcommand{\Bigabs}[1]{\Bigl\vert #1 \Bigr\vert}
\newcommand{\biggabs}[1]{\biggl\vert #1 \biggr\vert}
\newcommand{\norm}[1]{\left\Vert #1 \right\Vert}
\newcommand{\fixednorm}[1]{\Vert #1 \Vert}
\newcommand{\bignorm}[1]{\bigl\Vert #1 \bigr\Vert}
\newcommand{\Bignorm}[1]{\Bigl\Vert #1 \Bigr\Vert}
\newcommand{\Sobnorm}[2]{\norm{#1}_{H^{#2}}}

\newcommand{\C}{\mathbb{C}}
\newcommand{\N}{\mathbb{N}}
\newcommand{\Z}{\mathbb{Z}}
\newcommand{\R}{\mathbb{R}}
\newcommand{\T}{\mathbb{T}}

\newcommand{\Innerprod}[2]{\left\langle \, #1 , #2 \, \right\rangle}
\newcommand{\innerprod}[2]{\langle \, #1 , #2 \, \rangle}
\newcommand{\biginnerprod}[2]{\bigl\langle \, #1 , #2 \, \bigr\rangle}
\newcommand{\japanese}[1]{\left\langle \, #1 \, \right\rangle}

\newcommand{\vertiii}[1]{{\left\vert\kern-0.25ex\left\vert\kern-0.25ex\left\vert #1 
    \right\vert\kern-0.25ex\right\vert\kern-0.25ex\right\vert}}

\newcommand{\define}[1]{\textbf{\emph{#1}}}

\DeclareMathOperator{\im}{Im}
\DeclareMathOperator{\re}{Re}
\DeclareMathOperator{\sgn}{sgn}
\DeclareMathOperator{\supp}{supp}
\DeclareMathOperator{\trace}{tr}

\newtheorem{theorem}{Theorem}
\newtheorem{proposition}{Proposition}
\newtheorem{lemma}{Lemma}
\newtheorem{corollary}{Corollary}

\theoremstyle{definition}
\newtheorem{definition}{Definition}

\theoremstyle{remark}
\newtheorem{remark}{Remark}

\title
[
	Stochastic Dirac-Klein-Gordon system 
]
{
	A conservative stochastic Dirac-Klein-Gordon system
}

\author[Dinvay]{Evgueni Dinvay}
\author[Selberg]{Sigmund Selberg}

\email{ Evgueni.Dinvay@uit.no }
\email{ Sigmund.Selberg@uib.no }

\address
{
	Department of Chemistry
	\\
    UiT The Arctic University of Norway
    \\
	PO Box 6050 Langnes
	\\
	N-9037 Tromsø
	\\
	Norway
}
\address
{
	Department of Mathematics
	\\
	University of Bergen
	\\
	PO Box 7803
	\\
	5020 Bergen
	\\
	Norway
}

\date{\today}


\begin{document}

\begin{abstract}
Considered herein is a particular nonlinear dispersive stochastic system consisting of Dirac and Klein-Gordon equations. They are coupled by nonlinear terms due to the Yukawa interaction. We consider a case of homogeneous multiplicative noise that seems to be very natural from the perspective of the least action formalism. We are able to show existence and uniqueness of a corresponding Cauchy problem in Bourgain spaces. Moreover, the regarded model implies charge conservation, known for the deterministic analogue of the system, and this is used to prove a global existence result for suitable initial data.
\end{abstract}

\maketitle
\tableofcontents

\section{Introduction}
\setcounter{equation}{0}

Consideration is given to the following 
Dirac-Klein-Gordon equations, one-dimensional in the space variable,
containing homogeneous multiplicative noise of the Stratonovich type,
\begin{equation}
\label{Stratonovich_Dirac}
\left \{
\begin{aligned}
	\left( -i \partial_t - i\alpha\partial_x + M\beta \right) \psi
	&=
	\phi \beta \psi
	+
	\beta \psi \xi_1
	,
	\\
	\left( \partial_t^2-\partial_x^2 + m^2 \right) \phi
	&=
	\psi^* \beta \psi
	+
	\phi \xi_2
	,
	\end{aligned}
\right.
\end{equation} 
with initial data
\begin{equation}\label{Data}
	\psi(0,x,\omega) = \psi_0(x,\omega)
	, \qquad
	\phi(0,x,\omega) = \phi_0(x,\omega)
	, \qquad
	\partial_t\phi(0,x,\omega) = \phi_1(x,\omega).
\end{equation}
The unknowns are random processes $\psi(t, x, \omega) \in \C^2$ and $\phi(t, x, \omega) \in \R$, for $t \ge 0$, $x \in \R$ and $\omega$ in a probability space $(\Omega,\mathcal F,\mathbb P)$. Here $m, M \geqslant 0$ are constants and
the $2 \times 2$ Dirac matrices $\alpha,\beta$ satisfy
\(
	\alpha = \alpha^*,
    \,
	\beta = \beta^*,
    \,
    \alpha^2 = \beta^2 = I
\)
and
\(
	\alpha \beta + \beta \alpha = 0.
\)
For simplicity we choose the particular representation
\[
	\alpha
	=
	\begin{pmatrix}
		1
		&
		0
		\\
		0
		&
		-1 
	\end{pmatrix}
	, \qquad
	\beta
	=
	\begin{pmatrix}
		0
		&
		1
		\\
		1
		&
		0 
	\end{pmatrix}
	.
\]
Let $W$ be the cylindrical Wiener process defined by
a complete orthonormal sequence
\(
	\{ e_k \} _{ k \in \mathbb N }
\)
in
\(
	L^2( \R, \R )
\)
and a sequence
\(
	\{ B_k \} _{ k \in \mathbb N }
\)
of independent real-valued Brownian motions on $(\Omega,\mathcal F,\{\mathcal F_t\}_{t \ge 0}, \mathbb P)$, where $\mathcal F_t$ is an associated filtration of $\mathcal F$.
We assume that the noise is of the form
\[
  \xi_j = \frac{d \mathcal W_j}{dt},
  \qquad
  \mathcal W_j = \mathfrak K_j W \qquad (j=1,2),
\]
where the $\mathfrak K_j$ are convolution operators
\begin{equation}
\label{convolution_operator}    
	\mathfrak K_jf(x)
	=
	\int_{ \R }
	\mathfrak k_j(x - y) f(y) dy
\end{equation}
with real-valued kernels $\mathfrak k_j \in H^{\sigma_j}( \R )$. Here $\sigma_j \ge 0$ will be chosen depending on the Sobolev regularity of the initial data.

Interpreting the stochastic integrals in the Stratonovich sense, we can then write \eqref{Stratonovich_Dirac} as
\begin{equation}\label{StratonovichDKG}
\left\{
\begin{gathered}
  d\psi = (-\alpha \partial_x - iM \beta) \psi \, dt + i\phi \beta \psi \, dt + i\beta \psi \mathfrak K_1 \circ dW,
  \\
  d\phi = \dot\phi \, dt,
  \\
  d\dot\phi = (\partial_x^2 - m^2) \phi \, dt + \psi^* \beta \psi \, dt + \phi \mathfrak K_2 \circ dW,
\end{gathered}
\right.
\end{equation}
where $\dot\phi = \partial \phi/\partial t$ and $\psi \mathfrak K_1$, $\phi \mathfrak K_2$ are understood as compositions of the convolution operators $\mathfrak K_1$, $\mathfrak K_2$ with the multiplication operators given by $\psi$, $\phi$. Thus
\[
  (\psi\mathfrak K_1)f(x)
  =
  \psi(x)\int_{ \R }
  \mathfrak k_1(x - y) f(y) dy
\]
and similarly for $\phi\mathfrak K_2$. Depending on the regularity of the initial data, the Sobolev regularity $\sigma_j$ of the kernel $\mathfrak k_j$ will be chosen so that the above compositions are Hilbert-Schmidt operators from $L^2(\R,\R)$ into suitable Sobolev spaces.

By introducing noise in the Stratonovich sense we respect two of the key physical
properties of the original deterministic Dirac-Klein-Gordon system:
the principle of least action and the conservation of the charge, $\int \Abs{\psi(t,x)}^2 \, dx$.

From an analysis perspective it is more convenient to work
with the It\^o stochastic integral. The It\^o form of the above system is
\begin{equation}\label{ItoDKG}
\left\{
\begin{gathered}
  d\psi = 
  (-\alpha \partial_x - iM \beta) \psi \, dt + i\phi \beta \psi \, dt
  - M_{\mathfrak K_1} \psi \, dt + i\beta \psi \mathfrak K_1 \, dW,
  \\
  d\phi = \dot\phi \, dt,
  \\
  d\dot\phi = (\partial_x^2 - m^2) \phi \, dt + \psi^* \beta \psi \, dt + \phi \mathfrak K_2 \, dW,
\end{gathered}
\right.
\end{equation}
where $M_{\mathfrak K_1} = (1/2) \norm{\mathfrak k_1}_{L^2}^2$. To see this, write \eqref{StratonovichDKG} in the abstract form
\[
  dX = AX dt + \mathcal N(X) dt + \mathcal M(X) \circ d W,
\]
where
\[
  X = \begin{pmatrix} \psi \\ \phi \\ \dot\phi \end{pmatrix},
  \qquad
  A = \begin{pmatrix} -\alpha \partial_x - iM\beta & 0 & 0 \\ 0 & 0 & 1 \\ 0 & \partial_x^2-m^2 & 0 \end{pmatrix},
  \qquad
  \mathcal M(X) = \begin{pmatrix} i\beta\psi \mathfrak K_1 \\ 0 \\ \phi \mathfrak K_2 \end{pmatrix}.
\]
Then, at least formally, the corresponding It\^o form is (see e.g.~\cite{Goodair2022})
\[
  dX = AX dt + \left( \mathcal N(X) + \frac12 \sum_k \mathcal M_k\left( \mathcal M_k(X) \right) \right) \, dt + \mathcal M(X) d W,
\]
where $\mathcal M_k(X) = \mathcal M(X)e_k$ and we calculate
\[
  \mathcal M_k\left( \mathcal M_k(X) \right)
  =
  \begin{pmatrix} (i \mathfrak K_1 e_k \beta)^2\psi \\ 0 \\ 0\end{pmatrix}
  =
  \begin{pmatrix} - (\mathfrak K_1 e_k)^2 \psi \\ 0 \\ 0\end{pmatrix}
\]
and note that for all $x \in \R$,
\[
  \sum_k \left( \mathfrak K_1e_k(x) \right)^2 = 
  \sum_k \Innerprod{\mathfrak k_1(x - \cdot)}{e_k}_{L^2}^2 =
  \norm{\mathfrak k_1}_{L^2}^2
\]
by Parseval's identity. Formally, this verifies the conversion from \eqref{StratonovichDKG} to \eqref{ItoDKG}.

Our aim is to prove existence and uniqueness for \eqref{ItoDKG}
with initial data $(\psi,\phi,\dot\phi)(0) = (\psi_0,\phi_0,\phi_1)$ in the spaces
\begin{equation}
\label{DataRegularity}
	\psi_0 \in L^2\left(\Omega,H^s \left( \R,\C^2 \right) \right)
	, \qquad
	\phi_0 \in L^2\left(\Omega,H^r( \R,\R )\right)
	, \qquad
	\phi_1 \in L^2\left(\Omega,H^{r-1}( \R, \R )\right)
	,
\end{equation}
for a certain range of Sobolev indices $s,r \in \R$.
In particular, using the charge conservation we will prove global existence when $s = 0$
and $1/4 < r < 1/2$,
under the additional assumption that $\psi_0 \in L^p \left( \Omega,L^2(\R) \right)$
for a sufficiently large $p \geq 4$, depending on $r$.

Assuming for the moment that $m > 0$, then by a rescaling we may take $m=1$. Applying the linear transformation $(\psi,\phi,\dot\phi) \mapsto (\psi_+,\psi_-,\phi_+,\phi_-)$ given by
\[
  \psi = \begin{pmatrix} \psi_+ \\ \psi_- \end{pmatrix},
  \qquad
  \phi=\phi_+ + \phi_-,
  \qquad
  \phi_\pm = \frac12 \left(\phi \pm \japanese{D_x}^{-1} i\dot\phi \right),
\]
where $D_x=-i\partial_x$ and $\japanese{\cdot} = \left( 1+\Abs{\cdot}^2 \right)^{1/2}$,
the Cauchy problem \eqref{ItoDKG}, \eqref{DataRegularity} then transforms to
\begin{equation}
\label{Ito-DKG-split}
\left\{
\begin{aligned}
    &-i d \psi_+ + D_x \psi_+ \, dt
	=
	-  M \psi_- \, dt
	+
	\phi \psi_- \, dt
	+
	\psi_- \mathfrak K_1 \, d W
	+
	iM_{\mathfrak K_1} \psi_+ \, dt,
    \\
    &- i d \psi_- - D_x \psi_- \, dt
	=
	-  M \psi_+ \, dt
	+
	\phi \psi_+ \, dt
	+
	\psi_+  \mathfrak K_1 \, d W
	+
	iM_{\mathfrak K_1} \psi_- \, dt,
    \\
    &-id \phi_+ + \japanese{D_x} \phi_+ dt = + \japanese{D_x}^{-1} \re \left( \overline{\psi_+} \psi_- \right) dt
    + \frac12 \japanese{D_x}^{-1} \phi \mathfrak K_2 \, dW,
    \\
    &-id \phi_- - \japanese{D_x} \phi_- dt = - \japanese{D_x}^{-1} \re \left( \overline{\psi_+} \psi_- \right) dt
    - \frac12 \japanese{D_x}^{-1} \phi \mathfrak K_2 \, dW,
\end{aligned}
\right.
\end{equation}
with
\begin{equation}
\label{Ito-DKG-split-data}
    \psi_\pm(0) = f_\pm \in L^2\left(\Omega,H^s(\R,\C)\right),
    \qquad
    \phi_\pm(0) = g_\pm \in L^2\left(\Omega,H^r(\R,\C)\right),
    \qquad
    \overline{g_+} = g_-.
\end{equation}
Here $\japanese{D_x}^{-1} \phi \mathfrak K_2$ is understood as a composition of operators. We remark that $\overline{\phi_+} = \phi_-$. Thus $\phi=\phi_+ + \overline{\phi_+}$ and it suffices to solve for $\psi_+$, $\psi_-$ and $\phi_+$. %

The deterministic Dirac-Klein-Gordon system has been extensively studied in space dimensions $d \le 3$. For $d=1$, the first global existence result was obtained in \cite{Chadam1973}, for $\psi_0 \in H^1(\R)$. This was improved to the charge class, that is, $\psi_0 \in L^2(\R)$, in \cite{Bournaveas2000}, by using space-time estimates of null form type. Both these results rely on the conservation of charge, of course. The complete null structure of the system, in dimensions $d \le 3$, was determined in \cite{Ancona_Foschi_Selberg2007}, and this opened the way for improvements in the low-regularity local well-posedness theory by using Bourgain's Fourier restriction norm spaces, see \cite{Ancona_Foschi_Selberg2007, Pecher2006, Selberg_Tesfahun,  Machihara2007, Machihara2010} and the references therein. Global existence in space dimension $d=2$ was proved in \cite{GP2010}, and for $d=3$ in \cite{Herr2014, Herr2016} for small data. Global existence below the charge in space dimension $d=1$ has been proved in \cite{Candy2012}. 

The optimal low-regularity result for the deterministic case in space dimension $d=1$ was obtained in \cite{Machihara2010}; it states that the problem is locally well posed for $s > -1/2$ and $\abs{s} \le r \le s+1$, and that this range is optimal since some form of ill-posedness holds outside it. In the present work we are primarily interested in getting a global result for $s=0$ in the presence of noise, and not so much in reaching the lowest possible regularity. Therefore, we restrict attention to the range $s > -1/4$, which corresponds to the results in \cite{Selberg_Tesfahun,  Machihara2007,Pecher2006} for the deterministic case. In those papers, the local existence proof is based on contraction in Bourgain spaces $X^{s,b}$, and that is also the approach we follow. However, the presence of noise introduces some further technical issues that have to be dealt with, including the fact that one has to work with $b<1/2$ (instead of $b > 1/2$ in the deterministic case), and that one has to introduce cutoffs to deal with the lack of uniform bounds with respect to the probabilistic variable $\omega$. 
In order to deal with these issues and obtain local well-posedness of the stochastic extension
\eqref{Stratonovich_Dirac}, we adopt techniques
developed in \cite{Bouard_Debussche2007} and based on analysis in Bourgain spaces.
To prove global existence we take advantage of the charge conservation,
\(
	\norm{\psi(t)} _{L^2} = \mathrm{const.}
	,
\)
and that is why we need a Stratonovich noise,
similar to one regarded in \cite{Bouard_Debussche1999}
with a nonlinear Schr\"odinger equation.
We work in Bourgain spaces of low time regularity ($b < 1/2$)
and so we need to extend product estimates proved in
\cite{Pecher2006, Selberg_Tesfahun}. Once these have been obtained, the local existence and uniqueness for \eqref{Ito-DKG-split} follows from an abstract framework for well-posedness of nonlinear dispersive PDE systems with homogeneous multiplicative noise, presented in Section \ref{Abstract-WP}.

As mentioned, we are motivated by ideas that were introduced in
\cite{Bouard_Debussche1999} and \cite{Bouard_Debussche2007}
to analyse the nonlinear Schr\"odinger equation (NLS) and the Korteweg-de Vries equation
with multiplicative noise.
They employ the truncation argument, and so do we.
Another approach worth mentioning
is the rescaling method developed in \cite{Barbu2014, Barbu2016} for stochastic NLS,
which was also used in study of scattering \cite{Herr2019}.
With this approach, the stochastic NLS is transformed to an equation with random coefficients.
This allows for pointwise estimations with respect to probability space, which in turn helps to avoid the use of cutoff estimates and provides a more general
result on $L^2$ theory of stochastic NLS compared to \cite{Bouard_Debussche1999}.
This approach relies on a generalisation of Strichartz estimates for a perturbed Schr\"odinger operator \cite{Marzuola2008}.

In the classical field theory one can determine the equations of motion
by the principle of least action.
As we restrict ourselves to the one dimensional space,
the action is an integral functional
\(
    \mathcal S
    =
    \int
    \mathcal L
    dtdx
    ,
\)
where the Lagrangian density, depending on the field and time,
is defined by the physical system under consideration.
The deterministic analogue of Equations \eqref{Stratonovich_Dirac}
is related to a particular choice of
the density $\mathcal L(\psi, \phi, t)$,
as explained in
\cite{Bjorken1964, Yukawa1935}.
We recall very briefly the corresponding physical background
and show how the noise can be naturally introduced here.
In particle physics, the Yukawa interaction \cite{Yukawa1935}
explains how forces between nucleons are mediated
by massive particles called mesons.
Mathematically, this is described by the action integral
\(
    \mathcal S(\psi,\phi)
\)
defined by the Lagrangian density
\[
    \mathcal L(\psi,\phi)
    =
    \mathcal L_{\mathrm{Dirac}}(\psi)
    +
    \mathcal L_{\mathrm{meson}}(\phi)
    +
    \mathcal L_{\mathrm{Yukawa}}(\psi,\phi)
    .
\]
Here $\psi$ is a spinor field (the fermion field) and
$\phi$ is a real scalar field (the meson field)
whose free-field dynamics are determined by the Lagrangians
\[
    \mathcal L_{\mathrm{Dirac}}(\psi)
    =
    \psi^*
    \left(
        i \partial_t + i \alpha \partial_x - M \beta
    \right)
    \psi
    , \qquad
    \mathcal L_{\mathrm{meson}}(\phi)
    =
    \frac 12 ( \partial_t \phi )^2
    -
    \frac 12 ( \partial_x \phi )^2
    -
    \frac12 m^2 \phi^2,
\]
corresponding to the free Dirac and Klein-Gordon equations. Here
$m,M \ge 0$ are masses and
$\psi^*$ denotes the complex conjugate transpose of $\psi$.
The interaction is determined by the Yukawa coupling term
\[
    \mathcal L_{\mathrm{Yukawa}}(\psi,\phi)
    =
    \phi \psi^* \beta \psi
    .
\]
The corresponding system of Euler-Lagrange equations
is the deterministic Dirac-Klein-Gordon system.
We introduce the noise by adding
\[
    \mathcal L_{\mathrm{noise}}(\psi, \phi, x, t)
    =
    \mathcal L_{\mathrm{Dirac}}^{\mathrm{noise}}(\psi, x,t)
    +
    \mathcal L_{\mathrm{meson}}^{\mathrm{noise}}(\phi, x,t)
    =
    \psi^* \beta \psi \xi_1
    +
    \frac 12 \phi^2 \xi_2
    .
\]
One can think about $\mathcal L_{\mathrm{Dirac}}^{\mathrm{noise}}$
as stochastic fluctuations of the initial Dirac potential
\(
    M \psi^* \beta \psi
    .
\)
Similarly,
$\mathcal L_{\mathrm{meson}}^{\mathrm{noise}}$
represents a noisy extension of the potential
\(
    m^2 \phi^2 / 2
\)
in the Klein-Gordon model.
Thus using the new Lagrangian density
\[
    \mathcal L(\psi, \phi, x, t)
    =
    \mathcal L_{\mathrm{Dirac}}(\psi)
    +
    \mathcal L_{\mathrm{meson}}(\phi)
    +
    \mathcal L_{\mathrm{Yukawa}}(\psi,\phi)
    +
    \mathcal L_{\mathrm{noise}}(\psi, \phi, x, t)
\]
one arrives at a stochastic variational principle
leading to the system \eqref{Stratonovich_Dirac}.
Note that the Stratonovich calculus obeys normal differentiation rules,
and so the derivation of the Euler-Lagrange equations
works out as in the deterministic case.

The paper is devoted to an analysis of existence and uniqueness of
a mild solution to the Cauchy problem for
 \eqref{Ito-DKG-split}
complemented with the initial data \eqref{Ito-DKG-split-data}.
It is organised as follows.
In the next section we introduce some preliminary notions
that will be used throughout the paper.
Section \ref{Main_result} provides the mild formulation
of the Cauchy problem and the statement of the main existence theorems.
Section \ref{Stochastic_integral} is devoted to an analysis
of the stochastic integrals we are dealing with.
Then in Section \ref{Bilinear_bounds} we prove bilinear estimates
necessary for treating nonlinear terms.
In Section \ref{Abstract-WP} we prove the local existence and uniqueness in an abstract setting.
As a result we obtain a local mild solution to
\eqref{Ito-DKG-split}, \eqref{Ito-DKG-split-data}.
Section \ref{Charge_conservation} is devoted to the proof
of charge conservation.
Finally, in Section \ref{Global_existence}
we prove existence of a global solution.
Proofs of some very technical results are left for the last three sections,
where we prove in general terms the so-called cutoff estimates.
The idea of making use of a Slobodeckij norm for this
comes from \cite{Bouard_Debussche2007}.
However, it turns out that the treatment should be more delicate
than the argument given in \cite{Bouard_Debussche2007}.

\section{Preliminaries}
\label{Prelims}
\setcounter{equation}{0}

First, we fix some general notational conventions.

As usual, the symbol $C$ will denote various positive constants, and its meaning can change from one instance to the next.

The characteristic function of a set $E$ will be denoted $\mathbb 1_E$. If $E$ is determined by some property $P$, say $E = \{ x \colon P(x) \}$, we will often use the convenient notation $\mathbb 1_{P(x)}$ for $\mathbb 1_E(x)$. If $E$ is a subset of a set $X$, and $f$ is a function defined on $E$, then by a slight abuse of notation we shall denote by $\mathbb 1_E f$ the extension of $f$ by zero outside $E$. We call this the \emph{trivial extension} (of $f$).

We will adhere to the following convention regarding restrictions of $\sigma$-algebras. Suppose that $\mathcal M$ is a $\sigma$-algebra and that $E \in \mathcal M$. Let $\mathcal M \vert_E$ be the $\sigma$-algebra on $E$ consisting of all sets $A \cap E$, where $A \in \mathcal M$. Then if $f \colon E \to H$ is $\mathcal M \vert_E$-measurable, we will simply say that it is $\mathcal M$-measurable. This is of course equivalent to saying that the trivial extension is $\mathcal M$-measurable.

We use the notation $a \wedge b = \min(a,b)$ for real numbers $a$ and $b$.

\subsection{Random variables}

We fix a filtered probability space $(\Omega,\mathcal F,\{\mathcal F_t\}_{t \ge 0},\mathbb P)$ admitting an independent sequence $\{B_k\}_{k \in \N}$ of one-dimensional Brownian motions. We write $\mathbb E(X) = \int_\Omega X(\omega) \, d\mathbb P(\omega)$ for $X \in L^1(\Omega)$.

A stochastic process $X(t)$, defined on a time interval $I = [S,T]$ or $I = [S,\infty)$, where $S \ge 0$, and taking values in a separable Hilbert space $H$, is said to be \emph{$H$-adapted} (or just \emph{adapted} if it is clear from the context which Hilbert space is meant) if $X(t)$ is $(\mathcal F_t,\mathcal B_H)$-measurable for all $t \in I$. In other words, $\innerprod{X(t)}{h}_H$ is $\mathcal F_t$-measurable for all $h \in H$. Here $\mathcal B_H$ denotes the Borel $\sigma$-algebra of $H$.

A process $Y(t)$ is a \emph{modification} of $X(t)$ if for each $t \in I$ we have $X(t)=Y(t)$ a.s.
We assume that $\mathcal F_0$ contains all sets in $\mathcal F$ with measure zero,
so that any modification of an adapted process is itself adapted.
Moreover, the filtration is supposed to be right-continuous,
i.e.
\(
    \bigcap_{s > t} \mathcal F_s = \mathcal F_t
\)
for any $t \geq 0$.

The process $X(t)$ is \emph{progressively measurable} if for each $t \in I$ the map $(s,\omega) \mapsto X(s,\omega)$, from $[S,t] \times \Omega$
into $H$,
is $\left(\mathcal B_{[S,t]} \otimes \mathcal F_t,\mathcal B_H \right)$-measurable.
Progressive measurability implies adaptedness (see \cite[Proposition 2.34]{Folland1999}), and the converse holds if the process has continuous paths (see \cite[Proposition 1.13]{Karatzas1991} or Lemma \ref{Stopping-lemma-2} below).

If $Z$ is some Banach space of functions from $[S,T]$ into $H$, we define
\begin{equation}\label{L2-prog}
  \mathbb L^2(\Omega,Z) = \left\{ u \in L^2(\Omega,Z) \colon \text{$u$ is progressively measurable} \right\}.
\end{equation}

\subsection{Stopping times}

A \emph{stopping time} is a random variable $\tau \colon \Omega \to [0,\infty]$ such that for all $t \ge 0$, the set
$\{ \tau \le t \} = \{ \omega \in \Omega \colon \tau(\omega) \le t \}$
is $\mathcal F_t$-measurable. Then also $\{ \tau < t \}$, $\{ \tau > t \}$ and $\{ \tau \ge t \}$ have this property, of course. Note that any constant $\tau \ge 0$ is a stopping time.

In the next three lemmas we establish some facts about stopping times.
We only consider strictly positive stopping times.

\begin{lemma}\label{Stopping-lemma-1}
Let $\tau \colon \Omega \to (0,\infty]$ be a stopping time. Then the set
\[
  E = \left\{ (s,\omega) \in [0,\infty) \times \Omega \colon 0 \le s < \tau(\omega) \right\}
\]
belongs to the product $\sigma$-algebra $\mathcal B_{[0,\infty)} \otimes \mathcal F$. For any $T > 0$, the set
\[
  E_T = E \cap \left( [0,T] \times \Omega \right)
\]
belongs to $\mathcal B_{[0,T]} \otimes \mathcal F_T$.
\end{lemma}

\begin{proof}
Let $A_n$ be the set of numbers $iT/2^n$, $i=0,\dots,2^n$. Then
$E_T = \bigcup_{n \in \N} \bigcup_{t \in A_n} [0,t] \times \{t < \tau\}$,
and of course $E = \bigcup_{N \in \N} E_N$.
\end{proof}

As a consequence of this lemma, if $u$ is a random variable defined on $E$ (so it is defined up to the time $\tau$), then it makes sense to ask whether $u$ is $\mathcal B_{[0,\infty)} \otimes \mathcal F$-measurable. That is, whether the trivial extension of $u$ has this property. Similarly, one can ask whether $u$ restricted to $E_t$ is $\mathcal B_{[0,t]} \otimes \mathcal F_t$-measurable for all $t \ge 0$, which amounts to progressive measurability of the trivial extension. The next lemma gives sufficient conditions for this to hold.

\begin{lemma}\label{Stopping-lemma-2}
Let $\tau \colon \Omega \to (0,\infty]$ be a stopping time, and let $E$ and $E_t$, for $t \ge 0$, be as in Lemma \ref{Stopping-lemma-1}. Let $u \colon E \to H$, where $H$ is a separable Hilbert space. Assume that $u$ has continuous paths, in the sense that
\begin{equation}\label{Stopping-lemma-2-1}
  \text{$t \mapsto u(t,\omega)$ is continuous on $[0,\tau(\omega))$, for each $\omega$},
\end{equation}
and assume that $u$ is adapted, in the sense that, for each $t \ge 0$ such that $\left\{ t < \tau \right\}$ is non-empty, we have
\begin{equation}\label{Stopping-lemma-2-2}
  \text{$\omega \mapsto u(t,\omega)$, defined for $\omega \in \{ t < \tau \}$, is $\mathcal F_t$-measurable.}
\end{equation}
Then $u \vert_{E_t}$ is $\mathcal B_{[0,t]} \otimes \mathcal F_t$-measurable for all $t \ge 0$. In other words, the trivial extension of $u$ is progressively measurable.
\end{lemma}

\begin{proof}
Let $U = \mathbb 1_{E} u$ be the trivial extension of $u$ to $[0,\infty) \times \Omega$. Then \eqref{Stopping-lemma-2-2} says that $U(t)$ is adapted for every $t \ge 0$. Now fix $t \ge 0$. For $n \in \N$ let $t_i = it/2^n$ for $i=0,\dots,2^n$, and define
\[
  U_n(s,\omega) =  U(0,\omega) \mathbb 1_{\{0\}}(s)
  + \sum_{i=1}^{2^n} U(t_i,\omega) \mathbb 1_{(t_{i-1},t_i]}(s)
  \qquad (0 \le s \le t, \omega \in \Omega).
\]
Then $U_n$ is $\mathcal B_{[0,t]} \otimes \mathcal F_t$-measurable by the adaptedness of $U$, and \eqref{Stopping-lemma-2-1} implies that $U_n$ converges pointwise to $U$ in $E_t$ (and therefore in $[0,t] \times \Omega$).
\end{proof}

\begin{lemma}\label{Stopping-lemma-3}
Let $\tau \colon \Omega \to (0,\infty]$ be a stopping time. Suppose that $f(t,\omega) \ge 0$ is defined for $\omega \in \Omega$ and $0 \le t < \tau(\omega)$, and that for each $\omega$,
\begin{equation}\label{Stopping-lemma-3-1}
  \text{$t \mapsto f(t,\omega)$ is continuous on $[0,\tau(\omega))$, and $f(0,\omega) =0$},
\end{equation}
and moreover that, for each $t \ge 0$ such that $\left\{ t < \tau \right\}$ is non-empty,
\begin{equation}\label{Stopping-lemma-3-2}
  \text{$\omega \mapsto f(t,\omega)$, defined for $\omega \in \{ t < \tau \}$, is $\mathcal F_t$-measurable.}
\end{equation}
For $R > 0$ define $\tau_R \colon \Omega \to (0,\infty]$ by
\[
  \tau_R(\omega) = \sup \left\{ t \in [0,\tau(\omega)) \colon \text{$f(s,\omega) < R$ for $0 \le s \le t$} \right\}.
\]
Then $\tau_R$ is a stopping time. Moreover, for each $\omega$,
\begin{equation}\label{StoppingTime1}
  \lim_{R \to \infty} \tau_R(\omega) = \tau(\omega),
\end{equation}
\begin{equation}\label{StoppingTime2}
  \text{$0 \le t \le \tau_R(\omega)$ and $t < \tau(\omega)$} \implies f(t,\omega) \le R
\end{equation}
and
\begin{equation}\label{StoppingTime3}
  \tau_R(\omega) < \tau(\omega) \implies f(\tau_R(\omega),\omega) = R.
\end{equation}
\end{lemma}

\begin{proof}

Let $t \ge 0$.
Note that, since $\tau_R \le \tau$,
\[
    \left\{ \tau_R > t \right\}
    =
    \left\{ \tau_R > t \right\}
    \cap
    \left\{ \tau > t \right\}
    .
\]
Using \eqref{Stopping-lemma-3-1},
and the compactness of the interval $[0,t]$,
we write
\[
    \left\{ \tau_R > t \right\}
    \cap
    \left\{ \tau > t \right\}
    =
    \bigcup_{ n=1 }^\infty
    \bigcap_{ s \in \mathbb Q \cap [0, t) }
    \left(
        \left \{
            f(s,\cdot) < R - \frac 1n
        \right \}
        \cap
        \left\{ \tau > t \right\}
    \right)
    ,
\]
which is evidently $\mathcal F_t$-measurable.
Indeed, for any $\omega$ belonging to the set on the right hand side
there exists $n \in \mathbb N$ such that for any $s \in \mathbb Q \cap [0, t)$
we have
\(
    f(s, \omega) < R - 1/n
    .
\)
Since $t < \tau(\omega)$, the continuity now implies
\(
    f(s, \omega) \leqslant R - 1/n
\)
for any $s \in [0, t]$, hence
\(
    f(s, \omega) < R
\)
for $s$ in some larger interval $[0, t + \delta]$.
This implies $\tau_R(\omega) > t$
and $\omega$ belongs to the set on the left hand side.
Conversely, if $\omega$ belongs to  the set on the left hand side, then
\(
    f(s, \omega) < R
\)
for any $s \in [0, t]$.
Hence there exists $n \in \mathbb N$ such that
\(
    f(s, \omega) < R - 1/n
\)
for any $s \in [0, t]$,
and in particular,
for any $s \in \mathbb Q \cap [0, t)$.
Therefore $\omega$ belongs to the set on the right hand side as well.

If $0 < a < \tau(\omega)$, then taking $R > \sup_{0 \le s \le a} f(s,\omega)$ gives $\tau_R(\omega) \ge a$, proving \eqref{StoppingTime1}. Finally, the properties \eqref{StoppingTime2} and \eqref{StoppingTime3} are immediate from the definition of $\tau_R$, and this concludes the proof of the lemma.
\end{proof}

\subsection{Stochastic integrals}

In this section, let $K$ and $H$ be separable Hilbert spaces, with orthonormal bases $\{e_k\}$ and $\{f_j\}$, respectively.

We denote by $\mathcal L(K,H)$ the space of bounded linear operators from $K$ into $H$, with the operator norm, and by $\mathcal L_2(K,H)$ the class of Hilbert-Schmidt operators
\[
  \mathcal L_2(K,H) = \left\{ T \in \mathcal L(K,H) \colon \trace(T^*T) < \infty \right\},
\]
which is a separable Hilbert space with the norm and inner product
\[
  \norm{T}_{\mathcal L_2(K,H)} = \trace(T^*T)^{1/2} = \left( \sum_k \norm{Te_k}_H^2 \right)^{1/2},
  \qquad
  \innerprod{S}{T}_{\mathcal L_2(K,H)} = \trace(T^*S).
\]
One can think of Hilbert-Schmidt operators as infinite-dimensional matrices. Indeed, defining
\[
  T_{jk} = \innerprod{Te_k}{f_j}_H,
\]
then $T \mapsto \{ T_{jk} \}$ is an isometry from $\mathcal L_2(K,H)$ onto $l^2(\N \times \N)$.

For later use we note the fact that if $S \in \mathcal L(H,H')$, where $H'$ is a Hilbert space, and $T \in \mathcal L_2(K,H)$ then the composition $ST$ belongs to $\mathcal L_2(K,H')$ and
\begin{equation}\label{HScomposition}
  \norm{ST}_{\mathcal L_2(K,H')} \le \norm{S}_{\mathcal L(H,H')}  \norm{T}_{\mathcal L_2(K,H)}.
\end{equation}

Consider now the cylindrical Wiener process
\[
  W(t) = \sum_{k=1}^\infty B_k(t) e_k,
\]
where the sum is formal. Given $T > 0$, the $H$-valued It\^o integral of an adapted process 
\[
  F \in L^2\left([0,T] \times \Omega , \mathcal L_2(K,H) \right)
\]
is a natural generalisation of the $n$-dimensional It\^o integral. It can be defined by
\[
  \int_0^T F(t) \, dW(t)
  =
  \lim_{n \to \infty} \sum_{j,k=1}^n \left( \int_0^T F_{jk}(t) \, dB_k(t) \right) f_j,
\]
where the integrals on the right hand side are ordinary It\^o integrals and $F_{jk} = \innerprod{Fe_k}{f_j}_H$ are the matrix entries of $F$. The sum converges in $L^2(\Omega;H)$ and the It\^o isometry holds:
\begin{equation}\label{Ito}
  \mathbb E\left( \norm{\int_0^T F(t) \, dW(t)}_H^2 \right)
  = \mathbb E\left( \int_0^T \norm{F(t)}_{\mathcal L_2(K,H)}^2  \, dt \right).
\end{equation}
Moreover, the $H$-valued random variable $I(t,\omega) = \int_0^t F(s,\omega) \, dW(s,\omega)$ is adapted and we can assume that it has continuous paths, since it has a modification with this property. Further, $I$ is a martingale, so by Doob's maximal inequality (see, e.g., \cite{Gawarecki_Mandrekar}),
\begin{equation}\label{Maximal}
  \mathbb E\left( \sup_{0 \le t \le T} \norm{\int_0^t F(s) \, dW(s)}_H^2 \right)
  \le 4\mathbb E\left( \norm{\int_0^T F(s) \, dW(s)}_H^2 \right).
\end{equation}

\subsection{Function spaces} Let $d \in \N$.
First, $H^s \left( \R^d \right)$ denotes the usual Sobolev space with norm
\begin{equation}
\label{Sob-norm}
  \norm{f}_{H^s(\R^d)} = \left( \int_{\R^d} \japanese{\xi}^{2s} \bigabs{\widehat f(\xi)}^2 \, d\xi \right)^{1/2},
\end{equation}
where $\japanese{\xi} = \left( 1+\Abs{\xi}^2 \right)^{1/2}$ and the Fourier transform is defined by
\[
    \widehat f(\xi)
    = \mathcal F f(\xi) = \int_{\R^d} e^{-ix\xi} f(x) \, dx
    \qquad
    \left( \xi \in \R^d \right).
\]
Then the inverse transform is given by $\mathcal F^{-1} g(x) = (2\pi)^{-d} \int_{\R^d} e^{ix\xi} g(\xi) \, d\xi$, the Plancherel identity reads $\norm{\mathcal F f}_{L^2(\R^d)} = (2\pi)^{d/2} \norm{f}_{L^2(\R^d)}$, and we have $\widehat{fg} = (2\pi)^{-d} \widehat f * \widehat g$ and
$\widehat{f * g} = \widehat f \, \widehat g$
for $f, g \in L^2 \left( \R^d \right)$.

We recall the Sobolev product law (see Theorem 2.2 in \cite{DAncona2012})
\begin{equation}\label{Sob-product-1}
  \norm{fg}_{H^{-s_1}(\R^d)} \le C \norm{f}_{H^{s_2}(\R^d)}\norm{g}_{H^{s_3}(\R^d)},
\end{equation}
which holds for all Schwartz functions $f$ and $g$ on $\R^d$ if and only if $s_1,s_2,s_3 \in \R$ satisfy
\begin{equation}\label{Sob-product-2}
  s_1+s_2+s_3 \ge \frac{d}{2}
  \quad \text{and} \quad
  \min_{i \neq j}(s_i+s_j) \ge 0,
  \quad \text{which are not both equalities}.
\end{equation}

For $d=1$ and $0 < b < 1$ we will make use of the norm equivalence
(see \cite[Lemma 3.15]{McLean2000})
\begin{equation}\label{HbEquiv}
  \norm{f}_{H^b(\R)} \sim \left( \norm{f}_{L^2(\R)}^2 + \norm{f}_{S^b(\R)}^2 \right)^{1/2}.
\end{equation}
Here $\norm{f}_{S^b(\Omega)}$ denotes the Slobodeckij seminorm on an open set $\Omega \subset \R$,
\[
  \norm{f}_{S^b(\Omega)}^2 = \int_\Omega \int_\Omega \frac{\Abs{f(t)-f(r)}^2}{\Abs{t-r}^{1+2b}} \, dr \, dt.
\]
On any finite time interval $I=(S,T)$ there is a similar norm equivalence (see \cite[Theorem 4.1]{Haroske2008})
\begin{equation}\label{HbEquivI}
  \norm{f}_{H^b(I)} \sim_I \left( \norm{f}_{L^2(I)}^2 + \norm{f}_{S^b(I)}^2 \right)^{1/2},
\end{equation}
but with the caveat that the constants depend on the interval.
The norm on the left hand side is the restriction norm, defined as the infimum
of $\norm{g}_{H^b(\R)}$ taken over $g \in H^b(\R)$ with $g=f$ on $I$.

Second, given a function $h \in C \left( \R^d,\R \right)$,
we denote by $X^{s,b}_{h(\xi)} \left( \R \times \R^d \right)$ the Bourgain space with norm
\begin{equation}
\label{Bourgain-norm}
  \norm{u}_{X^{s,b}_{h(\xi)}}
  =
  \left( \int_{\R^d} \int_{\R} \japanese{\xi}^{2s} \japanese{\tau+h(\xi)}^{2b}
  \Abs{\widehat u(\tau,\xi)}^2 \, d\tau \, d\xi \right) ^{1/2}
  ,
\end{equation}
where $\widehat u = \mathcal F_{(t,x)} u$ is the space-time Fourier transform.
The restriction to a time-slab $(S,T) \times \R^d$ is denoted $X^{s,b}_{h(\xi)}(S,T)$,
and is equipped with the norm
\begin{equation}
\label{XRestNorm}
  \norm{u}_{X^{s,b}_{h(\xi)}(S,T)} = \inf \left\{ \norm{v}_{X^{s,b}_{h(\xi)}} \colon \text{$u(t)=v(t)$ for $t \in (S,T)$} \right\}
  .
\end{equation}

The symbol $\tau+h(\xi)$ in \eqref{Bourgain-norm} is associated to the linear PDE
$-i\partial_t u + h(D_x) u = 0$ with the group
\[
  S_{h(\xi)}(t) = e^{-ith(D_x)},
\]
whose action on $f(x)$ is given by $\mathcal F_x \left\{ S_{h(\xi)}(t) f \right\}(\xi) = e^{-ith(\xi)} \widehat f(\xi)$. Here we recall that $D_x = -i\partial_x$, so that $\widehat{D_xf}(\xi) = \xi \widehat f(\xi)$.
Note that we can also write \eqref{Bourgain-norm} as
\begin{equation}
\label{XNorm_alternative}
	\norm
	{
		u
	}
	_{ X_{h(\xi)}^{ s, b } }^2
	=
	\int
	\langle \xi \rangle ^{2s}
	\norm
    {
        e^{ it h(\xi) }
        \mathcal F_x u(t, \xi)
    } _{ H_t^b(\R) }^2
	d\xi
	.
\end{equation}

We now mention some well-known properties of Bourgain norms; we refer to \cite{Ginibre} and \cite{TaoBook} for more details.
First, we note the obvious conjugation property
\begin{equation}
\label{XRestNorm_conjugate}
	\norm
	{
		\overline u
	}
	_{ X_{ h(\xi) }^{ s, b }(S, T) }
	=
	\norm
	{
		u
	}
	_{ X_{ -h(-\xi) }^{ s, b }(S, T) }
\end{equation}
valid also on the whole line $\mathbb R$, of course.
By $L^2$ duality and Plancherel's theorem it is clear that
\begin{equation}\label{Xduality}
  \norm{u}_{X^{s,b}_{h(\xi)}}
  =
  (2 \pi)^{d + 1}
  \sup_{\norm{v}_{X^{-s,-b}_{h(\xi)}} = 1} \Abs{ \int_{\R^d} \int_{\R} u(t,x) \overline{v(t,x)} \, dt \,dx}
\end{equation}
and
\begin{equation}\label{Xduality2}
  (2 \pi)^{d + 1}
  \Abs{\int_{\R^d} \int_{\R} u(t,x) \overline{v(t,x)} \, dt \,dx}
  \le
  \norm{u}_{X^{s,b}_{h(\xi)}}\norm{v}_{X^{-s,-b}_{h(\xi)}}.
\end{equation}
Now let $\theta(t)$ be any smooth, compactly supported function.
From \eqref{XNorm_alternative} it is clear that
\begin{equation}\label{Bourgain-1}
    \norm{\theta(t) S_{h(\xi)}(t)f}_{X^{s,b}_{h(\xi)}} = \norm{\theta}_{H^b(\R)} \norm{f}_{H^s(\R^d)}
    \quad \text{for $b \in \R$}.
\end{equation}
Moreover,
\begin{equation}
\label{Bourgain-2}
    \norm{\theta(t) \int_0^t S_{h(\xi)}(t-t')F(t') \, dt'}_{X^{s,b}_{h(\xi)}}
    \leqslant
    C_b
    \left(
        \norm{\theta}_{H^b(\R)} + \norm{t \theta(t)}_{H_t^b(\R)}
    \right)
    \norm{F}_{X^{s,b-1}_{h(\xi)}}
    \quad
    \text{for $b > \frac12$}
    ,
\end{equation}
which by \eqref{XNorm_alternative} reduces to the inequality
\begin{equation}
\label{Bourgain-2-reduced}
    \norm{\theta(t) \int_0^t f(t') \, dt'}_{H^{b}_t(\R)}
    \leqslant
    C_b
    \left(
        \norm{\theta}_{H^b(\R)} + \norm{t \theta(t)}_{H_t^b(\R)}
    \right)
    \norm{f}_{H^{b-1}(\R)}
    \quad
    \text{for $b > \frac12$}
    .
\end{equation}
The latter can be proved by using Fourier inversion on $f$ as follows, cf. \cite{Ginibre}.
Assuming $f \in \mathcal S(\mathbb R)$
and using Plancherel's theorem we can rewrite the integral of $f$ as
\[
    \int_0^t f(t') \, dt'
    =
    \frac 1{2\pi}
    \int_{\mathbb R}
    \widehat f(\tau)
    \overline{ \mathcal F \left( \mathbb 1_{(0, t)} \right) (\tau) }
    \,
    d \tau
    =
    \frac 1{2\pi i}
    \int_{\mathbb R}
    \widehat f(\tau)
    \frac 1{\tau}
    \left(
        e^{it \tau} - 1
    \right)
    \,
    d \tau
    .
\]
Note that the right hand side here is well defined for any
$f \in H^{b - 1}(\mathbb R)$ with $b > 1/2$,
as we shall see below.
It serves as a definition for the left hand side.
We need to calculate the $L^2$-norm of
\[
    J(\lambda)
    =
    \langle \lambda \rangle ^b
    \mathcal F_t
    \left(
        \theta(t) \int_0^t f(t') \, dt'
    \right)
    (\lambda)
    =
    \frac{1}{2\pi i}
    \int_{\R}
    \widehat f(\tau)
    \frac { \langle \lambda \rangle ^b }{ \tau }
    \left(
        \widehat \theta( \lambda - \tau )
        -
        \widehat \theta( \lambda )
    \right)
    \, d\tau
    .
\]
At first we split this integral into integrals over
$|\tau| < 1$ and $|\tau| \geqslant 1$.
On the second domain we can just bound
\(
    |\tau| \geqslant \langle \tau \rangle / 2
\)
and then return to the integration over the whole line $\mathbb R$.
On the first domain we use the fundamental theorem of calculus for the difference
in brackets
\[
    \frac 1{ \tau }
    \left(
        \widehat \theta( \lambda - \tau )
        -
        \widehat \theta( \lambda )
    \right)
    =
    i
    \int_0^1
    \mathcal F_t( t \theta (t) )( \lambda - \mu \tau )
    \,
    d \mu
    .
\]
Hence
\[
    \Abs{J(\lambda)}
    \leqslant
    \frac{1}{2\pi}
    \int_0^1
    \int_{-1}^1
    \Abs{\widehat f(\tau)}
    \langle \lambda \rangle ^b
    \Abs{
        \mathcal F_t( t \theta (t) )( \lambda - \mu \tau )
    }
    \,
    d\tau
    d\mu
    +
    \frac{1}{\pi}
    \int_{\R}
    \Abs{\widehat f(\tau)}
    \frac
    {\langle \lambda \rangle ^b}
    { \langle \tau \rangle }
    \left(
        \Abs{\widehat \theta( \lambda - \tau )}
        +
        \Abs{\widehat \theta( \lambda )}
    \right)
    \, d\tau
    .
\]
In the first integral
we can bound
\(
    \japanese{\lambda}^b
    \lesssim
    \japanese{\lambda - \mu \tau}^b + 1
    \lesssim
    \japanese{\lambda - \mu \tau}^b
    ,
\)
so up to a constant its $L^2$-norm gives rise to the term
\[
    \norm{
        t \theta (t)
    }_{H_t^b}
    \int_{-1}^1
    \Abs{\widehat f(\tau)}
    \,
    d\tau
    \lesssim
    \norm{
        t \theta (t)
    }_{H_t^b}
    \int_{-1}^1
    \Abs{\widehat f(\tau)}
    \japanese{\tau}^{b - 1}
    \,
    d\tau
    \lesssim
    \norm{
        t \theta (t)
    }_{H_t^b}
    \norm{
        f
    }_{H^{b-1}}
\]
by Minkowski's integral inequality.
Now noticing
\(
    \japanese{\lambda}^b
    \lesssim
    \japanese{\lambda - \tau}^b + \japanese{\tau}^b
\)
the second integral can be split, up to a constant, into the following three integrals
\[
    \int_{\R}
    \Abs{\widehat f(\tau)}
    \frac
    {\langle \lambda \rangle ^b}
    { \langle \tau \rangle }
    \Abs{\widehat \theta( \lambda )}
    \, d\tau
    +
    \int_{\R}
    \Abs{\widehat f(\tau)}
    \frac
    {\langle \lambda - \tau \rangle ^b}
    { \langle \tau \rangle }
    \Abs{\widehat \theta( \lambda - \tau )}
    \, d\tau
    +
    \int_{\R}
    \Abs{\widehat f(\tau)}
    \langle \tau \rangle ^{b - 1}
    \Abs{\widehat \theta( \lambda - \tau )}
    \, d\tau
    =
    J_1 + J_2 + J_3
    .
\]
Here $J_3$ is a convolution,
hence
\[
    \norm{J_3}_{L^2}
    \leqslant
    \norm{ \widehat \theta }_{L^1}
    \norm{ \widehat f(\tau) \langle \tau \rangle ^{b - 1} }_{L^2}
    \leqslant
    \norm{ \langle \tau \rangle ^{- b} }_{L_{\tau}^2}
    \norm{ \theta }_{H^b}
    \norm f _{H^{b - 1}}
    ,
\]
where the $L^1$-norm was estimated by H\"older.
Similarly,
\[
    \norm{J_1}_{L^2}
    +
    \norm{J_2}_{L^2}
    \leqslant
    2
    \norm{ \widehat \theta(\lambda) \japanese{\lambda}^b }_{L^2}
    \norm{ \widehat f(\tau) \langle \tau \rangle ^{-1} }_{L^1}
    \leqslant
    2
    \norm{ \langle \tau \rangle ^{- b} }_{L_{\tau}^2}
    \norm{ \theta }_{H^b}
    \norm f _{H^{b - 1}}
\]
that finishes the proof of \eqref{Bourgain-2-reduced}.

From \eqref{Bourgain-1} and \eqref{Bourgain-2},
one immediately obtains the corresponding restriction norm inequalities
on any time interval $(0,T)$,
by choosing a bump function $\theta$ such that $\theta(t) = 1$ for $\abs{t} \le 1$.
If $T \in (0,1]$, we apply \eqref{Bourgain-1} and \eqref{Bourgain-2} with $\theta(t)$,
while if $T > 1$ we apply them with $\theta(t/T)$ instead of $\theta(t)$,
and use the fact that
\(
    \norm{\theta(t/T)}_{H^b_t}
    \leqslant
    \sqrt{T}
    \norm{\theta}_{H^{ \max(b, 0) }}
\)
and
\(
    \norm{t \theta(t/T)}_{H^b_t}
    \leqslant
    T^{3/2}
    \norm{t \theta(t)}_{H^{ \max(b, 0) }_t}
\)
for $T > 1$.
This gives
\begin{equation}\label{Bourgain-1b}
    \norm{S_{h(\xi)}(t)f}_{X^{s,b}_{h(\xi)}(0,T)} \le C_b \left(1+\sqrt{T}\right) \norm{f}_{H^s(\R^d)}
    \quad \text{for $b \in \R$ and $T>0$}
\end{equation}
and
\begin{equation}\label{Bourgain-2b}
    \norm{\int_0^t S_{h(\xi)}(t-t')F(t') \, dt'}_{X^{s,b}_{h(\xi)}(0,T)} \le
    C_b \left(1+ T^{3/2} \right) \norm{F}_{X^{s,b-1}_{h(\xi)}(0,T)}
    \quad
    \text{for $b > \frac12$ and $T>0$},
\end{equation}
where $C_b$ depends on $b$ but not on $T$.

Further, one has (see Lemma 2.11 in \cite{TaoBook})
\begin{equation}\label{Bourgain-3}
    \norm{u}_{X^{s,b}_{h(\xi)}(0,T)} \le C_{b,b'} T^{b'-b} \norm{u}_{X^{s,b'}_{h(\xi)}(0,T)}
    \quad
    \text{for $-\frac12 < b < b' < \frac12$ and $0 < T \le 1$}
\end{equation}
and
\begin{equation}
\label{Bourgain-4}
    b > \frac12 \implies X^{s,b}_{h(\xi)}(0,T) \hookrightarrow C\left( [0,T], H^s \right)
    \quad \text{with}
    \quad
    \sup_{0 \le t \le T} \norm{u(t)}_{H^s(\R^d)} \le C_b \norm{u}_{X^{s,b}_{h(\xi)}(0,T)},
\end{equation}
where the last inequality follows by applying the Sobolev embedding
$H^b(\R) \hookrightarrow L^\infty(\R)$, for $b > 1/2$,
to the function
\(
    t \mapsto e^{ith(\xi)} \widehat u(t,\xi)
\)
staying in
\(
    \norm{u(t)}_{H^s(\R^d)}^2
    =
    \int \japanese{\xi}^{2s} \Abs{e^{ith(\xi)} \widehat u(t,\xi)}^2 \, d\xi
    ,
\)
and recalling first \eqref{XNorm_alternative}, then \eqref{XRestNorm}.

We will also need the trivial fact that
\begin{equation}
\label{Bourgain-5}
    b \ge 0 \implies X^{s,b}_{h(\xi)}(0,T) \hookrightarrow L^2\left( [0,T], H^s \right)
    \quad \text{with}
    \quad
    \norm{u}_{L^2([0,T],H^s)} \le \norm{u}_{X^{s,b}_{h(\xi)}(0,T)}.
\end{equation}
In particular, this implies that the space
$\mathbb L^2 \left( \Omega, X^{s,b}_{h(\xi)} \right)$
is well-defined when $b \ge 0$, as in \eqref{L2-prog}.

Finally, we recall that for $-1/2 < b < 1/2$
the restriction norm on $H^b(S, T)$ is equivalent to the
$H^b(\R)$-norm 
of the trivial extension.
More precisely, for all $\phi \in H^{b}(\R)$ we have
\begin{equation}
\label{H_characteristic_norm}
    \norm{\phi}_{H^b(S, T)}
    \leqslant
    \norm{\mathbb 1_{(S, T)} \phi}_{H^b(\R)}
    \leqslant
    C_b \norm{\phi}_{H^b(S, T)}
    \quad
    \text{for $-\frac12 < b < \frac12$},
\end{equation}
where we emphasise that $C_b$ is independent of $(S,T)$.
Similarly, for all $u \in X_{h(\xi)}^{s,b}(\R \times \R^d)$,
\begin{equation}
\label{X_characteristic_norm}
    \norm{u}_{X^{s,b}_{h(\xi)}(S, T)}
    \leqslant
    \norm{\mathbb 1_{(S, T)} u}_{X^{s,b}_{h(\xi)}}
    \leqslant
    C_b \norm{u}_{X^{s,b}_{h(\xi)}(S, T)}
    \quad
    \text{for $-\frac12 < b < \frac12$.}
\end{equation}
Here $\mathbb 1_{(S, T)}(t)$ is the characteristic function
of the interval $(S, T)$. Note that the left inequalities in\eqref{H_characteristic_norm} and \eqref{X_characteristic_norm} hold trivially by the definition the restriction norm; for a proof of the right inequalities, by an argument relying on the Slobodeckij seminorm, see Lemma 4 in \cite{Candy2012} (alternatively one can use the Fourier transform, applying the ideas used to prove Lemma 3.2 in \cite{Tao2004}).

Combining \eqref{H_characteristic_norm} and \eqref{X_characteristic_norm} with \eqref{XNorm_alternative}, we get the restriction norm equivalence
\begin{equation}
\label{XRestNorm_alternative}
	\norm
	{
		u
	}
	_{ X_{h(\xi)}^{ s, b }(S,T) }
	\sim
	\left( \int
	\langle \xi \rangle ^{2s}
	\norm
    {
        e^{ it h(\xi) }
        \mathcal F_x u(t, \xi)
    } _{ H_t^b (S,T)}^2
	d\xi \right)^{1/2}
	\quad
    \text{for $-\frac12 < b < \frac12$},
\end{equation}
with constants depending on $b$ but not on $(S,T)$.
With an additional effort one can get a stronger result.

\begin{lemma}[Bourgain isometry]
\label{Bourgain_isometry_lemma}
    For any $s, b \in \mathbb R$, interval $(S, T)$
    and
    \(
        h \in C \left( \mathbb R^d, \mathbb R \right)
    \)
    we have the following
    \begin{equation}
    \label{strong_XRestNorm_alternative}
    	\norm u
    	_{ X_{h(\xi)}^{ s, b }(S,T) }
    	=
        \left( \int
        \langle \xi \rangle ^{2s}
        \norm
        {
            e^{ it h(\xi) }
            \mathcal F_x u(t, \xi)
        } _{ H_t^b (S,T)}^2
        d\xi \right)^{1/2}
        .
    \end{equation}
\end{lemma}

We find the proof of this lemma instructive and not completely straightforward,
so we put it in a separate section \ref{Bourgain_isometry}.
Moreover, we could not find it presented anywhere else.
As a matter of fact, even the weaker result \eqref{XRestNorm_alternative}
would serve all our needs below.
So Lemma \ref{Bourgain_isometry_lemma} together with its proof
given in Section \ref{Bourgain_isometry} can be regarded as a complementary material.

Another immediate consequence of \eqref{X_characteristic_norm} is that functions on adjacent time intervals can be glued together.

\begin{lemma}\label{Gluing-lemma}
Let $-1/2 < b < 1/2$.
Then there exists a constant $C_b$ such that if $u \in X^{s,b}_{h(\xi)}(t_0,t_1)$ and $v \in X^{s,b}_{h(\xi)}(t_1,t_2)$, where $t_0 < t_1 < t_2$, then the glued function
\[
  [u,v](t) = \begin{cases}
  u(t) &t_0 < t < t_1
  \\
  v(t) &t_1 < t < t_2
  \end{cases}
\]
belongs to $X^{s,b}_{h(\xi)}(t_0,t_2)$ and
\[
  \norm{[u,v]}_{X^{s,b}_{h(\xi)}(t_0,t_2)}
  \le C_b \left(
  \norm{u}_{X^{s,b}_{h(\xi)}(t_0,t_1)}
  +
  \norm{v}_{X^{s,b}_{h(\xi)}(t_1,t_2)}
  \right).
\]
\end{lemma}
\begin{proof} This follows from \eqref{X_characteristic_norm} and the triangle inequality since $\mathbb 1_{(t_0,t_2)} = \mathbb 1_{(t_0,t_1)} + \mathbb 1_{(t_1,t_2)}$ a.e.
\end{proof}

\subsection{Cutoffs and a modified Bourgain norm}

As usual with a multiplicative noise,
we have to truncate the nonlinearity in order to prove existence by iteration.
In the corresponding cutoffs we will use, for technical reasons,
not the restriction norm \eqref{XRestNorm}
but an equivalent norm, defined by

\begin{equation}
\label{Bourgain_Slobodeckij_norm}
	\norm
	{
		u
	}
	_{ \widetilde X_{h(\xi)}^{ s, b }(S, T) }^2
	=
	\int_{\R^d}
    \left(
        \frac 1{ (T - S)^{2b} }
        \int_S^T \Abs{ U(t, \xi) }^2 \, dt
        +
        \int_S^T \int_S^T
        \frac
        { | U(t, \xi) - U(r, \xi) |^2 }
        { | t - r |^{1 + 2b} }
        \, dr \, dt
    \right)
	d\xi
    ,
\end{equation}
where
\(
    U(t, \xi)
    =
    \langle \xi \rangle ^{s} e^{ it h(\xi) }
    \mathcal F_x u(t, \xi)
	.
\)
It is thoroughly studied below in Sections \ref{Mod-B-norm} and \ref{CutoffSection}.
Here it is crucial that the norm equivalence is uniform
with respect to the time interval to which we restrict.
To be precise, for any $T_0 > 0$ and $b \in (0,1/2)$ there exists a constant $C_{T_0,b}$
such that (see Lemma \ref{NormLemma})
\begin{equation}
    \label{ModEquivalence}
    C_{T_0,b}^{-1} \norm{u}_{X^{s,b}_{h(\xi)}(S,T)}
    \le
    \norm{u}_{\widetilde X^{s,b}_{h(\xi)}(S,T)}
    \le C_{T_0,b}
    \norm{u}_{X^{s,b}_{h(\xi)}(S,T)}
\end{equation}
for all $u \in X_{h(\xi)}^{s,b}(S,T)$ with $0 < S < T \le T_0$.

The idea of exploiting the Slobodeckij seminorm,
the double integral over $(S, T)$ in \eqref{Bourgain_Slobodeckij_norm},
comes from \cite{Bouard_Debussche2007}.
However, we point out that the factor $1/(T - S)^{2b}$
in front of the $L^2$-norm
turns out to be important to claim the uniform equivalence \eqref{ModEquivalence}; see Remark \ref{Cutoffremark} in Section \ref{CutoffSection} for a further discussion of this.
Moreover, the necessity of this factor becomes clear
when one calculates these norms on a concrete element,
say
\(
    u(t) = S_{ h(\xi) }(t) f,
\)
where $f \in H^s \left(\R^d\right)$.
In fact, by \eqref{strong_XRestNorm_alternative} it follows immediately that
\[
    \norm
    { S_{ h(\xi) }(t) f }
    _
    { X_{ h(\xi) }^{s,b}(0,T) }
    =
    \norm{1}_{H^b(0,T)}
    \norm{f}_{H^s}
\]
for all $T > 0$ and $b \in \R$, and it is easy to see that
\begin{equation}\label{Basic-Sob-norm}
    \norm{1}_{H^b(0,T)}
    \sim
    T^{1/2 - b}
    \quad
    \text{for $0 \leqslant b < \frac12$ and $0 < T \leqslant 1$.}
\end{equation}
Indeed, recalling \eqref{H_characteristic_norm}, we can calculate the equivalent norm
\[
    \norm{ \mathbb 1_{(0, T)} }_{H^b}^2
    =
    \int_{\R}
    \Abs{
        \int_0^T e^{-it\tau} dt
    }^2
    \japanese{\tau}^{2b} d\tau
    =
    T
    \int_{\R}
    \Abs{
        \frac 2{\tau}
        \sin \frac {\tau}2
    }^2
    \japanese{ \frac {\tau}T }^{2b} d\tau
    ,
\]
which is comparable to $T^2 + T^{1-2b} \sim T^{1-2b}$ (for $0 < T \le 1$ and $0 \le b < 1/2$) as one can see by splitting the last integration into $\abs{\tau} \le T$ and $\abs{\tau} > T$. 
On the other hand, it is easily seen from \eqref{Bourgain_Slobodeckij_norm}, with $U(t,\xi) =  \langle \xi \rangle ^{s} \widehat f(\xi)$, that \[
  \norm{S_{h(\xi)}(t) f}_{\widetilde X_{ h(\xi) }^{s,b}(0,T) } = T^{1/2-b} \norm{f}_{H^s}.
\]

For the modified restriction norm \eqref{Bourgain_Slobodeckij_norm} we have the following key estimates,
proved in Section \ref{CutoffSection}.
\begin{proposition}
\label{MainCutoffLemma}
    Let $T_0 > 0$ and $b \in (0,1/2)$. Let $\theta \colon \R \to \R$ be a smooth, compactly supported function and set $\theta_R(x) = \theta(x/R)$ for $R > 0$. Let $n \in \N$, and for $1 \le i \le n$ let $s_i \in \R$,  $h_i \in C(\R^d,\R)$ and $u_i,v_i \in X^{s_i,b}_{h_i(\xi)}(0,T_0)$. Then for $T \in (0,T_0]$, $R > 0$ and $1 \le j \le n$ we have the estimates
    \[
        \norm{
            \theta_R\left( \sum_{i=1}^n \norm{u_i}_{\widetilde X^{s_i,b}_{h_i(\xi)}(0,t)}^2 \right)
            u_j(t)
        }_{X^{s_j,b}_{h_j(\xi)}(0,T)}
        \leqslant
        C \sqrt{R}
        ,
    \]
    \[
        \norm{
            \theta_R\left( \sum_{i=1}^n \norm{u_i}_{\widetilde X^{s_i,b}_{h_i(\xi)}(0,t)}^2 \right)
            u_j(t)
            -
            \theta_R\left( \sum_{i=1}^n \norm{v_i}_{\widetilde X^{s_i,b}_{h_i(\xi)}(0,t)}^2 \right)
            v_j(t)
        }_{X^{s_j,b}_{h_j(\xi)}(0,T)}
        \leqslant
        C \sum_{i=1}^n \norm{u_i-v_i}_{X^{s_i,b}_{h_i(\xi)}(0,T)}
        ,
    \]
    where $C$ depends only on $b$, $T_0$ and $\theta$.
\end{proposition}

The modified norm \eqref{Bourgain_Slobodeckij_norm} can be formulated also for functions $\phi(t)$ depending only on the time variable $t$, and this gives cutoff estimates for functions in $H^b(0, T)$, $0 < b < 1/2$. See Section \ref{only-time-cutoff}.

\section{Main results}
\label{Main_result}
\setcounter{equation}{0}

We consider the mild form of \eqref{Ito-DKG-split}, which reads
\begin{multline}
\label{mild_psi}
	\psi_{\pm}(t)
	=
	S_{\pm\xi}(t) f_\pm
	-
	iM \int_0^t S_{\pm\xi}(t-\sigma)  \psi_{\mp}(\sigma) \, d\sigma
	+
	i \int_0^t  S_{\pm\xi}(t-\sigma) ( \phi\psi_{\mp})(\sigma) \, d\sigma
	\\
	+
	i \int_0^t S_{\pm\xi}(t-\sigma) \psi_\mp(\sigma) \mathfrak K_1 \, d W(\sigma)
	-
	M_{\mathfrak K_1} \int_0^t S_{\pm\xi}(t-\sigma) \psi_\pm(\sigma) \, d\sigma
\end{multline}
and
\begin{multline}
\label{mild_phi}
	\phi_+(t)
	=
	S_{+\japanese{\xi}}(t) g_+
	+
	i \int_0^t S_{+\japanese{\xi}}(t-\sigma) \japanese{D_x}^{-1}
    \re \left( \overline{\psi_+} \psi_- \right)(\sigma) \, d\sigma
	\\
	+
	\frac{i}{2} \int_0^t S_{+\japanese{\xi}}(t-\sigma) \japanese{D_x}^{-1}
    \phi(\sigma) \mathfrak K_2 \, dW(\sigma)
,
\end{multline}
where
\(
    \phi
    =
    \phi_+ + \overline{\phi_+}
    =
    2 \re \phi_+
    .
\)
We will look for solutions
\begin{equation}
\label{IterationSpace}
    \psi_\pm \in X_{\pm\xi}^{s,b}(0,T)
    , \qquad
    \phi_+ \in X_{+\japanese{\xi}}^{r,b}(0,T)
    ,
\end{equation}
where $b < 1/2$ is taken sufficiently close to $1/2$, depending on $s$ and $r$.
Note that by the conjugation property \eqref{XRestNorm_conjugate}
we have
\begin{equation}
\label{IterationSpace_conjugate}
    \overline{\phi_+} \in X_{ -\japanese{\xi} }^{r,b}(0,T)
    .
\end{equation}
It will be convenient to define
\begin{equation}\label{spaceHssr}
    \mathbf H^{(s,s,r)} = H^s(\R,\C) \times H^s(\R,\C) \times H^r(\R,\C)
\end{equation}
and
\begin{equation}\label{spaceXssr}
    \mathbf X^{(s,s,r),b}(0,T) = X_{+\xi}^{s,b}(0,T) \times X_{-\xi}^{s,b}(0,T) \times X_{+\japanese{\xi}}^{r,b}(0,T),
\end{equation}
with the product norms.

We now state our main results.

\begin{theorem}[Local existence]
\label{local_th}
Assume that $s,r \in \R$ satisfy
\[
  s > - \frac14, \qquad \Abs{s} \le r \le s+1, \qquad 0 < r < 1+2s.
\]
Assume further that the kernels $\mathfrak k_j$, defining the convolution operators $\mathfrak K_j$, have the regularity
\[
  \mathfrak k_1 \in H^{\Abs{s}}(\R,\R), \qquad \mathfrak k_2 \in H^{\max(0,r-1)}(\R,\R).
\]
Then for any $b < 1/2$ sufficiently close to $1/2$, the following holds. Assume that
\[
  (f_+,f_-,g_+) \in L^2\bigl( \Omega, \mathbf H^{(s,s,r)} \bigr)
  \quad \text{is $\mathcal F_0$-measurable}.
\]
Then there exists a stopping time $\tau \colon \Omega \to (0,\infty]$ and a random process
\[
  (\psi_+,\psi_-,\phi_+)(t) \in \mathbf H^{(s,s,r)} \quad \text{for $0 \le t < \tau$}
\]
such that for $0 < t < \tau$, \eqref{mild_psi} and \eqref{mild_phi} hold,
\[
  (\psi_+,\psi_-,\phi_+)(t) \colon \{ t < \tau \} \to \mathbf H^{(s,s,r)} \quad \text{is $\mathcal F_t$-measurable}
\]
and
\[
  (\psi_+,\psi_-,\phi_+) \in C\left( [0,t],\mathbf H^{(s,s,r)} \right) \cap \mathbf X^{(s,s,r),b}(0,t).
\]
Moreover, the solution is maximal in the sense that
\[
  \tau < \infty \implies \limsup_{t \nearrow \tau} \norm{(\psi_+,\psi_-,\phi_+)}_{\mathbf X^{(s,s,r),b}(0,t)} = \infty,
\]
and it is unique in the sense that if $(\Psi_+,\Psi_-,\Phi_+)$ is a solution with the same initial data, and satisfying the same assumptions but with a stopping time $\tau'$,
then almost surely
\[
  (\psi_+,\psi_-,\phi_+)(t) = (\Psi_+,\Psi_-,\Phi_+)(t)
  \quad \text{for $0 \le t < \min(\tau,\tau')$}.
\]
Further, if $s \ge 0$, then the charge is almost surely conserved:
\[
  \int_{\R} \left( \Abs{\psi_+(t,x)}^2 + \Abs{\psi_-(t,x)}^2 \right) \, dx
  = \int_{\R} \Abs{\psi_0(x)}^2 \, dx
  \quad \text{for $0 \le t < \tau$},
\]
where $\psi_0 = (f_+,f_-)$.
\end{theorem}

This theorem is a consequence of the abstract well-posedness theory presented in Section \ref{Abstract-WP}.
The existence follows from Theorem \ref{Existence-Theorem} and the uniqueness from Theorem \ref{Uniqueness-theorem}. The necessary assumptions stated in Section \ref{Abstract-WP} are verified here on account of the bounds stated in Lemmas \ref{DKG-HS-lemma}, \ref{DKG-bilinear-lemma}, \ref{DKG-linear-lemma} and \ref{DKG-high-reg-lemma} below; see Section \ref{ConcreteCase} for the details. The charge conservation is proved in Section \ref{Charge_conservation}.

Using the charge conservation, we will then deduce the following global result.

\begin{theorem}[Global existence]
\label{global_th}
Let $s=0$ and $1/4 < r < 1/2$. Let $\max(r,1-2r) < b < 1/2$. Given initial data as in Theorem \ref{local_th},
we impose the additional condition $f_+,f_- \in L^p \left( \Omega,L^2 \right)$,
where
\[ 
  p \ge \max\left( 4,  \frac{2b + 2r - 1}{b + 2r - 1} \right)
\]
Then the solution in Theorem \ref{local_th} extends globally in time. That is, $\tau=\infty$.
\end{theorem}

The proof is given in Section \ref{Global_existence}.

Although the local result will, as mentioned, be deduced from the abstract framework expounded in some detail in Section \ref{Abstract-WP}, we find it worthwhile to present here a broad outline of the key ideas behind the proof.

Existence for a short time interval $(0,T)$ will be proved by iteration in
\begin{equation}\label{iteration-space}
  \mathbb L^2\left(\Omega,\mathbf X^{(s,s,r),b}(0,T) \right)
  \cap
  \mathbb L^2\left(\Omega,C([0,T],\mathbf H^{(s,s,r)}) \right),
\end{equation}
and the first thing to notice is that we cannot expect the stochastic integrals to be in this space unless $b$ is strictly less than $1/2$, the reason being that the paths of any one-dimensional Brownian motion belong to $H^b(0, T)$ if and only if $b < 1/2$, as shown in \cite{Benyi2010}.

For $b < 1/2$ the stochastic integrals can indeed be controlled in \eqref{iteration-space}, provided that we know that the linear operators
\begin{equation}\label{M1-M2}
  f \mapsto M_1(f) = f \mathfrak K_1 \quad \text{and} \quad g \mapsto M_2(g) = \japanese{D_x}^{-1} g \mathfrak K_2 
\end{equation}
map $H^s(\R)$ and $H^r(\R)$, respectively, into Hilbert-Schmidt operators from $L^2(\R,\R)$ into $H^s(\R)$ and $H^r(\R)$, respectively. Thus, we need the following.

\begin{lemma}\label{DKG-HS-lemma}
Let $s,r$ be as in Theorem \ref{local_th}. Assume that $\mathfrak k_1 \in H^{\Abs{s}}(\R,\R)$ and $\mathfrak k_2 \in H^{\max(0,r-1)}(\R,\R)$. Then there exists a constant $C$ such that the linear operators $M_1$ and $M_2$ defined by \eqref{M1-M2} satisfy
\[
  \norm{M_1(f)}_{\mathcal L_2(L^2,H^s)} \le C \norm{f}_{H^s}
  \quad \text{and} \quad
  \norm{M_2(g)}_{\mathcal L_2(L^2,H^r)} \le C \norm{g}_{H^r}
\]
for all $f \in H^s(\R)$ and $g \in H^r(\R)$.
\end{lemma}

The proof of this lemma is given in Section \ref{Stochastic_integral}, and in Section \ref{Abstract-WP} we show how it is applied to control the stochastic integrals.

Now let us turn our attention to the deterministic terms in \eqref{mild_psi} and \eqref{mild_phi}. Here there is a difference from the purely deterministic case, where one works with $b > 1/2$, see e.g.~\cite{Pecher2006}. Since now we are forced to take $b < 1/2$, the required bilinear estimates are a bit tighter. We will prove the following bilinear bounds, extending those obtained in \cite{Pecher2006, Selberg_Tesfahun} to the case where $b$ is less than, but close to, $1/2$.

\begin{lemma}\label{DKG-bilinear-lemma}
Assume that $s,r \in \R$ satisfy
\[
  s > - \frac14, \qquad \Abs{s} \le r \le s+1, \qquad 0 < r < 1+2s.
\]
Then for any $b < 1/2$ sufficiently close to $1/2$, there exists a constant $C$ such that
\begin{align}
  \label{M+bound}
  \norm{\phi \psi}_{X^{s,-b}_{+\xi}}
  &\le C \norm{\phi}_{X^{r,b}_{\pm\japanese{\xi}}}
  \norm{\psi}_{X^{s,b}_{-\xi}},
  \\
  \label{M-bound}
  \norm{\phi \psi}_{X^{s,-b}_{-\xi}}
  &\le C \norm{\phi}_{X^{r,b}_{\pm\japanese{\xi}}}
  \norm{\psi}_{X^{s,b}_{+\xi}},
  \\
  \label{Nbound}
  \norm{\overline{\psi} \psi'}_{X^{r-1,-b}_{\pm\japanese{\xi}}}
  &\le C \norm{\psi}_{X^{s,b}_{+\xi}}
  \norm{\psi'}_{X^{s,b}_{-\xi}},
\end{align}
for all Schwartz functions $\psi$, $\psi'$ and $\phi$ on $\R_t \times \R_x$. In particular, in the case $s=0$ and $1/4 < r < 1/2$, relevant for Theorem \ref{global_th}, the above estimates hold for all $b > 1/4$.
\end{lemma}

This lemma is proved in Section \ref{Bilinear_bounds}. The method of proof does not differ significantly from that used in \cite{Pecher2006, Selberg_Tesfahun} for the case $b > 1/2$. Also, we remark that studying bilinear space-time estimates in Bourgain norms with $b < 1/2$ is nothing new. For example, general product estimates for wave-type spaces were studied in \cite{DAncona2010, DAncona2012}.

With $b$ as in the last lemma, and choosing $0 < \varepsilon < 1/2-b$, set
\[
  B = -b+1-\varepsilon.
\]
Then $B > 1/2$,
so we can apply \eqref{Bourgain-2b} and \eqref{Bourgain-3}
to control the deterministic integrals in $\mathbf X^{(s,s,r), B}(0,T)$.
And then, crucially, by \eqref{Bourgain-4} they are also controlled in
$C \left( [0,T];\mathbf H^{(s,s,r)} \right)$
(and of course also in $\mathbf X^{(s,s,r),b}(0,T)$, since $b < B$).
For example,
\[
\begin{alignedat}{2}
  \norm{\int_0^t  S_{+\xi}(t-\sigma) ( \phi_+\psi_- )(\sigma) \, d\sigma}_{X^{s,B}_{+\xi}(0,T)}
  &\le
  C \norm{\phi_+\psi_-}_{X^{s,B-1}_{+\xi}(0,T)}&
  \quad &\text{by \eqref{Bourgain-2b}}
  \\
  &\le
  C T^{\varepsilon} \norm{\phi_+\psi_-}_{X^{s,-b}_{+\xi}(0,T)}&
  &\text{by \eqref{Bourgain-3}}
  \\
  &\le
  C T^{\varepsilon} \norm{\phi_+}_{X^{r,b}_{+\japanese{\xi}}(0,T)}
  \norm{\psi_-}_{X^{s,b}_{-\xi}(0,T)}&
  &\text{by Lemma \ref{DKG-bilinear-lemma}},
\end{alignedat}
\]
and similarly for the other bilinear terms,
and also for the linear ones, for which we use the following.

\begin{lemma}\label{DKG-linear-lemma}
Let $s \in \R$ and $b \ge 0$. Then
\[
  \norm{u}_{X^{s,-b}_{h(\xi)}}
  \le
  \sqrt{2 \pi}
  \norm{u}_{L^2(\R_t,H^s)}
  \le
  \norm{u}_{X^{s,b}_{g(\xi)}}
\]
for all Schwartz functions $u$ on $\R_t \times \R_x$ and any choice of $g,h \in C(\R,\R)$.
\end{lemma}
\begin{proof}
This is obvious from the definitions \eqref{Bourgain-norm} and \eqref{Sob-norm}, and Plancherel's theorem.
\end{proof}

The bounds are pointwise in $\omega$, and for the iteration in the space \eqref{iteration-space} one must now take the $L^2(\Omega)$ norm of them. So for example, one needs to control
\[
  \mathbb E \left( \norm{\phi_+}_{X^{r,b}_{+\japanese{\xi}}(0,T)}^2
  \norm{\psi_- - \Psi_-}_{X^{s,b}_{-\xi}(0,T)}^2 \right),
\]
where $\psi_-$ and $\Psi_-$ represent different iterates. The only reasonable way to estimate this, seems to be
\[
  \left( \sup_{\omega} \norm{\phi_+}_{X^{r,b}_{+\japanese{\xi}}(0,T)}^2 \right) \mathbb E \left( 
  \norm{\psi_- - \Psi_-}_{X^{s,b}_{-\xi}(0,T)}^2 \right),
\]
so one needs to control the norms of the iterates uniformly in $\omega$. Here, a further difference from the deterministic case becomes apparent, since there one usually chooses $R > 0$ and considers initial data whose norm is at most $R$. Then for $T > 0$ small enough depending on $R$, the Bourgain norms in the iteration are all bounded by $R$ times some constant. This will not work in the stochastic case, since a bound in the space \eqref{iteration-space} does not imply a pointwise bound in $\omega$. Instead, as is usual in stochastic problems with multiplicative noise, one must truncate the equations. Following the approach in \cite{Bouard_Debussche2007}, for a given $R > 0$ we consider a truncated version of \eqref{mild_psi}, \eqref{mild_phi} where in the deterministic integrals, each unknown is multiplied by the cutoff
\begin{equation}\label{DKG-cutoff}
  \Theta(t) = \theta_R\left(
  \norm{\psi_+}_{\widetilde X^{s,b}_{+\xi}(0,t)}^2
  +
  \norm{\psi_-}_{\widetilde X^{s,b}_{-\xi}(0,t)}^2
  +
  \norm{\phi_+}_{\widetilde X^{r,b}_{+\japanese{\xi}}(0,t)}^2
  \right),
\end{equation}
where $\theta \colon \R \to \R$ is any smooth, compactly supported cutoff function with $\theta(t)=1$ for $t \in [0,1]$, and we define $\theta_R(x) = \theta(x/R)$. So for example, the integral term
\[
  \int_0^t  S_{+\xi}(t-\sigma) ( \phi_+\psi_- )(\sigma) \, d\sigma,
\]
considered above, is replaced by
\[
  \int_0^t  S_{+\xi}(t-\sigma) \left(
  \Theta \phi_+
  \Theta
  \psi_- \right)(\sigma) \, d\sigma
\]
and similarly for the other bilinear terms. For technical reasons, inside the cutoffs we do not use the Bourgain restriction norm as defined in \eqref{Bourgain-norm}, \eqref{XRestNorm}, but rather the equivalent norm \eqref{ModEquivalence}, discussed in detail in Section \ref{Mod-B-norm}.

With the truncation, and making use of Proposition \ref{MainCutoffLemma}, the bilinear terms can be controlled in the space \eqref{iteration-space}, and by iteration one can prove existence up to a small time $T > 0$ depending only on $R$. Repeating this argument one obtains existence on a time interval of any size. Letting $R$ tend to $\infty$, this implies the existence of a maximal solution of the original, non-truncated problem. The uniqueness requires a separate argument. The details are shown in Section \ref{Abstract-WP} in an abstract framework.
As mentioned, the necessary assumptions in Section \ref{Abstract-WP} are verified
on account of the bounds in Lemmas \ref{DKG-HS-lemma}, \ref{DKG-bilinear-lemma} and \ref{DKG-linear-lemma}, as well as Lemma \ref{DKG-high-reg-lemma} below. The details are discussed in Section \ref{ConcreteCase}.

The following estimates show that the deterministic integrals make sense as Bochner integrals if the regularity is sufficiently high, and that this fails if  $s \le 0$.

\begin{lemma}\label{DKG-high-reg-lemma}
Assume that
\[
  s > 0, \qquad s \le r \le s+1, \qquad \frac12 < r < \frac12 + 2s.
\]
Then there exists a constant $C$ such that
\[
  \norm{fg}_{H^{s}}
  \le C \norm{f}_{H^{r}} \norm{g}_{H^{s}}
  \quad \text{and} \quad
  \norm{fg}_{H^{r-1}}
  \le C \norm{f}_{H^{s}} \norm{g}_{H^{s}}
\]
for all Schwartz functions $f$ and $g$ on $\R$. Moreover, if $s \le 0$, the above estimates cannot both hold, for any $r \in \R$ and $C > 0$.
\end{lemma}

\begin{proof}
This follows from the Sobolev product law \eqref{Sob-product-1}, \eqref{Sob-product-2}.
\end{proof}

\begin{remark}
    The case $m = 0$ in \eqref{Stratonovich_Dirac} does not bring
    anything new to our analysis.
    Indeed, if $m = 0$ then we can add $\phi$
    to both sides of the second line in \eqref{Stratonovich_Dirac},
    which gives rise to Equation \eqref{mild_phi} with an additional linear term
    \[
        \frac{i}{2} \int_0^t S_{+\japanese{\xi}}(t-s) \japanese{D_x}^{-1}
        \phi(s) ds
    \]
    on the right hand side.
    It can be treated by Lemma \ref{DKG-linear-lemma} and \eqref{Bourgain-2}--\eqref{Bourgain-4}.
\end{remark}

\section{Bounds for Hilbert-Schmidt operators}
\label{Stochastic_integral}
\setcounter{equation}{0}

Our main aim in this section is to prove Lemma \ref{DKG-HS-lemma}.

To this end, we require the following lemma. It corresponds to Lemma 2.6 in \cite{Bouard_Debussche2007}, but we remove an additional assumption made there, namely that the convolution kernel $\mathfrak k$ is in $L^1 \cap L^2$.

\begin{lemma}
\label{Hilbert_Schmidt_operator_lemma}
	Let
    \(
        \mathfrak k
        \in
        L^2 \left( \R^d, \R \right)
	,
    \)
	\(
		v \in L^2 \left( \R^d, \mathbb C \right)
	\)
	and let
	\(
		\mathfrak K
	\)
	be the convolution operator defined by
	\begin{equation}\label{conv-op}
	  \mathfrak Kf(x)
	=
	\int_{ \R^d }
	\mathfrak k(x - y) f(y) dy.
	\end{equation}
	Then for any orthonormal basis
	\(
		\{ e_j \} _{ j \in \mathbb N }
	\)
    of
    \(
        L^2 \left( \R^d, \R \right)
    \)
    we have
	\[
		\sum_{j = 1}^{\infty}
		\left|
			\mathcal F \left( v \mathfrak K e_j \right)(\xi)
		\right| ^2
		=
		\frac{1}{(2\pi)^d} \int_{\R^d} \Abs{\widehat v(\xi-\eta) \widehat{\mathfrak k}(\eta)}^2 d\eta
    < \infty
	\]
	for a.e. $\xi \in \R^d$.
\end{lemma}

\begin{proof} Set $f_j(x) = e_j(-x)$. Then $\{f_j\}$ is also an orthonormal basis of $L^2(\R^d)$, and $\overline{\widehat{f_j}(\xi)} = \widehat{e_j}(\xi)$, by the assumption that $e_j$ is real valued.

Since
\[
  \int_{\R^d} \int_{\R^d} \Abs{\widehat v(\xi-\eta) \widehat{\mathfrak k}(\eta)}^2 d\xi \, d\eta
  =
  \norm{\widehat v\,}_{L^2}^2 \fixednorm{\widehat{\mathfrak k}}_{L^2}^2
  =
  (2\pi)^{2d} \norm{v}_{L^2}^2 \norm{\mathfrak k}_{L^2}^2
\]
it follows that
\[
  \norm{F_\xi}_{L^2}^2 
  =
  \frac{1}{(2\pi)^d} \int_{\R^d} \Abs{\widehat v(\xi-\eta) \widehat{\mathfrak k}(\eta)}^2 d\eta
  < \infty \quad \text{for a.e.~$\xi \in \R^d$},
\]
where $F_\xi(x) = \mathcal F^{-1}\left\{ \eta \mapsto \widehat v(\xi-\eta) \widehat{\mathfrak k}(\eta) \right\}$. Applying Parseval's identity we have, for a.e.~$\xi$,
\[
  \norm{F_\xi}_{L^2}^2
  =
  \sum_{j=1}^\infty \Abs{\innerprod{F_\xi}{f_j}_{L^2}}^2
  =
  \sum_{j=1}^\infty \Abs{\frac{1}{(2\pi)^d} \Innerprod{\widehat{F_\xi}}{\widehat{f_j}}_{L^2}}^2
  =
  \sum_{j=1}^\infty
  \Abs{\frac{1}{(2\pi)^d} \int \widehat v(\xi-\eta) \widehat{\mathfrak k}(\eta) \widehat{e_j}(\eta) \, d\eta}^2
  < \infty.
\]
To finish the proof we only have to notice that $v\mathfrak Ke_j$ belongs to $L^2$ (the operator $\mathfrak K$ maps $L^2$ into $L^\infty$, by H\" older's inequality) and has Fourier transform
\[
  \mathcal F ( v \mathfrak K e_j )(\xi)
  = \frac{1}{(2\pi)^d} \int \widehat v(\xi-\eta) \widehat{\mathfrak k * e_j}(\eta) \, d\eta
  = \frac{1}{(2\pi)^d} \int \widehat v(\xi-\eta) \widehat{\mathfrak k}(\eta) \widehat{e_j}(\eta) \, d\eta.
\]
We claim that this equality holds in $L^2(\R^d)$, hence for a.e.~$\xi$. To prove the claim, approximate the $L^2$ functions $v$, $\mathfrak k$ and $e_j$ by Schwartz functions, for which the equality clearly holds, and then pass to the limit using the fact that $(2\pi)^{-d/2} \mathcal F$ is an isometry on $L^2$, and the bound
\[
  \norm{\int \widehat v(\xi-\eta) \widehat{\mathfrak k}(\eta) \widehat{e_j}(\eta) \, d\eta}_{L^2_\xi}
  \le
  \norm{\widehat v}_{L^2} \int \Abs{\widehat{\mathfrak k}(\eta) \widehat{e_j}(\eta)} \, d\eta
  \le
  \norm{\widehat v}_{L^2}
  \fixednorm{\widehat{\mathfrak k}}_{L^2}
  \norm{\widehat{e_j}}_{L^2},
\]
where we applied Minkowski's integral inequality and H\"older's inequality.
\end{proof}

\begin{remark}
\label{Hilbert_Schmidt_remark}
    Integrating both sides of the equality in Lemma \ref{Hilbert_Schmidt_operator_lemma},
    one recovers the well-known fact
    that $v \mathfrak K$ is a Hilbert-Schmidt operator on $L^2(\R)$, and
	$
       \norm{v\mathfrak K}_{\mathcal L_2(L^2,L^2)}
       =
       \norm{v} _{L^2}
		\norm{\mathfrak k} _{L^2}.
    $
\end{remark}

\begin{corollary}
\label{Hilbert_Schmidt_operator_corollary}
Let $s \in \R$. Assume that $\mathfrak k \in H^{\Abs{s}} \left(\R^d,\R\right)$. Then the convolution operator $\mathfrak K$ defined by \eqref{conv-op} satisfies
\[
  \norm{ v \mathfrak K }_{\mathcal L_2 \left( L^2,H^{s} \right)}
  \le C \norm{v}_{H^{s}} \norm{\mathfrak k}_{H^{\Abs{s}}}
  \quad
  \text{for all $v \in \mathcal S\left( \R^d \right)$},
\]
where the constant depends only on $s$. Thus the map $v \mapsto v \mathfrak K$
extends to a bounded linear map from $H^s \left(\R^d \right)$
into $\mathcal L_2 \left(L^2 \left(\R^d \right), H^s \left(\R^d \right) \right)$.
\end{corollary}

\begin{proof}
Integrating both sides of the equality in Lemma 
\ref{Hilbert_Schmidt_operator_lemma} with respect to
\(
    \japanese{\xi}^{ 2s } d\xi
\)
one obtains
\[
    \norm{ v \mathfrak K }
    _{ \mathcal L_2( L^2, H^s ) }^2
    =
    \frac 1{(2\pi)^d}
    \int_{\R^d} \int_{\R^d}
    \Abs{\widehat v(\xi-\eta) \widehat{\mathfrak k}(\eta)}^2
    \japanese{\xi}^{ 2s }
    d\xi \, d\eta
    .
\]
Applying the inequality
\begin{equation}\label{japbound}
    \japanese{\xi}^s \le C_s \japanese{\xi-\eta}^s \japanese{\eta}^{\Abs{s}}
  \quad \text{for all $\xi,\eta \in \R^d$, $s \in \R$},
\end{equation}
the claimed bound follows immediately.
\end{proof}

With this corollary in hand, we can now prove Lemma \ref{DKG-HS-lemma}.

\begin{proof}[Proof of Lemma \ref{DKG-HS-lemma}]
The bound on $M_1(f) = f \mathfrak K_1$ is immediate from Corollary \ref{Hilbert_Schmidt_operator_corollary}. For $M_2(g) = \japanese{D_x}^{-1} g \mathfrak K_2$ we write
\[
  \norm{M_2(g)}_{\mathcal L_2(L^2,H^r)} =  \norm{g \mathfrak K_2}_{\mathcal L_2(L^2,H^{r-1})}
\]
and use again the corollary; if $0 \le r \le 1$, we bound by
\[
  \norm{g \mathfrak K_2}_{\mathcal L_2(L^2,L^2)}
  \le C\norm{\mathfrak k}_{L^2}\norm{g}_{L^2},
\]
while if $r \ge 1$ we bound by
\[
  C\norm{\mathfrak k}_{H^{r-1}}\norm{g}_{H^{r-1}}.
\]
This concludes the proof of the lemma.
\end{proof}

\section{Bilinear bounds}
\label{Bilinear_bounds}
\setcounter{equation}{0}

In this section we prove Lemma \ref{DKG-bilinear-lemma}, and we prove some additional null form bound, stated in two lemmas at the end of the section, that will be needed in the proof of global existence. Throughout this section the space dimension is $d=1$.

We first note the following basic product law for Bourgain norms.

\begin{lemma}\label{ProdLemma}
Let $s_1,s_2,s_3 \in \R$ and $b_1,b_2,b_3 \ge 0$, and assume that
\begin{itemize}
\item $s_1 + s_2 + s_3 \ge 1/2$, and
\item $\min_{i \neq j} (s_i + s_j) \ge 0$, and
\item the two preceding inequalities are not both equalities, and
\item $b_1 + b_2 + b_3 > 1/2$.
\end{itemize}
Then there is a constant $C$ such that the bound
\begin{equation}\label{XProd}
  \norm{uv}_{X^{-s_1,-b_1}_{h_1(\xi)}} \le C \norm{u}_{X^{s_2,b_2}_{h_2(\xi)}} \norm{v}_{X^{s_3,b_3}_{h_3(\xi)}}
\end{equation}
holds for all Schwartz functions $u$ and $v$ on $\R_t \times \R_x$ and any choice of $h_1,h_2,h_3 \in C(\R,\R)$.
\end{lemma}

\begin{proof}
By Plancherel's theorem and $L^2$ duality, \eqref{XProd} can be reformulated as
\begin{equation*}
  \int_{\R^4} 
  \frac{
  f_1(\tau,\xi) f_2(\tau-\lambda,\xi-\eta) f_3(\lambda,\eta)
   \, d\lambda \, d\eta \, d\tau \, d\xi
  }
  {
  \japanese{\xi}^{s_1} \japanese{\xi-\eta}^{s_2} \japanese{\eta}^{s_3}
  \japanese{\tau+h_1(\xi)}^{b_1}
  \japanese{\tau-\lambda+h_2(\xi-\eta)}^{b_2}
  \japanese{\lambda+h_3(\eta)}^{b_3}
  }
  \le C \prod_{i=1}^3 \norm{f_i}_{L^2},
\end{equation*}
where the $f_i$ are non-negative. By symmetry, it suffices to consider the region where $\japanese{\lambda+h_3(\eta)}$ is the minimum among $\japanese{\tau+h_1(\xi)}$, $\japanese{\tau-\lambda+h_2(\xi-\eta)}$ and $\japanese{\lambda+h_3(\eta)}$, and then the left side is bounded by
\begin{equation*}
  \int_{\R^4} 
  \frac{
  f_1(\tau,\xi) f_2(\tau-\lambda,\xi-\eta) f_3(\lambda,\eta)
   \, d\lambda \, d\eta \, d\tau \, d\xi
  }
  {
  \japanese{\xi}^{s_1} \japanese{\xi-\eta}^{s_2} \japanese{\eta}^{s_3}
  \japanese{\lambda+h_3(\eta)}^{b_1+b_2+b_3}
  },
\end{equation*}
where we used the assumption $b_1,b_2,b_3 \ge 0$. This shows that it is enough to prove
\begin{equation}\label{XProd-2}
  \norm{uv}_{X^{-s_1,0}_{h_1(\xi)}} \le C \norm{u}_{X^{s_2,0}_{h_2(\xi)}} \norm{v}_{X^{s_3,b_1+b_2+b_3}_{h_3(\xi)}}.
\end{equation}
But by the assumptions on $s_1,s_2,s_3$ we can apply the product law \eqref{Sob-product-1} to get
\[
  \norm{u(t)v(t)}_{H^{-s_1}} \le C_{s_1,s_2,s_3} \norm{u(t)}_{H^{s_2}} \norm{v(t)}_{H^{s_3}}
\]
for all $t$. Taking now the $L^2$ norm with respect to $t$ of both sides, and using \eqref{Bourgain-4} to bound
\[
  \sup_{t \in \R}\norm{v(t)}_{H^{s_3}} \le C_{b_1,b_2,b_3} \norm{v}_{X^{s_3,b_1+b_2+b_3}_{h_3(\xi)}},
\]
we then obtain \eqref{XProd-2}, and this concludes the proof.
\end{proof}

Now we apply the above lemma to obtain estimates of the form
\begin{equation}\label{NullProd}
  \norm{uv}_{X_{\pm\japanese{\xi}}^{-s_1,-b_1}} \le C \norm{u}_{X^{s_2,b_2}_{+\xi}} \norm{v}_{X^{s_3,b_3}_{-\xi}}.
\end{equation}
But here we can gain some regularity compared to the generic case, due to the opposite signs in the dispersion relations on the right hand side; the symbols $\tau+\xi$ and $\tau-\xi$ correspond to transport equations with propagation in transverse directions. Thus, \eqref{NullProd} is a \emph{null form} estimate.

Let us denote by $b_\mathrm{min}, b_\mathrm{med}, b_\mathrm{max}$ the minimum, median and maximum, respectively, of the three numbers $b_1, b_2, b_3$.
We then have the following result.

\begin{lemma}\label{NullProdLemma}
Suppose $s_1,s_2,s_3 \in \R$ and $b_1,b_2,b_3 \ge 0$. Then the following conditions are sufficient for the null form estimate \eqref{NullProd} to hold for all Schwartz functions $u$ and $v$ on $\R_t \times \R_x$:
\begin{itemize}
\item $s_1 + s_2 + s_3 + b_\mathrm{min} \ge 1/2$, and
\item $\min\left( s_2 + s_3 + b_\mathrm{min}, s_1 + s_2, s_1 + s_3 \right) \ge 0$, and
\item the two preceding inequalities are not both equalities, and
\item $b_\mathrm{med} + b_\mathrm{max} > 1/2$.
\end{itemize}
\end{lemma}

\begin{proof}
We reformulate \eqref{NullProd} as
\begin{equation*}
  \int_{\R^4} \frac{f_1(\tau,\xi) f_2(\tau-\lambda,\xi-\eta) f_3(\lambda,\eta) \, d\lambda \, d\eta \, d\tau d\xi}
  {\japanese{\xi}^{s_1} \japanese{\xi-\eta}^{s_2} \japanese{\eta}^{s_3}
  \japanese{\tau\pm\japanese{\xi}}^{b_1} \japanese{\tau-\lambda+(\xi-\eta)}^{b_2} \japanese{\lambda-\eta}^{b_3} }
  \le C \prod_{i=1}^3 \norm{f_i}_{L^2},
\end{equation*}
where the $f_i$ are non-negative. By the triangle inequality,
\begin{align*}
  2\Abs{\eta} &\le \Abs{\tau+\xi} + \Abs{\tau-\lambda+(\xi-\eta)} + \Abs{\lambda-\eta},
  \\
  2\Abs{\xi-\eta} &\le \Abs{\tau-\xi} + \Abs{\tau-\lambda+(\xi-\eta)} + \Abs{\lambda-\eta},
\end{align*}
implying
\[
  \min\left(\japanese{\eta},\japanese{\xi-\eta}\right)
  \le C \max\left( \japanese{\tau\pm\japanese{\xi}}, \japanese{\tau-\lambda+(\xi-\eta)}, \japanese{\lambda-\eta} \right),
\]
since $\japanese{\tau\pm\Abs{\xi}}$ is comparable to $\japanese{\tau\pm\japanese{\xi}}$. Thus, we can reduce \eqref{NullProd} to estimates
\begin{align*}
  \norm{uv}_{X_{\pm\japanese{\xi}}^{-s_1,-B_1}} &\le C \norm{u}_{X^{s_2+b_\mathrm{min},B_2}_{+\xi}} \norm{v}_{X^{s_3,B_3}_{-\xi}},
  \\
  \norm{uv}_{X_{\pm\japanese{\xi}}^{-s_1,-B_1}} &\le C \norm{u}_{X^{s_2,B_2}_{+\xi}} \norm{v}_{X^{s_3+b_\mathrm{min},B_3}_{-\xi}},
\end{align*}
where $B_1,B_2,B_3 \ge 0$ and $B_1+B_2+B_3=b_\mathrm{med} + b_\mathrm{max}$. Applying Lemma \ref{ProdLemma} we therefore obtain the claimed result.
\end{proof}

We are now ready to prove Lemma \ref{DKG-bilinear-lemma}.

\begin{proof}[Proof of Lemma \ref{DKG-bilinear-lemma}]
We start with \eqref{Nbound}, which reads
\[
  \norm{\overline{\psi} \psi'}_{X^{r-1,-b}_{\pm\japanese{\xi}}}
  \le C \norm{\psi}_{X^{s,b}_{+\xi}}
  \norm{\psi'}_{X^{s,b}_{-\xi}}.
\]
By \eqref{XRestNorm_conjugate}, $\norm{\psi}_{X^{s,b}_{+\xi}} = \fixednorm{\overline\psi}_{X^{s,b}_{+\xi}}$, so we can remove the conjugation on $\psi$. Now we apply Lemma \ref{NullProdLemma} with $s_1=1-r$, $s_2=s_3=s$ and $b_1=b_2=b_3=b$, and conclude that \eqref{Nbound} holds if
\begin{equation}\label{DKG-bilinear-lemma-1}
  b > \frac14,
  \quad
  2s+b \ge 0,
  \quad
  1-r+s \ge 0,
  \quad
  1-r+2s+b > \frac12.
\end{equation}

It remains to consider \eqref{M+bound} (the proof of \eqref{M-bound} is similar). We have to show
\[
  \norm{\phi \psi}_{X^{s,-b}_{+\xi}}
  \le C \norm{\phi}_{X^{r,b}_{\pm\japanese{\xi}}}
  \norm{\psi}_{X^{s,b}_{-\xi}}.
\]
 By \eqref{Xduality},
\[
  \norm{\phi \psi}_{X^{s,-b}_{+\xi}}
  =
  (2 \pi)^2
  \sup_{\norm{\psi'}_{X^{-s,b}_{+\xi}} = 1}
  \Abs{\int_{\R^2} \phi \psi \overline{\psi'} \, dt \,dx},
\]
where by \eqref{Xduality2},
\[
  (2 \pi)^2
  \Abs{\int_{\R^2} \phi \psi \overline{\psi'} \, dt \,dx}
  \le
  \norm{\phi}_{X^{r,b}_{\pm\japanese{\xi}}}
  \norm{\psi \overline{\psi'}}_{X^{-r,-b}_{\pm\japanese{\xi}}}.
\]
Thus we have reduced to obtaining a bound
\[
  \norm{\overline{\psi'}\psi }_{X^{-r,-b}_{\pm\japanese{\xi}}}
  \le C \norm{\psi'}_{X^{-s,b}_{+\xi}} \norm{\psi}_{X^{s,b}_{-\xi}}.
\]
Applying Lemma \ref{NullProdLemma} with $s_1=r$, $s_2=-s$, $s_3=s$ and $b_1=b_2=b_3=b$, we conclude that \eqref{M+bound} holds if
\begin{equation}\label{DKG-bilinear-lemma-2}
  b > \frac14,
  \quad
  r+b > \frac12,
  \quad
  r - s \ge 0,
  \quad
  r + s \ge 0.
\end{equation}
If we take $b=1/2-\delta$ with $\delta > 0$ sufficiently small, then both \eqref{DKG-bilinear-lemma-1} and \eqref{DKG-bilinear-lemma-2} are satisfied if
\[
  s > - \frac14, \qquad \Abs{s} \le r \le s+1, \qquad 0 < r < 1+2s,
\]
proving the main part of Lemma \ref{DKG-bilinear-lemma}. In the special case $s=0$, $r \in (1/4, 1/2)$, mentioned at the end of Lemma \ref{DKG-bilinear-lemma}, it is clear that \eqref{DKG-bilinear-lemma-1} and \eqref{DKG-bilinear-lemma-2} hold for any $b > 1/4$.
\end{proof}

We conclude this section with some additional null form estimates, given in the next two lemmas.

\begin{lemma}\label{DKG-null-form-lemma1}
Let $1 \le p \le 2$.
Then for all Schwartz functions $u$ and $v$ on $\R_t \times \R_x$
the following bound holds true
\begin{equation}\label{Simple-null-bound}
   \norm{uv}_{L^p_{t,x}} \le C \norm{u}_{X^{0,b}_{+\xi}} \norm{v}_{X^{0,b}_{-\xi}}
   \quad \text{where} \quad
   b=\begin{cases}
       1-1/p &\text{if $1 \le p < 2$},
       \\
       1/2+\varepsilon &\text{if $p = 2$}.
   \end{cases}
\end{equation}
Here $\varepsilon > 0$ is arbitrarily small; $C$ depends on $p$, and on $\varepsilon$ if $p=2$.
\end{lemma}

\begin{proof}
In null coordinates $(s,y) = (t+x,t-x)$ on $\R^2$, the desired inequality reads
\[
  \norm{uv}_{L^p_{s,y}} \le C \norm{\japanese{D_s}^b u}_{L^2_{s,y}} \norm{\japanese{D_y}^b v}_{L^2_{s,y}}.
\]
But by H\"older's inequality, letting $q \in [2,\infty]$ be defined by $1/p = 1/2 + 1/q$,
\[
    \norm{uv}_{L^p_{s,y}} \le \norm{u}_{L^q_s(L^2_y)} \norm{v}_{L^2_s(L^q_y)} 
    \le \norm{u}_{L^2_y(L^q_s)} \norm{v}_{L^2_s(L^q_y)}
    ,
\]
where we used Minkowski's integral inequality in the last step. The desired inequality now follows by applying the Sobolev embedding $H^b(\R) \hookrightarrow L^q(\R)$, which holds for $b = 1/2+\varepsilon$ if $q=\infty$, and for $b=1/2-1/q$ if $2 \le q < \infty$.
\end{proof}

\begin{lemma}\label{Null-form-lemma2} Let $r > 0$ and $h \in C(\R,\R)$. Then for all Schwartz functions $u$ and $v$ on $\R_t \times \R_x$, we have the estimate
\begin{equation}\label{Sobolev-bound-cor}
   \norm{uv}_{X_{h(\xi)}^{-r,-r}} \le C \norm{u}_{X^{0,b}_{+\xi}} \norm{v}_{X^{0,b}_{-\xi}},
   \quad \text{where} \quad
   b=\begin{cases}
       1/2-r &\text{if $0 < r < 1/2$},
       \\
       \varepsilon &\text{if $r = 1/2$},
       \\
       0 &\text{if $r > 1/2$}.
   \end{cases}
\end{equation}
Here $\varepsilon > 0$ is arbitrarily small; $C$ depends on $r$, and on $\varepsilon$ if $r=1/2$, but not on $h$.
\end{lemma}

\begin{proof}
Define $p=p(r) \in [1,2)$ by (i) $1/p = r + 1/2$ if $0 < r < 1/2$, (ii) $1/p = 1-\varepsilon$ if $r=1/2$, and (iii) $p=1$ if $r > 1/2$. Here we can take any $0 < \varepsilon < 1/2$.
Then the Sobolev embedding
\begin{equation}\label{Sob_emb}
    L^p(\R) \hookrightarrow H^{-r}(\R)
\end{equation}
holds.
Setting $F(t,\xi) = \japanese{\xi}^{-r} e^{ith(\xi)} \widehat{uv}(t,\xi)$, we then have
\[
  \norm{uv}_{X_{h(\xi)}^{-r,-r}}
  =
  \norm{ \norm{F(t,\xi)}_{H^{-r}_t}}_{L^2_\xi}
  \le
  C
  \norm{ \norm{F(t,\xi)}_{L^p_t}}_{L^2_\xi}
  \le
  C
  \norm{ \norm{F(t,\xi)}_{L^2_\xi}}_{L^p_t}
  =
  C
  \norm{ \norm{uv}_{H^{-r}_x}}_{L^p_t}
  \le
  C
  \norm{ uv }_{L^p_{t,x}}
  ,
\]
where we applied \eqref{Sob_emb} twice, to get the first and third inequalities, and we used Minkowski's integral inequality to get the second inequality. The proof can now be concluded by appealing to Lemma \ref{DKG-null-form-lemma1}; we are in the case $1 \le p < 2$, so \eqref{Simple-null-bound} holds with $b=1-1/p$.
\end{proof}

As a consequence we get the following result used later in Section \ref{Global_existence}.

\begin{corollary}
\label{Null-form-corollary}
    Let
    \[
        0 < r < b < \frac 12
        , \quad
        0 < \mu < 1
        , \quad
        \frac 12 - r = \mu b
        .
    \]
    Then for any $h \in C(\R,\R)$ and $T > 0$
    we have the estimates
    \[
        \norm{ \phi \psi }_{ X^{0, -b}_{\pm \xi}(0,T) }
        \leqslant
        C
        \norm{ \phi }_{ X^{r, b}_{h(\xi)}(0,T) }
        \norm{ \psi }_{ X^{0, b}_{\mp \xi}(0,T) }^{ \mu }
        \norm{ \psi }_{ L^2((0,T) \times \mathbb R) }^{ 1 - \mu }
        ,
    \]
    where $C$ depends on $r, b$ but neither on $h$ nor on $T$.
\end{corollary}

\begin{proof}

Firstly, by \eqref{X_characteristic_norm} we can substitute the restriction norm
with the norm of trivial extension
\[
    \norm{ \phi \psi }_{ X^{0, -b}_{\pm \xi}(0,T) }
    \leqslant
    \norm{ \mathbb 1_{(0, T)} \phi \mathbb 1_{(0, T)} \psi }_{ X^{0, -b}_{\pm \xi} }
\]
and without loss of generality we will write simply $\phi, \psi$
while meaning in fact the trivial extensions
\(
    \mathbb 1_{(0, T)} \phi, \mathbb 1_{(0, T)} \psi
    .
\)
Secondly,
we apply Lemma \ref{Null-form-lemma2} via duality \eqref{Xduality}, \eqref{Xduality2} as follows.
Given an arbitrary test function with
\(
    \norm u _{ X^{0, b}_{\pm \xi} } = 1
\)
consider the integral
\[
    (2 \pi)^2
    \Abs{
        \int
        \phi \psi \overline{u}
        \,
        dt dx
    }
    \leqslant
    \norm{ \phi }_{ X^{r, r}_{h(\xi)} }
    \norm{ \overline{\psi} u }_{ X^{-r, -r}_{h(\xi)} }
    \leqslant
    C_r
    \norm{ \phi }_{ X^{r, b}_{h(\xi)} }
    \norm{ \overline{\psi} }_{ X^{0, 1/2-r}_{\mp \xi} }
    \norm{ u }_{ X^{0, 1/2-r}_{\pm \xi} }
    \leqslant
    C_r
    \norm{ \phi }_{ X^{r, b}_{h(\xi)} }
    \norm{ \psi }_{ X^{0, \mu b}_{\mp \xi} }
    ,
\]
where $C_r$ comes exactly from the previous lemma.
Here we have used the conjugation property \eqref{XRestNorm_conjugate}
and obvious embeddings of Bourgain spaces.
Appealing to the interpolation argument
one obtains
\[
    \norm{ \psi }_{ X^{0, \mu b}_{\mp \xi} }
    \leqslant
    \norm{ \psi }_{ X^{0, b}_{\mp \xi} }^{ \mu }
    \norm{ \psi }_{ X^{0, 0}_{\mp \xi} }^{ 1 - \mu }
    =
    (2 \pi)^{1 - \mu}
    \norm{ \psi }_{ X^{0, b}_{\mp \xi} }^{ \mu }
    \norm{ \psi }_{ L^2 \left( \mathbb R^2 \right) }^{ 1 - \mu }
    .
\]
Now taking the supremum of the above integral over test functions
and recalling that $\phi, \psi$ stand for the trivial extensions
\(
    \mathbb 1_{(0, T)} \phi, \mathbb 1_{(0, T)} \psi
\)
one obtains the needed statement
\[
    \norm{ \phi \psi }_{ X^{0, -b}_{\pm \xi}(0,T) }
    \leqslant
    (2 \pi)^{1 - \mu}
    C_r C_b^{1 + \mu}
    \norm{ \phi }_{ X^{r, b}_{h(\xi)} (0,T)}
    \norm{ \psi }_{ X^{0, b}_{\mp \xi}(0,T) }^{ \mu }
    \norm{ \psi }_{ L^2((0,T) \times \mathbb R) }^{ 1 - \mu }
\]
by \eqref{Xduality} and \eqref{X_characteristic_norm}.
Finally,
the interpolation used above is justified as
\[
    \norm{ f(\tau, \xi) \japanese{\tau}^{ \mu b } }_{ L^2_{\tau, \xi} }
    \leqslant
    \norm{ f^{\mu} (\tau, \xi) \japanese{\tau}^{ \mu b } }_{ L^{2/\mu}_{\tau, \xi} }
    \norm{ f^{1 - \mu}(\tau, \xi) }_{ L^{2/(1 - \mu)}_{\tau, \xi} }
\]
by H\"older's inequality with
\(
    f(\tau, \xi)
    =
    \Abs{ \widehat \psi(\tau \pm \xi, \xi) }
\)
 and the Bourgain norm definition \eqref{Bourgain-norm}.

\end{proof}

\section{Abstract well-posedness for dispersive PDE systems with noise}
\label{Abstract-WP}
\setcounter{equation}{0}

Let $W(t)$ be a cylindrical Wiener process, as in Section \ref{Prelims}.

\subsection{Notation and definitions}
Let $d,n \in \N$.
Given $\mathbf s = (s_1,\dots,s_n) \in \R^n$,
define
\[
  \mathbf H^{\mathbf s} \left( \R^d \right) = H^{s_1}\left( \R^d \right)
  \times \dots \times H^{s_n}\left( \R^d \right)
\]
with the product norm
\[
  \norm{\mathbf f}_{\mathbf H^{\mathbf s}\left( \R^d \right)} =
  \left( \norm{f_1}_{H^{s_1}\left( \R^d \right)}^2 + \dots + \norm{f_n}_{H^{s_n}\left( \R^d \right)}^2 \right)^{1/2}
  \quad
  \text{for}
  \quad
  \mathbf f = \begin{pmatrix} f_1 \\ \vdots \\ f_n \end{pmatrix}.
\]
Given $h_1,\dots,h_n \in C(\R^d,\R)$, define the Fourier multiplier $\mathbf h(D_x)$ and the group $\mathbf S(t)$ by
\[
  \mathbf h(D_x) \mathbf f = \begin{pmatrix} h_1(D_x) f_1 \\ \vdots \\ h_n(D_x) f_n \end{pmatrix},
  \qquad  \mathbf S(t) \mathbf f = \begin{pmatrix} S_{h_1(\xi)}(t)f_1 \\ \vdots \\ S_{h_n(\xi)}(t)f_n \end{pmatrix},
\]
where the latter is then an isometry of $\mathbf H^{\mathbf s}$. Further, given $b \in \R$,
we define
\[
    \mathbf X^{\mathbf s,b} \left( \R \times \R^d \right )
    =
    X_{h_1(\xi)}^{s_1,b} \left( \R \times \R^d \right )
    \times  \dots \times
    X_{h_n(\xi)}^{s_n,b} \left( \R \times \R^d \right )
\]
with the product norm. The restriction to $(S,T) \times \R^d$ is denoted $\mathbf X^{\mathbf s,b}(S,T)$. 

The following space will play a key role.

\begin{definition}\label{Z-def}
For $0 \le S < T$ let, with notation as in \eqref{L2-prog},
\[
  Z^{\mathbf s, b}(S,T) = \mathbb L^2\left(\Omega,\mathbf X^{\mathbf s,b}(S,T)
  \cap
  C\left([S,T],\mathbf H^{\mathbf s}\right) \right)
 \]
with norm
\[
  \norm{\mathbf u}_{Z^{\mathbf s, b}(S,T)} = \norm{\mathbf u}_{L^2\left(\Omega,\mathbf X^{\mathbf s,b}(S,T) \right)}
  +
  \norm{\mathbf u}_{L^2\left(\Omega,C\left([S,T],\mathbf H^{\mathbf s}\right) \right)}.
\]
Note that this space is complete.
\end{definition}

\begin{remark}\label{Z-remark}
From the embedding
\begin{equation}\label{X-emb1}
  \norm{\mathbf u}_{L^2 \left( [S,T] , \mathbf H^{\mathbf s} \right)}
  \le
  \norm{\mathbf u}_{\mathbf X^{\mathbf s, b}(S,T)}
  \quad \text{if $b \ge 0$},
\end{equation}
we infer that $Z^{\mathbf s, b}(S,T) \hookrightarrow  L^2 \left( [S,T] \times \Omega , \mathbf H^{\mathbf s} \right)$ for $b \ge 0$.
\end{remark}

By Lemma \ref{Gluing-lemma}, we see immediately that the $Z$-space has the following gluing property.

\begin{lemma}\label{Concatenation-lemma}
Let $0 < S < S'$ and $-1/2 < b < 1/2$.
Given $\mathbf u \in Z^{\mathbf s,b}(0,S)$ and $\mathbf v \in Z^{\mathbf s,b}(S,S')$, define $[\mathbf u,\mathbf v]$ by
\[
  [\mathbf u,\mathbf v](t)
  =
  \begin{cases}
  \mathbf u(t) &\text{for $0 \le t \le S$}
  \\
  \mathbf v(t) &\text{for $S < t \le S'$}.
  \end{cases}
\]
Then
\[
  \norm{ [\mathbf u,\mathbf v] }_{ \mathbf X^{\mathbf s,b}(0,S') }
  \le C_{b}
  \left(
  \norm{ \mathbf u }_{ \mathbf X^{\mathbf s,b}(0,S) }
  +
  \norm{ \mathbf v }_{ \mathbf X^{\mathbf s,b}(S,S') }
  \right),
\]
where $C_{b}$ depends only on $b$. Moreover, if $\mathbf u(S) = \mathbf v(S)$, then $[\mathbf u,\mathbf v] \in Z^{\mathbf s,b}(0,S')$.
\end{lemma}

\subsection{Initial value problem}
Now consider the Cauchy problem for a system of dispersive nonlinear stochastic PDE,
\begin{equation}\label{General-PDE-Ito}
    -id \mathbf u(t)
    +
    \mathbf h(D) \mathbf u(t) \, dt
    =
    \left[ \mathbf N(\mathbf u(t)) + \mathbf L(\mathbf u(t)) \right] \, dt
    +
    \mathbf M(\mathbf u(t)) \, dW(t)
    , \qquad
    \mathbf u(0) = \mathbf u_0,
\end{equation}
where the unknown is a random variable $\mathbf u(t)$ taking values in $\mathbf H^{\mathbf s}$ for a given $\mathbf s \in \R^n$,
\begin{equation}\label{General-PDE-data}
  \mathbf u_0 \colon \Omega \to \mathbf H^{\mathbf s} \quad \text{is $\mathcal F_0$-measurable},
\end{equation}
and the operators
\[
  \mathbf M(\mathbf f) = \begin{pmatrix} M_1(\mathbf f) \\ \vdots \\ M_n(\mathbf f) \end{pmatrix},
  \quad
  \mathbf L(\mathbf f) = \begin{pmatrix} L_1(\mathbf f) \\ \vdots \\ L_n(\mathbf f) \end{pmatrix}
  \quad \text{and} \quad
  \mathbf N(\mathbf f) = \begin{pmatrix} N_1(\mathbf f) \\ \vdots \\ N_n(\mathbf f) \end{pmatrix},
\]
acting only in the space variable $x$, are assumed to have the following properties:
\begin{equation}\label{M1}
  \mathbf M \colon \mathbf H^{\mathbf s} \to \mathcal L_2( K, \mathbf H^{\mathbf s})
  \quad \text{is linear, with a bound}
  \quad
  \norm{\mathbf M(\mathbf f)}_{\mathcal L_2( K, \mathbf H^{\mathbf s})}
  \le C \norm{\mathbf f}_{ \mathbf H^{\mathbf s} }.
\end{equation}
Further,
\begin{equation}\label{N0}
  \mathbf N(\mathbf 0) = \mathbf 0
\end{equation}
and
\begin{equation}\label{N1}
  \mathbf N \colon \mathbf X^{\mathbf s,b} \to \mathbf X^{\mathbf s,b'}
  \quad \text{is locally Lipschitz, for some $  -\frac12 < b' < 0 < b < \frac12$},
\end{equation}
with the bound, for some constants $p \in \N$ and $C > 0$,
\begin{equation}\label{N2}
  \norm{\mathbf N(\mathbf u) - \mathbf N(\mathbf v)}_{\mathbf X^{\mathbf s, b'}}
  \le
  C \left( 1 + \norm{\mathbf u}_{\mathbf X^{\mathbf s, b}}
  +
  \norm{\mathbf v}_{\mathbf X^{\mathbf s, b}} \right)^{p-1}
  \norm{\mathbf u - \mathbf v}_{\mathbf X^{\mathbf s, b}}.
\end{equation}
These estimates of course imply the corresponding ones with time restriction to any slab $(S,T) \times \R^d$. Finally, we assume that with the same $b,b'$ as above,
\begin{equation}\label{L1}
  \mathbf L \colon \mathbf X^{\mathbf s,b} \to \mathbf X^{\mathbf s,b'}
  \quad \text{is linear, with a bound} \quad
  \norm{\mathbf L(\mathbf u)}_{\mathbf X^{\mathbf s, b'}}
  \le
  C
  \norm{\mathbf u}_{\mathbf X^{\mathbf s, b}}.
\end{equation}

We emphasise that for the examples we have in mind, $\mathbf N$ may fail to map $\mathbf H^{\mathbf s}$ into itself, hence the deterministic integral in \eqref{PDEsolution} below may not make sense as a Bochner integral in $\mathbf H^{\mathbf s}$. However, one expects that this obstruction disappears at sufficiently high regularity. We therefore add the assumption
\begin{equation}\label{N3}
  \text{there exists $\mathbf s' \in \R^n$, with $s_i' \ge s_i$, such that
  $\mathbf N$ and $\mathbf L$ map $\mathbf H^{\mathbf s'}$ continuously into $\mathbf H^{\mathbf s'}$}.
\end{equation}
This will be used to establish measurability properties,
and to regularise \eqref{General-PDE-Ito}.

\begin{remark}
The reason for separating the linear part $\mathbf L$ from the nonlinear part $\mathbf N$ is that this allows us to avoid truncating the linear terms; see \eqref{PDEsolution-truncated}. This is not essential for the arguments used to prove existence and uniqueness in this section, but is used in the proof of conservation of charge in Section \ref{Charge_conservation}.
\end{remark}

\begin{remark}
It is easy to construct operators satisfying \eqref{M1}.
Taking $K = L^2\left(\R^d,\R \right)$, we consider
\[
  M_i(\mathbf f) =  \sum_{j=1}^n \japanese{D}^{-\sigma_{i,j}} f_j \mathfrak K_{i,j},
\]
where $\sigma_{i,j}$ are real numbers and $\mathfrak K_{i,j}$ are of the form \eqref{conv-op} with kernels $\mathfrak k_{i,j}$. Then by Corollary \ref{Hilbert_Schmidt_operator_corollary},
\[
  \norm{M_i(\mathbf f)}_{\mathcal L_2(L^2,H^{s_i})}
  \le
  \sum_{j=1}^n \norm{f_j \mathfrak K_{i,j}}_{\mathcal L_2(L^2,H^{s_i-\sigma_{i,j}})}
  \le
  C
  \sum_{j=1}^n \norm{f_j}_{H^{s_j}} \norm{\mathfrak k_{i,j}}_{H^{\abs{s_j}}}
\]
provided that $s_i - \sigma_{i,j} \le s_j$
and $\mathfrak k_{i,j} \in H^{\abs{s_j}}(\R^d,\R)$ for all $1 \le i,j \le n$.
\end{remark}

Let us now define precisely what we mean by a solution of \eqref{General-PDE-Ito}.

\begin{definition}\label{Solution-def}
Let $\tau \colon \Omega \to (0,\infty]$ be a stopping time. By a \emph{solution of \eqref{General-PDE-Ito}, \eqref{General-PDE-data} up to time $\tau$}, we mean a random variable
\[
  \mathbf u(t,\omega) \in \mathbf H^{\mathbf s} \quad \text{for $0 \le t < \tau(\omega)$}
\]
such that
\begin{equation}\label{PDEsolution1}
  \mathbf u(t) \colon \{ t < \tau \} \to \mathbf H^{\mathbf s} \quad \text{is $\mathcal F_t$-measurable},
\end{equation}
and such that, almost surely,
\begin{equation}\label{PDEsolution2}
  \mathbf u \in C\left( [0,t],\mathbf H^{\mathbf s} \right) \cap \mathbf X^{\mathbf s,b}(0,t)
  \quad \text{for $0 < t < \tau$}
\end{equation}
and
\begin{equation}\label{PDEsolution}
  \mathbf u(t) = \mathbf S(t) \mathbf u_0 + i \int_0^t \mathbf S(t-s) \left[ \mathbf N(\mathbf u(s)) + \mathbf L(\mathbf u(s))  \right] \, ds
  \\
  + i \int_0^t \mathbf S(t-s) \mathbf M(\mathbf u(s)) \, dW(s)
\end{equation}
for $0 \le t < \tau$. So in particular, $\mathbf u(0) = \mathbf u_0$ almost surely.
\end{definition}

Some remarks are in order.
First, by the assumptions on $\mathbf N$, the first integral in \eqref{PDEsolution} is well defined pointwise in $\omega$, and belongs to $C\left( [0,T],\mathbf H^{\mathbf s} \right)$ and $\mathbf X^{\mathbf s,b}(0,T)$ for any $0 < T < \tau(\omega)$, as we show in Section \ref{N-properties}. We emphasise, however, that it may not make sense as an $\mathbf H^{\mathbf s}$-valued Bochner integral, but only when interpreted in the Bourgain space, where dispersive effects are taken into account.

Second, to see that the stochastic integral exists, define
\begin{equation}\label{f-stop}
  f(t) = \norm{\mathbf u}_{\widetilde{\mathbf X}^{\mathbf s,b}(0,t)}^2
  \quad \text{for $0 < t < \tau$}, \qquad f(0) = 0.
\end{equation}
By Lemmas \ref{X-cont-lemma} and \ref{X-adapted-lemma}, proved in Section \ref{Mod-B-norm}, $f$ satisfies the hypotheses of Lemma \ref{Stopping-lemma-3}, hence
\begin{equation}\label{R-stop}
  \tau_R(\omega) = \sup \left\{ t \in [0,\tau(\omega)) \colon \text{$f(s,\omega) < R$ for $0 \le s \le t$} \right\}
\end{equation}
is a stopping time for each $R > 0$, and $\lim_{R \to \infty} \tau_R(\omega) = \tau(\omega)$.
Now consider the process
\[
  \mathbb 1_{t \le \tau_R(\omega)} \mathbf u(t,\omega),
\]
that is, the trivial extension beyond the time $\tau_R$. It is progressively measurable by Lemma \ref{Stopping-lemma-2}, so by \eqref{X-emb1} it belongs to $L^2 \left([0,T] \times \Omega, \mathbf H^{\mathbf s} \right)$ for all $T > 0$, and is $\mathbf H^{\mathbf s}$-adapted, hence
\[
  \int_0^t \mathbb 1_{s \le \tau_R}(s)\mathbf S(t-s) \mathbf M\left( \mathbf u(s) \right) \, dW(s)
\]
exists in $L^2(\Omega,\mathbf H^{\mathbf s})$ for all $t \ge 0$, by the assumption \eqref{M1}. 
Thus the It\^o integral appearing in \eqref{PDEsolution} is well defined
by the localisation procedure.

\subsection{Existence and uniqueness}\label{Abstract-WP-results}

We now formulate the main existence and uniqueness results that will be proved.

\begin{theorem}[Maximal local existence]\label{Existence-Theorem}
Let $\mathbf s \in \R^n$ and $-1/2 < b' < 0 < b < 1/2$. Assume that \eqref{M1}--\eqref{N3} hold and that $\mathbf u_0 \in L^2(\Omega,\mathbf H^{\mathbf s})$ is $\mathcal F_0$-measurable.
Then the problem \eqref{General-PDE-Ito}, \eqref{General-PDE-data} has a solution $\mathbf u(t)$ in the sense of Definition \ref{Solution-def}, with a stopping time $\tau \colon \Omega \to (0,\infty]$. The solution is maximal in the sense that, almost surely,
\begin{equation}\label{blowup}
  \tau < \infty \implies \limsup_{t \nearrow \tau} \norm{\mathbf u}_{\mathbf X^{\mathbf s,b}(0,t)}
  = \infty.
\end{equation}
\end{theorem}

The proof is given at the end of this subsection.

\begin{theorem}[Uniqueness]\label{Uniqueness-theorem}
Let $\mathbf s \in \R^n$ and $-1/2 < b' < 0 < b < 1/2$. Assume that \eqref{M1}--\eqref{N3} hold and that $\mathbf u_0 \in L^2(\Omega,\mathbf H^{\mathbf s})$ is $\mathcal F_0$-measurable.
Suppose $\mathbf u$ and $\mathbf v$ are solutions of \eqref{General-PDE-Ito}, \eqref{General-PDE-data}, as in Definition \ref{Solution-def}, up to stopping times $\tau$ and $\tau'$, respectively, and with the same initial datum $\mathbf u_0$. 
Then almost surely
\[
    \mathbf u(t) = \mathbf v(t) \quad \text{for $0 \le t < \min(\tau,\tau')$}.
\]
\end{theorem}

This is proved in Section \ref{Proof-Uniqueness}. For the proof, we must be able to compare the two solutions on slabs $[0,T] \times \Omega$. For this reason, we require also the following extension theorem.

\begin{theorem}[Extension]\label{Extension-theorem}
Let $\mathbf s \in \R^n$ and $-1/2 < b' < 0 < b < 1/2$. Assume that \eqref{M1}--\eqref{N3} hold and that $\mathbf u_0 \in L^2(\Omega,\mathbf H^{\mathbf s})$ is $\mathcal F_0$-measurable.
Suppose $\mathbf u$ is a solution of \eqref{General-PDE-Ito}, \eqref{General-PDE-data}, as in Definition \ref{Solution-def}, with stopping time $\tau$. Define the conditional stopping times $\tau_R$ as in \eqref{f-stop}, \eqref{R-stop}. Then for any $R > 0$ and $T > 0$ the equation
\[
  \mathbf U(t) = \mathbf S(t) \mathbf u_0 + i \int_0^{t \wedge \tau_R} \mathbf S(t-s) \left[ \mathbf N(\mathbf u(s)) + \mathbf L(\mathbf u(s)) \right] \, ds
  \\
  + i \int_0^t \mathbf S(t-s) \mathbf M(\mathbf U(s)) \, dW(s)
\]
has a unique solution $\mathbf U \in Z^{\mathbf s, b}(0,T)$ such that, almost surely, $\mathbf U(t) = \mathbf u(t)$ for $0 \le t \le \min( T , \tau_R)$.
\end{theorem}

This is proved in Section \ref{Proof-Extension}.

Now let us return to the existence result, Theorem \ref{Existence-Theorem}. To prove it, we consider a truncated version of \eqref{PDEsolution}, depending on a parameter $R > 0$,
\begin{multline}\label{PDEsolution-truncated}
  \mathbf u(t) = \mathbf S(t) \mathbf u_0 + i \int_0^t \mathbf S(t-s) \mathbf N\left( \Theta_R^{\mathbf u}(s)\mathbf u(s) \right) \, ds
  + i \int_0^t \mathbf S(t-s) \mathbf L\left( \mathbf u(s) \right) \, ds
  \\
  + i \int_0^t \mathbf S(t-s) \mathbf M\left( \mathbf u(s) \right) \, dW(s),
\end{multline}
where we use the notation
\begin{equation}\label{vector-cutoff}
  \Theta_R^{\mathbf u}(t)
  =
  \theta_R\left( \sum_{i=1}^n \norm{u_i}_{\widetilde X^{s_i,b}_{h_i(\xi)}(0,t)}^2\right).
\end{equation}
Here $\theta \colon \R \to \R$ is a smooth and compactly supported function with $\theta(x) = 1$ for $\abs{x} \le 1$, and we write $\theta_R(x) = \theta(x/R)$. Inside the cutoffs we use the norm \eqref{Bourgain_Slobodeckij_norm}.

We shall prove the following global result for the truncated problem.

\begin{theorem}[Global existence with truncation]\label{Existence-Theorem-Truncated}
Let $R > 0$, $\mathbf s \in \R^n$ and $-1/2 < b' < 0 < b < 1/2$. Assume that \eqref{M1}--\eqref{N3} are satisfied. Assume that $\mathbf u_0 \in L^2(\Omega,\mathbf H^{\mathbf s})$ is $\mathcal F_0$-measurable.
Then the truncated problem \eqref{PDEsolution-truncated} has a unique global solution $\mathbf u^R$ such that $\mathbf u^R \in Z^{\mathbf s, b}(0,T)$ for each $T > 0$. Moreover, for each $T > 0$ we have
\begin{equation}\label{Bound-truncated-solution}
  \norm{\mathbf u^R}_{Z^{\mathbf s,b}(0,T)} \le C_{T,R,b} \norm{\mathbf u_0}_{L^2(\Omega,\mathbf H^s)},
\end{equation}
and if $\mathbf U^R \in Z^{\mathbf s, b}(0,T)$ is the solution with $\mathcal F_0$-measurable data $\mathbf 
U_0 \in L^2(\Omega,\mathbf H^{\mathbf s})$, then
\begin{equation}\label{Bound-truncated-difference}
  \norm{\mathbf u^R-\mathbf U^R}_{Z^{\mathbf s,b}(0,T)} \le C_{T,R,b} \norm{\mathbf u_0-\mathbf U_0}_{L^2(\Omega,\mathbf H^s)}.
\end{equation}
\end{theorem}

Granting this last result for the moment, we can prove the local result, Theorem \ref{Existence-Theorem}. Define
\[
  f_R(t) = \norm{\mathbf u^R}_{\widetilde{\mathbf X}^{\mathbf s,b}(0,t)}^2
  \quad \text{for $t > 0$}, \qquad f_R(0) = 0.
\]
By Lemmas \ref{X-cont-lemma} and \ref{X-adapted-lemma}, proved in Section \ref{Mod-B-norm}, this function satisfies the hypotheses of Lemma \ref{Stopping-lemma-3}, hence
\[
  \tau_R(\omega) = \sup \left\{ t \in [0,\infty) \colon \text{$f_R(s,\omega) < R$ for $0 \le s \le t$} \right\}.
\]
is a stopping time. Up to this time, $\mathbf u^{R}$ is a solution of the non-truncated problem \eqref{PDEsolution}, and we let $R \to \infty$ to get a maximal solution. To this end, we use the following.

\begin{lemma}\label{TruncationLemma}
Let $\mathbf u^R$ be as in Theorem \ref{Existence-Theorem-Truncated}, and define the stopping time $\tau_R$ as above. Then
\begin{equation}\label{Truncation1}
  \mathbf u^{R}(t) = \mathbf u^{R'}(t) \quad \text{for $0 \le t \le \min(\tau_R,\tau_{R'})$}
\end{equation}
and
\begin{equation}\label{Truncation2}
  R \le R' \implies \tau_R \le \tau_{R'}.
\end{equation}
Moreover,
\begin{equation}\label{Truncation3}
  \text{$R < R'$ and $\tau_{R'} < \infty$} \implies \tau_R < \tau_{R'}.
\end{equation}
\end{lemma}

\begin{proof}
First, \eqref{Truncation1} follows from Theorem \ref{Uniqueness-theorem} (proved in Section \ref{Proof-Uniqueness}, and independently of Theorem \ref{Existence-Theorem-Truncated} and the present lemma).

Now let us prove \eqref{Truncation3} (which of course implies \eqref{Truncation2}). Suppose that $R < R'$ and $\tau_{R'} < \infty$. To get a contradiction, assume that $\tau_{R'} \le \tau_R$. Then by \eqref{Truncation1}, $f_R(t)=f_{R'}(t)$ for $0 < t \le \tau_{R'}$. But by \eqref{StoppingTime3}, $f_{R'}(\tau_{R'}) = R'$. Thus $f_R(\tau_{R'}) = R' > R$, contradicting \eqref{StoppingTime2}. Hence we must have $\tau_R < \tau_{R'}$.
\end{proof}

In view of the last lemma, setting
\[
  \tau = \sup_{R} \tau_R,
\]
which is a stopping time, we can consistently define $\mathbf u(t)$ for $t \in [0,\tau)$ by setting $\mathbf u(t) = \mathbf u^{R}(t)$ for $t \in [0,\tau_R]$. By \eqref{StoppingTime3} we have
\[
  \tau_R < \infty \implies \norm{\mathbf u^R}_{\widetilde{\mathbf X}^{\mathbf s,b}(0,\tau_R)}^2 = R,
\]
and by the estimate \eqref{Large-T-estimate} in Lemma \ref{NormLemma}, this implies
\begin{equation}\label{Truncated-growth}
  \tau_R < \infty \implies \norm{\mathbf u^R}_{\mathbf X^{\mathbf s,b}(0,\tau_R)}^2 \ge C R,
\end{equation}
where $C > 0$ depends only on $b$. Then \eqref{blowup} follows. Thus we have shown that Theorem \ref{Existence-Theorem} is a consequence of Theorem \ref{Existence-Theorem-Truncated} (and of Theorem \ref{Uniqueness-theorem}, which is used to prove the above lemma).

The remainder of this section is devoted to the proof of Theorems \ref{Existence-Theorem-Truncated}, \ref{Extension-theorem} and \ref{Uniqueness-theorem}, in that order. We also prove a regularisation result for the truncated system, in Section \ref{Regularisation}. Finally, in Section \ref{ConcreteCase} we show how the existence and uniqueness parts of Theorem \ref{local_th} follow from the abstract results.

In preparation for the proofs, we discuss in the next two subsections some key consequences of the assumptions made on the operators $\mathbf M$ and $\mathbf N$.

\subsection{Properties of $\mathbf M$}\label{M-properties}

Assume that $\mathbf u \in L^2 \left( [0,T] \times \Omega, \mathbf H^{\mathbf s} \right)$ is $\mathbf H^{\mathbf s}$-adapted. Then by \eqref{M1},
\begin{equation}\label{Hilbert-Schmidt-integral-bound}
  \mathbb E \left(   \int_0^t \norm{\mathbf S(t-s) \mathbf M(\mathbf u(s))}_{\mathcal L_2( K, \mathbf H^{\mathbf s})}^2 \, ds \right)
  =
  \mathbb E \left(   \int_0^t \norm{ \mathbf M(\mathbf u(s))}_{\mathcal L_2( K, \mathbf H^{\mathbf s})}^2 \, ds \right)
  \le C
  \mathbb E \left( \int_0^t \norm{\mathbf u(s)}_{ \mathbf H^{\mathbf s} }^2 \, ds \right),
\end{equation}
so for $0 \le t \le T$ the It\^o integral
\[
  \int_0^t \mathbf S(t-s) \mathbf M(\mathbf u(s)) \, dW(s)
\]
is well defined in $L^2\left( \Omega,\mathbf H^{\mathbf s} \right)$, is $\mathbf H^{\mathbf s}$-adapted and pathwise continuous, and by the It\^o isometry,
\begin{equation}\label{Ito-bound}
  \mathbb E \left(  \norm{ \int_0^t \mathbf S(t-s) \mathbf M(\mathbf u(s)) \, dW(s) }_{\mathbf H^{\mathbf s}}^2 \right)
  =
  \mathbb E \left(   \int_0^t \norm{\mathbf M(\mathbf u(s))}_{\mathcal L_2( K, \mathbf H^{\mathbf s})}^2 \, ds \right).
\end{equation}
By the maximal inequality \eqref{Maximal},
\begin{equation}\label{Maximal-bound}
  \mathbb E \left( \sup_{0 \le t \le T} \norm{ \int_0^t \mathbf S(t-s) \mathbf M(\mathbf u(s)) \, dW(s) }_{\mathbf H^{\mathbf s}}^2 \right)
  \le
  4\mathbb E \left(   \int_0^T \norm{\mathbf M(\mathbf u(s))}_{\mathcal L_2( K, \mathbf H^{\mathbf s})}^2 \, ds \right).
\end{equation}

Moreover, the stochastic integral belongs to $L^2\left( \Omega, \mathbf X^{\mathbf s, b}(0,T) \right)$, as we now show.

\begin{lemma}\label{X-lemma} Let $T > 0$. Assume that $\mathbf M$ satisfies \eqref{M1}, and that $0 \le b < 1/2$. Then
\begin{equation}\label{Ito-X}
  \mathbb E \left( \norm{ \int_0^t \mathbf S(t-s) \mathbf M(\mathbf u(s)) \, dW(s) }_{\mathbf X^{\mathbf s, b}(0,T)}^2 \right)
  \le
  C
  \mathbb E \left(   \int_0^T \norm{\mathbf M(\mathbf u(s))}_{\mathcal L_2( K, \mathbf H^{\mathbf s})}^2 \, ds \right)
\end{equation}
for all $\mathbf H^{\mathbf s}$-adapted $\mathbf u \in L^2 \left( [0,T] \times \Omega, \mathbf H^{\mathbf s} \right)$. Here the constant $C$ depends on $b$, but not on $T$.
\end{lemma}

\begin{proof}
Set
\[
  \mathbf I(t) =  \mathbf  \Lambda^{\mathbf s} \int_0^t \mathbf S(-s) \mathbf M(\mathbf u(s)) \, dW(s)
  = \int_0^t \mathbf \Lambda^{\mathbf s} \mathbf S(-s) \mathbf M(\mathbf u(s)) \, dW(s),
\]
where the operator
\[
  \mathbf f \mapsto \mathbf \Lambda^{\mathbf s} \mathbf f = \begin{pmatrix} \japanese{D}^{s_1} f_1 \\ \vdots \\ \japanese{D}^{s_n} f_n \end{pmatrix}
\]
is an isometry of $\mathbf H^{\mathbf s}$ onto $L^2_x = L^2 \left(\R^d,\C^n\right)$. Thus $\mathbf I(t) \in L^2\left(\Omega,L^2_x \right)$
and the left side of \eqref{Ito-X} equals
\begin{equation}\label{Ito-X-1}
  \mathbb E \left( \int_{\R^d} \norm{ \widehat{\mathbf I}(t,\xi) }_{\mathbf H^{b}(0,T)}^2 \, d\xi \right),
\end{equation}
where
\[
  \widehat{\mathbf I}(t) = \mathcal F \mathbf I(t) =
  \int_0^t \mathcal F  \mathbf\Lambda^{\mathbf s} \mathbf S(-s) \mathbf M(\mathbf u(s)) \, dW(s)
  \in L^2\left(\Omega,L^2_\xi \right)
\]
and $\mathcal F \colon L^2_x \to L^2_\xi$ is the Fourier transform in $x$. Using repeatedly \eqref{HScomposition}, we see that
\begin{multline}\label{Ito-X-2}
  \norm{\mathcal F  \mathbf\Lambda^{\mathbf s} \mathbf S(-s) \mathbf M(\mathbf u(s))}_{\mathcal L_2 \left(K,L^2_\xi \right)}
  \le
  (2\pi)^{d/2} \norm{\mathbf\Lambda^{\mathbf s} \mathbf S(-s) \mathbf M(\mathbf u(s))}_{\mathcal L_2(K,L^2_x)}
  \\
  \le
  \norm{\mathbf S(-s) \mathbf M(\mathbf u(s))}_{\mathcal L_2(K,\mathbf H^{\mathbf s})}
  =
  \norm{\mathbf M(\mathbf u(s))}_{\mathcal L_2(K,\mathbf H^{\mathbf s})}
  .
\end{multline}
Thus the It\^o isometry gives, for $0 \le t \le T$,
\[
  \mathbb E \left( \norm{\widehat{\mathbf I}(t,\xi)}_{L^2_\xi}^2 \right)
  \le
  C \mathbb E \left( \int_0^t \norm{\mathbf M(\mathbf u(s))}_{\mathcal L_2(K,\mathbf H^{\mathbf s})}^2 \, ds \right).
\]
Integrating this over $0 \le t \le T$ and using Tonelli's theorem gives
\[
  \int_\Omega \int_{\R^d} \int_0^T \Abs{\widehat{\mathbf I}(t,\xi,\omega)}^2 \, dt \, d\xi \, d\mathbb P(\omega)
  \le
  C T\mathbb E \left( \int_0^T \norm{\mathbf M(\mathbf u(s))}_{\mathcal L_2(K,\mathbf H^{\mathbf s})}^2 \, ds \right),
\]
implying that
\[
  \int_0^T \Abs{\widehat{\mathbf I}(t,\xi,\omega)}^2 \, dt < \infty
\]
for a.e.~$(\xi,\omega)$. So extending $\widehat{\mathbf I}(t)$ by zero outside the interval $(0,T)$, its Fourier transform with respect to $t$ is well defined:
\[
  \widetilde{\mathbf I}(\tau,\xi,\omega)
  =
  \int_0^T e^{-it\tau} \, \widehat{\mathbf I}(t,\xi,\omega) \, dt,
\]
and
\begin{equation}\label{Ito-X-3}
  \norm{ \widehat{\mathbf I}(t,\xi) }_{H^{b}(0,T)}^2
  \le
  \sum_{j=1}^n \int_\R \japanese{\tau}^{2b} \Abs{\widetilde{I}_j(\tau,\xi)}^2 \, d\tau.
\end{equation}
Now we calculate
\begin{multline*}
  \widetilde{\mathbf I}(\tau,\xi)
  =
  \int_0^T e^{-it\tau} \, \widehat{\mathbf I}(t,\xi) \, dt
  =
  \int_0^T e^{-it\tau} \left( \int_0^t \mathcal F  \mathbf\Lambda^{\mathbf s} \mathbf S(-s) \mathbf M(\mathbf u(s)) \, dW(s) \right) \, dt
  \\
  =
  \int_0^T \left( \int_s^T e^{-it\tau} \, dt \right)
  \mathcal F  \mathbf\Lambda^{\mathbf s} \mathbf S(-s) \mathbf M(\mathbf u(s)) \, dW(s),
\end{multline*}
where we used the stochastic Fubini's theorem (see \cite{Gawarecki_Mandrekar}). This is justified on account of the bound \eqref{Ito-X-2}. Combining that bound with
\[
  \Abs{\int_s^T e^{-it\tau} \, dt} \le C\japanese{\tau}^{-1},
\]
where $C$ is independent of $s$ and $T$, we obtain
\[
  \norm{ \left( \int_s^T e^{-it\tau} \, dt \right)
  \mathcal F  \mathbf\Lambda^{\mathbf s} \mathbf S(-s) \mathbf M(\mathbf u(s))}_{\mathcal L_2(K,L^2_\xi)}
  \le C \japanese{\tau}^{-1}
  \norm{\mathbf M(\mathbf u(s))}_{\mathcal L_2(K,\mathbf H^{\mathbf s})},
\]
hence by It\^o's isometry
\[
  \mathbb E \left( \norm{\widetilde{I}_j(\tau,\xi)}_{L^2_\xi}^2 \right)
  \le
  C \japanese{\tau}^{-2} \mathbb E \left( \int_0^T \norm{\mathbf M(\mathbf u(s))}_{\mathcal L_2(K,\mathbf H^{\mathbf s})}^2 \, ds \right)
\]
for $1 \le j \le n$. Multiplying both sides by $\japanese{\tau}^{2b}$, integrating in $\tau$, and using Tonelli's theorem and \eqref{Ito-X-1} and \eqref{Ito-X-3}, we conclude that the left side of \eqref{Ito-X} equals
\begin{align*}
  \mathbb E \left( \int_{\R^d} \norm{ \widehat{\mathbf I}(t,\xi) }_{\mathbf H^{b}(0,T)}^2 \, d\xi \right)
  &\le
  C \sum_{j=1}^n \int_\R \japanese{\tau}^{2b} \mathbb E \left( \int_{\R^d} \Abs{\widetilde{I}_j(\tau,\xi)}^2 \, d\xi \right) \, d\tau 
  \\
  &\le
  C \left( \sum_{j=1}^n \int_\R \japanese{\tau}^{2b-2} \, d\tau \right)
  \mathbb E \left( \int_0^T \norm{\mathbf M(\mathbf u(s))}_{\mathcal L_2(K,\mathbf H^{\mathbf s})}^2 \, ds \right),
\end{align*}
completing the proof of the lemma.
\end{proof}

Combining \eqref{Maximal-bound}, the last lemma and the embedding \eqref{Bourgain-3}, we obtain the following key fact.

\begin{corollary}\label{X-cor} Let $0 \le S < T \le S+1$. Assume that $\mathbf M$ satisfies \eqref{M1}, and that $0 \le b < 1/2$. Then we have the bound
\begin{equation}\label{X-cor-bound}
  \norm{ \int_S^t \mathbf S(t-s) \mathbf M(\mathbf u(s)) \, dW(s) }_{Z^{\mathbf s,b}(S,T)}
  \le
  C (T-S)^{b}
  \norm{ \mathbf u }_{L^2\left( \Omega, \mathbf X^{\mathbf s,b}(S,T) \right)}
\end{equation}
for all $\mathbf u \in Z^{\mathbf s,b}(S,T)$,
where the constant $C$ depends on $b$, but not on $T$ or $S$.
\end{corollary}

\begin{proof}
Extend $\mathbf u$ by zero outside $S < t < T$. Then $\mathbf u$ belongs to $L^2 \left( [0,T] \times \Omega, \mathbf H^{\mathbf s} \right)$ (see Remark \ref{Z-remark}), and applying \eqref{Maximal-bound}, \eqref{Hilbert-Schmidt-integral-bound} and Lemma \ref{X-lemma} we get
\[
  \norm{ \int_S^t \mathbf S(t-s) \mathbf M(\mathbf u(s)) \, dW(s) }_{Z^{\mathbf s,b}(S,T)}
  \le
  C
  \norm{ \mathbf u }_{ L^2\left( [S,T] \times \Omega,\mathbf H^{\mathbf s} \right)}.
\]
Applying now  \eqref{Bourgain-3}, we obtain \eqref{X-cor-bound}.
\end{proof}

\subsection{Properties of $\mathbf N$ and $\mathbf L$}\label{N-properties}

Recalling that $-1/2 < b' < 0$, choose $0 < \varepsilon < b'+1/2$ and set
\[
  B := b'+1-\varepsilon > \frac12.
\]
Let $0 \le S < T \le S+1$ and assume that $\mathbf u,\mathbf v \in \mathbf X^{\mathbf s,b}(S,T)$. Applying \eqref{Bourgain-4}, \eqref{Bourgain-2b}, \eqref{Bourgain-3} and the assumption \eqref{N2} we get
\begin{multline}\label{N-dispersive-difference-bound}
  \sup_{S \le t \le T} \norm{\int_S^t \mathbf S(t-s) \left[ \mathbf N(\mathbf u(s)) - \mathbf N(\mathbf v(s)) \right] \, ds}_{\mathbf H^{\mathbf s}}
  \le C
  \norm{\int_S^t \mathbf S(t-s) \left[ \mathbf N(\mathbf u(s)) - \mathbf N(\mathbf v(s)) \right] \, ds}_{\mathbf X^{\mathbf s, B}(S,T)}
  \\
  \le
  C
  \norm{\mathbf N(\mathbf u) - \mathbf N(\mathbf v)}_{\mathbf X^{\mathbf s, B-1}(S,T)}
  \le
  C (T-S)^{\varepsilon} \norm{\mathbf N(\mathbf u) - \mathbf N(\mathbf v)}_{\mathbf X^{\mathbf s, b'}(S,T)}
  \\
  \le
  C (T-S)^{\varepsilon} \left( 1 + \norm{\mathbf u}_{\mathbf X^{\mathbf s,b}(S,T)}
  +
  \norm{\mathbf v}_{\mathbf X^{\mathbf s, b}(S,T)} \right)^{p-1}
  \norm{\mathbf u - \mathbf v}_{\mathbf X^{\mathbf s, b}(S,T)}
\end{multline}
and (taking $\mathbf v = \mathbf 0$ and using the assumption \eqref{N0})
\begin{equation}\label{N-continuity}
  t \mapsto \int_S^t \mathbf S(t-s) \mathbf N(\mathbf u(s)) \, ds
  \quad
  \text{belongs to $C([S,T],\mathbf H^{\mathbf s})$}.
\end{equation}
Note that in \eqref{N-dispersive-difference-bound} the constants $C$ depend only on $b$, $b'$ and $B$. Moreover, we claim that
\begin{equation}\label{N-adapted}
  \text{$\mathbf u \in \mathbf X^{\mathbf s,b}(S,T) \cap C([S,T],\mathbf H^{\mathbf s})$ is $\mathbf H^{\mathbf s}$-adapted}
  \implies
  \text{$t \mapsto \int_S^t \mathbf S(t-s) \mathbf N(\mathbf u(s)) \, ds$ is $\mathbf H^{\mathbf s}$-adapted}.
\end{equation}
To see this, let $\mathbf s'$ be as in \eqref{N3} and use mollification in the $x$-variable to obtain a sequence $\mathbf u_m$ such that
\begin{itemize}
\item $\mathbf u_m \in \mathbf X^{\mathbf s',b}(S,T) \cap C([S,T],\mathbf H^{\mathbf s'})$,
\item $\mathbf u_m$ is $\mathbf H^{\mathbf s'}$-adapted,
\item $\mathbf u_m \to \mathbf u$ in $\mathbf X^{\mathbf s,b}(S,T) \cap C([S,T],\mathbf H^{\mathbf s})$ as $m \to \infty$.
\end{itemize}
Then by the assumption \eqref{N3}, $\mathbf N(\mathbf u_m) \in C([S,T],\mathbf H^{\mathbf s'})$ is adapted, and therefore progressively measurable, hence the $\mathbf H^{\mathbf s'}$-valued integral
\[
  I_m(t) = \int_S^t \mathbf S(t-s) \mathbf N(\mathbf u_m(s)) \, ds
\]
exists and is adapted. Moreover, by \eqref{N-dispersive-difference-bound}, $I_m$ converges, as $m \to \infty$,  in $C([S,T],\mathbf H^{\mathbf s})$ to the integral appearing in \eqref{N-adapted}, thereby proving that the latter is adapted. Finally, we note that \eqref{N-dispersive-difference-bound}--\eqref{N-adapted} of course also hold for $\mathbf L$, but then with $p=1$ in \eqref{N-dispersive-difference-bound}.

\subsection{Existence for the truncated problem}\label{Section-truncated-existence}

We now prove Theorem \ref{Existence-Theorem-Truncated}. To simplify the notation, instead of $\mathbf u^R$ we simply write $\mathbf u$.

Note that \eqref{PDEsolution-truncated}, with the cutoff given by \eqref{vector-cutoff}, is equivalent to
\begin{multline}\label{PDEsolution-truncated-split}
  \mathbf u(t) = \mathbf S(t-S) \mathbf u(S) + i \int_S^t \mathbf S(t-s) \mathbf N\left( \Theta_R^{\mathbf u}(s)\mathbf u(s) \right) \, ds
  + i \int_S^t \mathbf S(t-s) \mathbf L\left( \mathbf u(s) \right) \, ds
  \\
  + i \int_S^t \mathbf S(t-s) \mathbf M\left( \mathbf u(s) \right) \, dW(s)
\end{multline}
for $0 \le S \le t \le T$. By Proposition \ref{MainCutoffLemma}, for $\mathbf u,\mathbf v \in Z^{\mathbf s,b}(0,T)$ we have, for $S \in [0,T]$,
\begin{equation}\label{Cutoff-fact1}
  \norm{\Theta_R^{\mathbf u}(t)\mathbf u(t)}_{\mathbf X^{\mathbf s,b}(0,S)}
  \le C \sqrt{R},
\end{equation}
\begin{equation}\label{Cutoff-fact2}
  \norm{\Theta_R^{\mathbf u}(t)\mathbf u(t) - \Theta_R^{\mathbf v}(t)\mathbf v(t)}_{\mathbf X^{\mathbf s,b}(0,S)}
  \le C \norm{\mathbf u - \mathbf v}_{\mathbf X^{\mathbf s,b}(0,S)}
\end{equation}
where the constant depends on $b$ and $T$. By Lemma \ref{X-adapted-lemma}, the cutoffs $\Theta_R^{\mathbf u}(t)$ and $\Theta_R^{\mathbf v}(t)$ are adapted.

Now fix a target time $T > 0$, and divide $[0,T]$ into $N$ subintervals of length $\delta=T/N$, where $N$ will be chosen large enough depending on $R$ and $T$. On each subinterval $[0,\delta], [\delta,2\delta],\dots$ we prove existence by a contraction argument in the $Z$-space.

Proceeding inductively, let us assume that for some $0 \le j < N$ we have proved existence up to time $S=j\delta$, so $\mathbf u \in Z^{\mathbf s, b}(0,S)$ (for $S=0$ this just means that $\mathbf u_0 \in L^2(\Omega,\mathbf H^{\mathbf s})$). Set $S' = S + \delta$. Then for $t \in [S,S']$ we must solve
\begin{multline}\label{induction1-truncated}
  \mathbf v(t) = \mathbf S(t-S) \mathbf u(S)
  + i \int_S^t \mathbf S(t-\sigma) \mathbf N\left(\Theta_R^{[\mathbf u,\mathbf v]}(\sigma)\mathbf v(\sigma)\right) \, d\sigma
  + i \int_S^t \mathbf S(t-\sigma) \mathbf L(\mathbf v(\sigma)) \, d\sigma
  \\
  + i \int_S^t \mathbf S(t-\sigma) \mathbf M(\mathbf v(\sigma)) \, dW(\sigma),
\end{multline}
where $[\mathbf u,\mathbf v]$ is defined as in Lemma \ref{Concatenation-lemma}. If we can show that \eqref{induction1-truncated} 
has a unique solution $\mathbf v \in Z^{\mathbf s, b}(S,S')$, then by Lemma \ref{Concatenation-lemma} we have $[\mathbf u,\mathbf v] \in Z^{\mathbf s, b}(0,S')$. Renaming the latter function $\mathbf u$, we have then extended the solution to $[0,S']$, and by induction this proves Theorem \ref{Existence-Theorem-Truncated}.

To solve \eqref{induction1-truncated} on $[S,S'] = [S,S+\delta]$, we set up a contraction argument in $Z^{\mathbf s, b}(S,S')$ for the operator
\[
  \mathfrak T(\mathbf v)(t)
  =
  \text{r.h.s.}\eqref{induction1-truncated}
  =: \mathfrak T_0(t) + \mathfrak T_1(\mathbf v)(t) + \mathfrak T_2(\mathbf v)(t) + \mathfrak T_3(\mathbf v)(t),
  \quad \text{for $S \le t \le S'$}.
\]
So now let $\mathbf v, \mathbf w \in Z^{\mathbf s, b}(S,S')$. We will prove that
\begin{equation}\label{Contraction-map-1}
  \mathfrak T(\mathbf v) \in Z^{\mathbf s,b}(S,S')
\end{equation}
and
\begin{equation}\label{Contraction-map-2}
  \norm{\mathfrak T(\mathbf v) - \mathfrak T(\mathbf w)}_{Z^{\mathbf s, b}(S,S')} \le \frac12 \norm{\mathbf v-\mathbf w}_{Z^{\mathbf s, b}(S,S')}
\end{equation}
provided that $S'-S = \delta > 0$ is taken sufficiently small, depending on $R$ and $T$. Thus $\mathfrak T$ is a contraction on $Z^{\mathbf s,b}(S,S')$, so it has a unique fixed point $\mathbf v$ in that space.

We now prove \eqref{Contraction-map-1} and \eqref{Contraction-map-2} for each of the terms constituting $\mathfrak T$.

\subsubsection{The term $\mathfrak T_0$}

By the induction hypothesis, $\mathbf u(S)$ belongs to $L^2\left(\Omega,\mathbf H^{\mathbf s}\right)$
and is $\mathcal F_{S}$-measurable, hence the same is true of $\mathfrak T_0(t) = \mathbf S(t-S) \mathbf u(S)$ for $t \ge S$. By \eqref{Bourgain-1b} and \eqref{Bourgain-4},
\[
  \sup_{t \in [S,S']} \norm{\mathbf S(t-S) \mathbf u(S)}_{\mathbf H^{\mathbf s}}
  \le C \norm{\mathbf S(t-S) \mathbf u(S)}_{\mathbf X^{\mathbf s,1}(S,S')}
  \le C \norm{\mathbf u(S)}_{\mathbf H^{\mathbf s}},
\]
implying $\mathfrak T_0 \in Z^{\mathbf s,1}(S,S')$. This verifies \eqref{Contraction-map-1} for the term $\mathfrak T_0$.

\subsubsection{The term $\mathfrak T_1$}

Applying \eqref{N-dispersive-difference-bound} on $[S,S']$ to the difference
\[
  \mathfrak T_1(\mathbf v)  - \mathfrak T_1(\mathbf w)
  =
  i \int_S^t \mathbf S(t-\sigma) 
  \left[
  \mathbf N\left( \Theta_R^{[\mathbf u,\mathbf v]}(\sigma)\mathbf v(\sigma) \right) 
  -
  \mathbf N\left( \Theta_R^{[\mathbf u,\mathbf w]}(\sigma)\mathbf w(\sigma) \right)
  \right]
  \, d\sigma
\]
yields
\begin{multline*}
  \sup_{t \in [S,S']} \norm{\mathfrak T_1(\mathbf v)(t)  - \mathfrak T_1(\mathbf w)(t)}_{ \mathbf H^{\mathbf s}}
  \le
  C\norm{\mathfrak T_1(\mathbf v)  - \mathfrak T_1(\mathbf w)}_{ \mathbf X^{\mathbf s,B}(S,S') }
  \\
  \le
  C (S'-S)^{\varepsilon}
  \left( 1 + \norm{\Theta_R^{[\mathbf u,\mathbf v]}\mathbf v}_{\mathbf X^{\mathbf s,b}(S,S')}
  +
  \norm{\Theta_R^{[\mathbf u,\mathbf w]}\mathbf w}_{\mathbf X^{\mathbf s,b}(S,S')} \right)^{p-1}
  \norm{\Theta_R^{[\mathbf u,\mathbf v]}\mathbf v -  \Theta_R^{[\mathbf u,\mathbf w]}\mathbf w}_{\mathbf X^{\mathbf s,b}(S,S')}.
\end{multline*}
But by \eqref{Cutoff-fact1},
\[
  \norm{\Theta_R^{[\mathbf u,\mathbf v]}\mathbf v}_{\mathbf X^{\mathbf s,b}(S,S')}
  \le
  \norm{\Theta_R^{[\mathbf u,\mathbf v]}[\mathbf u,\mathbf v]}_{\mathbf X^{\mathbf s,b}(0,S')}
  \le C\sqrt{R},
\]
where $C$ depends on $T$ and $b$. The same holds with $\mathbf w$ instead of $\mathbf v$. Similarly, \eqref{Cutoff-fact2} gives
\begin{multline}\label{Cutoff-fact3}
  \norm{\Theta_R^{[\mathbf u,\mathbf v]}\mathbf v - \Theta_R^{[\mathbf u,\mathbf w]}\mathbf w}_{\mathbf X^{\mathbf s,b}(S,S')}
  \le
  \norm{\Theta_R^{[\mathbf u,\mathbf v]}[\mathbf u,\mathbf v] - \Theta_R^{[\mathbf u,\mathbf w]}[\mathbf u,\mathbf w]}_{\mathbf X^{\mathbf s,b}(0,S')}
  \\
  \le
  C\norm{[\mathbf u,\mathbf v] - [\mathbf u,\mathbf w]}_{\mathbf X^{\mathbf s,b}(0,S')}
  \le
  C\norm{\mathbf v-\mathbf w}_{\mathbf X^{\mathbf s,b}(S,S')},
\end{multline}
where we used Lemma \ref{Concatenation-lemma} in the last step. Taking the $L^2(\Omega)$-norm we therefore obtain
\[
  \norm{\mathfrak T_1(\mathbf v)  - \mathfrak T_1(\mathbf w)}_{ Z^{\mathbf s,b}(S,S') }
  \le
  C (S'-S)^{\varepsilon}
  \left( 1+\sqrt{R} \right)^{p-1} \norm{\mathbf v-\mathbf w}_{L^2(\Omega,\mathbf X^{\mathbf s,b}(S,S'))},
\]
where $C$ depends on $T$ and $b$. Taking $\mathbf w=\mathbf 0$, the bounds above also imply that $\mathfrak T_1(\mathbf v)$ belongs to $Z^{\mathbf s,b}(S,S')$, by \eqref{N-continuity} and \eqref{N-adapted}.

\subsubsection{The term $\mathfrak T_2$} The arguments used for $\mathfrak T_1$ apply also here, but simplify since we take $p=1$ and there is no cutoff.  

\subsubsection{The term $\mathfrak T_3$}

By Remark \ref{Z-remark}, $\mathbf v \in L^2 \left( [S,S'] \times \Omega , \mathbf H^{\mathbf s} \right)$. Extending $\mathbf v$ by zero outside $[S,S']$, the considerations in Section \ref{M-properties}, and in particular \eqref{Maximal-bound}, \eqref{Hilbert-Schmidt-integral-bound} and Lemma \ref{X-lemma}, show that $\mathfrak T_3(\mathbf v)$ belongs to $Z^{\mathbf s,b}(S,S')$. 
Moreover, by Corollary \ref{X-cor} and the linearity of $\mathbf M$ we have
\[
  \norm{ \mathfrak T_3(\mathbf v) - \mathfrak T_2(\mathbf w) }_{Z^{\mathbf s,b}(S,S')}
  \le
  C (S'-S)^{b}
  \norm{ \mathbf v - \mathbf w }_{L^2\left( \Omega, \mathbf X^{\mathbf s, b}(S,S') \right)},
\]
which proves \eqref{Contraction-map-2} for the term $\mathfrak T_3$, if $\delta = S'-S$ is small enough. This concludes the proof of \eqref{Contraction-map-1} and \eqref{Contraction-map-2}.

\subsubsection{The bounds \eqref{Bound-truncated-solution} and \eqref{Bound-truncated-difference}} Taking $\mathbf w = \mathbf 0$, the above bounds show that the fixed point $\mathbf v$ satisfies
\[
  \norm{\mathbf v}_{Z^{\mathbf s,b}(S,S')}
  \le
  C
  \norm{\mathbf u(S)}_{L^2(\Omega,\mathbf H^{\mathbf s})},
\]
where $C$ is an absolute constant. By induction it follows that the solution $\mathbf u \in Z^{\mathbf s,b}(0,T)$ satisfies 
\[
  \norm{\mathbf u}_{Z^{\mathbf s,b}(0,T)}
  \le
  C^N
  \norm{\mathbf u_0}_{L^2(\Omega,\mathbf H^{\mathbf s})},
\]
where $N = T/\delta$ depends on $T$ and $R$. This proves \eqref{Bound-truncated-solution}, and the same argument gives \eqref{Bound-truncated-difference} (let the $\mathbf w$ above be the fixed point corresponding to the solution $\mathbf U$ with data $\mathbf U_0$).

This concludes the proof of Theorem \ref{Existence-Theorem-Truncated}.

\subsection{Extension}\label{Proof-Extension}

Here we prove Theorem \ref{Extension-theorem}.

Assume that $0 \le S < T$ and that we have found $\mathbf U$, with the desired properties, on $[0,S]$ (for $S=0$ this just means that $\mathbf u_0 \in L^2(\Omega,\mathbf H^{\mathbf s})$). Set $S'=S+\delta$, where $\delta > 0$ will be chosen sufficiently small, depending on $T$ and $R$. For $t \in [S,S']$ we must then solve
\begin{equation}\label{induction1}
  \mathbf V(t) = \mathbf S(t-S) \mathbf U(S) + i \int_{S \wedge \tau_R}^{t \wedge \tau_R} \mathbf S(t-\sigma) \left[ \mathbf N(\mathbf u(s)) + \mathbf L(\mathbf u(s)) \right] \, ds
  + i \int_{S}^t \mathbf S(t-s) \mathbf M(\mathbf V(s)) \, dW(s).
\end{equation}
The solution $\mathbf V$ should be in $Z^{\mathbf s,b}(S,S')$, it should satisfy, almost surely,
\begin{equation}\label{induction2}
  \mathbf V(t) = \mathbf u(t) \quad \text{for $S \le t \le S' \wedge \tau_R$},
\end{equation}
and it should be the only solution with these properties.

For $\mathbf V \in Z^{\mathbf s, b}(S,S')$ define
\[
  \mathbf \Phi(\mathbf V)(t)
  = \mathbf S(t-S) \mathbf U(S)
  + i \int_{S \wedge \tau_R}^{t \wedge \tau_R} \mathbf S(t-s) \left[ \mathbf N(\mathbf u(s)) + \mathbf L(\mathbf u(s)) \right] \, ds
  + i \int_{S}^t \mathbf S(t-s) \mathbf M([\mathbf u,\mathbf V](s)) \, dW(s),
\]
where
\[
  [\mathbf u,\mathbf V](t)
  =
  \begin{cases}
  \mathbf u(t) &\text{for $0 \le t \le S' \wedge \tau_R$}
  \\
  \mathbf V(t) &\text{for $S' \wedge \tau_R < t \le S'$}.
  \end{cases}
\]
Now observe that, almost surely,
\begin{equation}\label{Phi-1}
  S \le t \le S' \wedge \tau_R \implies \mathbf \Phi(\mathbf V)(t) = \mathbf u(t),
\end{equation}
since for such $t$ we have $[\mathbf u,\mathbf V](s) = \mathbf u(s)$ for $S \le s \le t$, and by \eqref{PDEsolution},
\[
  \mathbf u(t) = \mathbf S(t-S) \mathbf u(S) 
  + i \int_{S}^t \mathbf S(t-s) \left[ \mathbf N(\mathbf u(s)) + \mathbf L(\mathbf u(s)) \right] \, ds
  + i \int_{S}^t \mathbf S(t-s) \mathbf M(\mathbf u(s)) \, dW(s),
\]
which equals $\mathbf \Phi(\mathbf V)(t)$ since $S \le t \le \tau_R$ and $\mathbf u(S) = \mathbf U(S)$.

So if $\mathbf V$ is a fixed point of $\mathbf\Phi$, then by \eqref{Phi-1} we have, almost surely, $[\mathbf u,\mathbf V] = \mathbf V$ on $[S,S']$, hence $\mathbf V$ satisfies \eqref{induction1} and \eqref{induction2}. Conversely, if $\mathbf V$ satisfies \eqref{induction1} and \eqref{induction2}, it is clearly a fixed point. Thus it only remains to prove that $\mathbf \Phi$ has a unique fixed point in $Z^{\mathbf s,b}(S,S')$. But this follows as in the proof of Theorem \ref{Existence-Theorem-Truncated}, if $\delta = S'-S > 0$ is small enough. This concludes the proof of Theorem \ref{Extension-theorem}.

\subsection{Uniqueness}\label{Proof-Uniqueness}

Here we prove Theorem \ref{Uniqueness-theorem}.

Fix $R > 0$ and $T > 0$ and define the conditional stopping time $\tau_R$ as in \eqref{f-stop}, \eqref{R-stop}. 
Similarly define $\tau_R'$ for $\mathbf v$.
It is enough to prove that, almost surely, $\mathbf u(t) = \mathbf v(t)$ for $0 \le t \le \min(T,\mu)$,
where $\mu = \min(\tau_R,\tau_R')$.

Note that if $\mu_R^{\mathbf u}$ is the conditional stopping time
defined by the pair $(\mu, \mathbf u)$,
then
\(
    \mu_R^{\mathbf u} = \mu
    .
\)
Similarly,
\(
    \mu_R^{\mathbf v} = \mu
    .
\)
Therefore, by Theorem \ref{Extension-theorem}
there exist $\mathbf U, \mathbf V \in Z^{\mathbf s,b}(0,T)$ such that, almost surely, $\mathbf U(t) = \mathbf u(t)$ and $\mathbf V(t) = \mathbf v(t)$ for $0 \le t \le \min(T,\mu)$, and
\begin{align*}
  \mathbf U(t) &= \mathbf S(t) \mathbf u_0 + i \int_0^{t \wedge \mu} \mathbf S(t-s) \left[ \mathbf N(\mathbf U(s)) + \mathbf L(\mathbf U(s)) \right] \, ds
  + i \int_0^t \mathbf S(t-s) \mathbf M(\mathbf U(s)) \, dW(s),
  \\
  \mathbf V(t) &= \mathbf S(t) \mathbf u_0 + i \int_0^{t \wedge \mu} \mathbf S(t-s) \left[ \mathbf N(\mathbf V(s)) + \mathbf L(\mathbf V(s)) \right] \, ds
  + i \int_0^t \mathbf S(t-s) \mathbf M(\mathbf V(s)) \, dW(s).
\end{align*}
for $0 \le t \le T$.

Then it is enough to prove that, almost surely, $\mathbf U(t) = \mathbf V(t)$ for $0 \le t \le T$. We know this holds for $t=0$. As in the proof of Theorem \ref{Existence-Theorem-Truncated} we now cut $[0,T]$ into short intervals of length $\delta$ and proceed inductively. Assume that $0 \le S < T$ and that we have proved that, almost surely, $\mathbf U(t) = \mathbf V(t)$ for $t \in [0,S]$. Then we prove that this is true also on $[S,S']$, where $S'=S+\delta$.

To this end, write
\[
  \mathbf U(t) - \mathbf V(t)
  =
  \Delta_1(t) + \Delta_2(t) + \Delta_3(t)
  \quad \text{for $S \le t \le T$},
\]
where
\begin{align*}
  \Delta_1(t)
  &=
  i \int_{S \wedge \mu}^{t \wedge \mu} \mathbf S(t-s) \left[ \mathbf N(\mathbf U(s)) - \mathbf N(\mathbf V(s)) \right] \, ds,
  \\
  \Delta_2(t)
  &=
  i \int_{S \wedge \mu}^{t \wedge \mu} \mathbf S(t-s) \left[ \mathbf L(\mathbf U(s)) - \mathbf L(\mathbf V(s)) \right] \, ds,
  \\
  \Delta_3(t)
  &=
  i \int_S^t \mathbf S(t-s) \left[ \mathbf M(\mathbf U(s)) - \mathbf M(\mathbf V(s)) \right] \, dW(s).
\end{align*}
We are going to first estimate $\Delta_1(t)$ pointwise in $\omega$, and then take the $L^2$ norm with respect to $\omega$. So for the pointwise estimate we may restrict to $\omega$ at which $\mathbf U(t) = \mathbf u(t)$ and $\mathbf V(t) = \mathbf v(t)$ for $0 \le t \le \min(T,\mu)$. We may also assume $S \le \mu$, as otherwise the integral $\Delta_1(t)$ vanishes. Write
\[
  \Delta_1(t)
  =
  i \int_S^{t} \mathbf S(t-s) \left[ \mathbf N(\mathbb 1_{s \le \mu} \mathbf U(s)) - \mathbf N(\mathbb 1_{s \le \mu} \mathbf V(s)) \right] \, ds.
\]
Let $0 < \delta \le 1$. Observe that
\[
  \norm{\mathbb 1_{t \le \mu} \mathbf U}_{\mathbf X^{\mathbf s,b}(S,S')}
  \le
  C \norm{\mathbf u}_{\mathbf X^{\mathbf s,b}(0,\mu)}
  \le
  C \norm{\mathbf u}_{\widetilde{\mathbf X}^{\mathbf s,b}(0,\mu)}
  \le C\sqrt{R},
\]
since $\mu \le \tau_R$. Here $C$ depends on $T$ and $b$. The same holds for $\mathbf V$, since $\mu \le \tau_R'$. Thus by \eqref{N-dispersive-difference-bound} we get the bound, pointwise a.e.~in $\omega$,
\[
  \norm{\Delta_1}_{\mathbf X^{\mathbf s,b}(S,S')}
  \le
  C \left( 1+\sqrt{R} \right)^{p-1} \delta^{\varepsilon}
  \norm{\mathbf U - \mathbf V}_{\mathbf X^{\mathbf s,b}(S,S')},
\]
which we then square and integrate with respect to $\omega$. The same estimate holds for $\Delta_2$, but with $p=1$. Finally, we bound $\Delta_3$. By Corollary \ref{X-cor},
\[
  \mathbb E \left( \norm{\Delta_3}_{\mathbf X^{\mathbf s,b}(S,S')}^2 \right)
  \le
  C \delta^{2b}
  \mathbb E \left( \norm{\mathbf U - \mathbf V}_{\mathbf X^{\mathbf s,b}(S,S')}^2 \right).
\]
Combining the above bounds, we obtain
\[
  \mathbb E \left( \norm{\mathbf U - \mathbf V}_{\mathbf X^{\mathbf s,b}(S,S')}^2 \right)
  \le
  C (1+R)^{p-1} \delta^{2\min(\varepsilon,b)}
  \mathbb E \left( \norm{\mathbf U - \mathbf V}_{\mathbf X^{\mathbf s,b}(S,S')}^2 \right),
\]
where $C$ depends on $T$ and $b$. So for $\delta > 0$ small enough,
\[
  \mathbb E \left( \norm{\mathbf U - \mathbf V}_{\mathbf X^{\mathbf s,b}(S,S')}^2 \right) = 0,
\]
hence, almost surely, $\mathbf U(t)=\mathbf V(t)$ for $S \le t \le S'$. This concludes the proof of Theorem \ref{Uniqueness-theorem}.

\subsection{Regularisation and It\^o's formula}\label{Regularisation}

For $\mu \ge 1$ let $P_\mu$ be the Fourier multiplier with symbol $\theta_\mu(\xi)$, that is, for $f \in \mathcal S'(\R^d)$,
\[
  \widehat{P_\mu f}(\xi) = \theta\left( \frac{\xi}{\mu} \right) \widehat f(\xi).
\]
We assume that $\supp\theta \subset [-2,2]$,
hence $\widehat{P_\mu f}$ is supported in $[-2\mu,2\mu]$, so $P_\mu f(x)$ is smooth.
We also assume that $\abs{\theta} \le 1$.
Then $P_\mu$ is bounded, with operator norm $\le 1$, on $H^s$ and $X^{s,b}$ for any $s,b$. 
Moreover, $P_\mu$ maps $H^s$ into $H^N$ for arbitrarily large $N$.
Finally, we assume that $\theta = 1$ on $[-1, 1]$,
so that $P_\mu$ converges strongly to the identity operator
in $H^s$ and $X^{s,b}$ as $\mu \to \infty$.

Now consider the following frequency-truncated version of \eqref{PDEsolution-truncated}: 
\begin{multline}\label{PDEsolution-truncated-regularised}
  \mathbf u(t) = P_\mu\left( \mathbf S(t) \mathbf u_0 + i \int_0^t \mathbf S(t-s) \mathbf N\left( \Theta_R^{\mathbf u}(s) P_\mu \mathbf u(s) \right) \, ds
  + i \int_0^t \mathbf S(t-s) \mathbf L\left( P_\mu \mathbf u(s) \right) \, ds \right.
  \\
  \left.
  + i \int_0^t \mathbf S(t-s) \mathbf M\left( P_\mu \mathbf u(s) \right) \, dW(s) \right),
\end{multline}
and let us write $\mathbf u_0^\mu = P_\mu \mathbf u_0$. By the dominated convergence theorem,
\begin{equation}\label{inductive-start-bound}
  \norm{\mathbf u_0^\mu - \mathbf u_0}_{L^2(\Omega,\mathbf H^{\mathbf s})} \to 0 \quad \text{as $\mu \to \infty$}.
\end{equation}

To prove that the solutions of the regularised problem converge, we need the fact that
\begin{equation}\label{Regularity-M}
    \lim_{\mu \to \infty} \norm{(1-P_\mu)\mathbf M(\mathbf f)}_{\mathcal L_2(K,\mathbf H^\mathbf s)}
    =
    0
\end{equation}
for all $\mathbf f \in \mathbf H^{\mathbf s}$.
For any orthonormal basis $\{ e_j \}$ of $K$
we have, for each component $M_i(\mathbf f)$,
%
\begin{equation*}
    \norm{(1-P_\mu) M_i(\mathbf f)}
    _{\mathcal L_2(K,H^{s_i})}^2
    =
    \sum_j
    \norm{(1-P_\mu) M_i(\mathbf f)e_j}
    _{H^{s_i}}^2
    =
    \sum_j
    \int
    \left( 1 - \theta_\mu(\xi) \right)^2 \Abs{\widehat m_{i,j} (\xi)}^2
    \japanese{\xi}^{2s_i}
    d\xi
    ,
\end{equation*}
where $m_{i,j} = M_i(\mathbf f)e_j \in H^{s_i}$ and
\begin{equation*}
    \sum_j
    \int
    \Abs{\widehat m_{i,j} (\xi)}^2
    \japanese{\xi}^{2s_i}
    d\xi
    =
    \norm{M_i(\mathbf f)}
    _{\mathcal L_2(K,H^{s_i})}^2
    < \infty
    .
\end{equation*}
The dominated convergence theorem therefore implies \eqref{Regularity-M}.

We shall prove the following.

\begin{theorem}[Regularised global existence]\label{Existence-Theorem-Truncated-regularised}
Let $R > 0$, $\mathbf s \in \R^n$ and $-1/2 < b' < 0 < b < 1/2$. Assume that \eqref{M1}--\eqref{N3} are satisfied. Assume that $\mathbf u_0 \in L^2(\Omega,\mathbf H^{\mathbf s})$ is $\mathcal F_0$-measurable. Then for all $\mu \ge 1$ and $T > 0$ the regularised problem \eqref{PDEsolution-truncated-regularised} has a unique solution $\mathbf u^\mu \in Z^{\mathbf s, b}(0,T)$, with initial value $\mathbf u_0^\mu$, and $\mathcal F_x \mathbf u^\mu$ is supported in $[-2\mu,2\mu]$. Moreover,
\begin{equation}\label{Bound-truncated-solution-regularised}
  \norm{\mathbf u^\mu}_{Z^{\mathbf s,b}(0,T)} \le C_{T,R,b} \norm{\mathbf u_0}_{L^2(\Omega,\mathbf H^s)},
\end{equation}
where the constant is independent of $\mu$.
Finally, letting $\mathbf u$ be as in Theorem \ref{Existence-Theorem-Truncated} (where it is denoted $\mathbf u^R$),
we have
\begin{equation}\label{convergence-regularised-solution}
  \lim_{\mu \to \infty} \norm{\mathbf u^\mu - \mathbf u}_{Z^{\mathbf s,b}(0,T)} = 0.
\end{equation}
\end{theorem}

Before proving the above theorem, let us remark that the main reason for regularising is that we can then apply It\^o's formula, as formulated in Theorem 2.10 in \cite{Gawarecki_Mandrekar}; it does not apply directly to the problem \eqref{PDEsolution-truncated}, since the deterministic integral may not make sense as a Bochner integral in $\mathbf H^{\mathbf s}$. But in the frequency-truncated problem \eqref{PDEsolution-truncated-regularised}, the corresponding integral makes sense even in the more regular space $\mathbf H^{\mathbf s'}$, by the assumption \eqref{N3} and the fact that $P_\mu$ maps $\mathbf H^{\mathbf s}$ into $\mathbf H^{\mathbf s'}$. Then \cite[Theorem 2.10]{Gawarecki_Mandrekar} can be applied (after applying $\mathbf S(-t)$ on both sides of \eqref{PDEsolution-truncated-regularised}, and passing $P_\mu$ inside the integrals). Passing to the limit $\mu \to \infty$, one can then hope to get It\^o's formula also for \eqref{PDEsolution-truncated}. Indeed, this works out in a case of particular interest to us here, namely the conservation of charge for the stochastic Dirac-Klein-Gordon system. The details of this are shown in Section \ref{Charge_conservation}.

We now prove Theorem \ref{Existence-Theorem-Truncated-regularised}.
The existence and uniqueness works by a fixed point argument as in the proof of Theorem \ref{Existence-Theorem-Truncated}, up to some obvious modifications, since $P_\mu$ is bounded on all the spaces involved. From \eqref{PDEsolution-truncated-regularised} it is obvious that $\mathcal F_x \mathbf u^\mu$ is supported in $[-2\mu,2\mu]$. So it remains to prove \eqref{convergence-regularised-solution}. As in the proof of Theorem \ref{Existence-Theorem-Truncated} we cut $[0,T]$ into short intervals of length $\delta$, where $\delta > 0$ is chosen sufficiently small depending on $R$ and $T$ (but not on $\mu$). Suppose that we have proved \eqref{convergence-regularised-solution} on $[0,S]$ for some $0 \le S < T$; if $S=0$, we appeal to \eqref{inductive-start-bound}. Now we must prove \eqref{convergence-regularised-solution} on $[S,S']$ with $S'=S+\delta$. To this end, write
\[
  \mathbf u^\mu(t) - \mathbf u(t) = \Delta_1^\mu(t) + \Delta_2^\mu(t) + \Delta_3^\mu(t) + \Delta_4^\mu(t)
  \quad \text{for $t \ge S$},
\]
where
\begin{align*}
  \Delta_1^\mu(t) &= \mathbf S(t-S) \left( \mathbf u^\mu(S) - \mathbf u(S) \right),
  \\
  \Delta_2^\mu(t) &= i \int_S^t \mathbf S(t-s) \left[ P_\mu \mathbf N\left( \Theta_R^{\mathbf u^\mu}(s) P_\mu \mathbf u^\mu(s) \right) - \mathbf N\left( \Theta_R^{\mathbf u}(s) \mathbf u(s) \right) \right] \, ds,
  \\
  \Delta_3^\mu(t) &= i \int_S^t \mathbf S(t-s) \left[ P_\mu \mathbf L\left( P_\mu \mathbf u^\mu(s) \right) - \mathbf L\left( \mathbf u(s) \right) \right] \, ds,
  \\
  \Delta_4^\mu(t) &= i \int_S^t \mathbf S(t-s) \left[ P_\mu \mathbf M\left( P_\mu \mathbf u^\mu(s) \right)
  - \mathbf M\left( \mathbf u(s) \right) \right] \, dW(s).
\end{align*}

First, by \eqref{Bourgain-1b} and the induction hypothesis,
\begin{equation}\label{Delta-1-bound}
  \norm{\Delta_1^\mu}_{Z^{\mathbf s,b}(S,S')}
  \le C \norm{\mathbf u^\mu(S) - \mathbf u(S)}_{L^2(\Omega,\mathbf H^{\mathbf s})} \to 0 \quad \text{as $\mu \to \infty$}.
\end{equation}

Second, write
\[
  \Delta_2^\mu(t) = \Delta_{2,1}^\mu(t) + \Delta_{2,2}^\mu(t),
\]
where
\begin{align*}
  \Delta_{2,1}^\mu(t)
  &=
  i  P_\mu \int_0^t \mathbf S(t-s) 
  \left[\mathbf N\left( \Theta_R^{\mathbf u^\mu}(s) P_\mu \mathbf u^\mu(s) \right) - \mathbf N\left( \Theta_R^{\mathbf u}(s) \mathbf u(s) \right) \right] \, ds,
  \\
  \Delta_{2,2}^\mu(t)
  &=
  i  (P_\mu-1) \int_0^t \mathbf S(t-s) \mathbf N\left( \Theta_R^{\mathbf u}(s) \mathbf u(s) \right) \, ds.
\end{align*}
Estimating as in the proof of Theorem \ref{Existence-Theorem-Truncated}, we get
\begin{equation}\label{Delta-2-1-bound}
\begin{aligned}
  \norm{\Delta_{2,1}^\mu}_{Z^{\mathbf s,b}(S,S')}
  &\le C (1+R)^{\frac{p-1}{2}} \delta^{\varepsilon} \norm{P_\mu \Theta_R^{\mathbf u^\mu} \mathbf u^\mu - \Theta_R^{\mathbf u} \mathbf u}_{L^2(\Omega,\mathbf X^{\mathbf s,b}(S,S'))}
  \\
  &\le C (1+R)^{\frac{p-1}{2}}  \delta^{\varepsilon} \left( \norm{P_\mu \left[ \Theta_R^{\mathbf u^\mu}  \mathbf u^\mu - \Theta_R^{\mathbf u} \mathbf u \right]}_{L^2(\Omega,\mathbf X^{\mathbf s,b}(S,S'))}
  +
  \norm{(1-P_\mu) \Theta_R^{\mathbf u} \mathbf u}_{L^2(\Omega,\mathbf X^{\mathbf s,b}(S,S'))} \right)
  \\
  &\le C (1+R)^{\frac{p-1}{2}} \left( \delta^{\varepsilon} \norm{\mathbf u^\mu - \mathbf u}_{L^2(\Omega,\mathbf X^{\mathbf s,b}(S,S'))}
  +
  \rho(\mu,S) \right),
\end{aligned}
\end{equation}
where
\begin{equation}\label{Delta-2-1-bound-2}
  \rho(\mu,S)
  =
  \norm{\mathbf u^\mu - \mathbf u}_{L^2(\Omega,\mathbf X^{\mathbf s,b}(0,S))}
  +
  \norm{(1-P_\mu) \mathbf u}_{L^2(\Omega,\mathbf X^{\mathbf s,b}(0,T))}
  \to 0 \quad \text{as $\mu \to \infty$},
\end{equation}
by the induction hypothesis and the dominated convergence theorem. Write $\Delta_{2,2}^\mu = (1-P_\mu) \mathbf v$, where $\mathbf v \in L^2 \left(\Omega,\mathbf X^{\mathbf s,B}(0,T) \right)$, by the proof of Theorem \ref{Existence-Theorem-Truncated}. Then
by \eqref{Bourgain-4} and dominated convergence,
\begin{equation}\label{Delta-2-2-bound}
  \norm{\Delta_{2,2}^\mu}_{Z^{\mathbf s,b}(S,S')}
  \le
  C
  \norm{(1-P_\mu) \mathbf v}_{L^2(\Omega,\mathbf X^{\mathbf s,B}(0,T))}
  \to 0 \quad \text{as $\mu \to \infty$}.
\end{equation}

The estimates for $\Delta_2$ apply also to $\Delta_3$ (with $p=1$).

Finally, we split
\[
  \Delta_4^\mu(t)
  =
  \Delta_{4,1}^\mu(t)
  +
  \Delta_{4,2}^\mu(t)
  +
  \Delta_{4,3}^\mu(t),
\]
where
\begin{align*}
  \Delta_{4,1}^\mu(t)
  &=
  i \int_S^t \mathbf S(t-s) P_\mu \mathbf M\left( P_\mu \left[  \mathbf u^\mu -  \mathbf u \right](s) \right)
  \, dW(s),
  \\
  \Delta_{4,2}^\mu(t)
  &=
  i \int_S^t \mathbf S(t-s) P_\mu \mathbf M\left( [ P_\mu - 1] \mathbf u(s)  \right) \, dW(s),
  \\
  \Delta_{4,3}^\mu(t)
  &=
  i \int_S^t \mathbf S(t-s) \left( P_\mu -1 \right) \mathbf M\left( \mathbf u(s) \right)
  \, dW(s).
\end{align*}
Then as the proof of Theorem \ref{Existence-Theorem-Truncated}, and using the boundedness of $P_\mu$,
\begin{equation}\label{Delta-3-bound-1}
  \norm{\Delta_{4,1}^\mu}_{Z^{\mathbf s,b}(S,S')}
  \le C \delta^{b}
  \norm{\mathbf u^\mu - \mathbf u}_{L^2(\Omega,\mathbf X^{\mathbf s,b}(S,S'))}.
\end{equation}
By the dominated convergence theorem,
\begin{equation}\label{Delta-3-bound-2}
  \norm{\Delta_{4,2}^\mu}_{Z^{\mathbf s,b}(S,S')}
  \le
  C
  \norm{(1-P_\mu) \mathbf u}_{L^2(\Omega,\mathbf X^{\mathbf s,b}(0,T))}
  \to 0 \quad \text{as $\mu \to \infty$}.
\end{equation}
By \eqref{Maximal-bound} and Lemma \ref{X-lemma},
\begin{equation}\label{Delta-3-bound-3}
  \norm{\Delta_{4,3}^\mu}_{Z^{\mathbf s,b}(S,S')}
  \le
  C
  \mathbb E \left(   \int_0^T \norm{(1-P_\mu)\mathbf M(\mathbf u(s))}_{\mathcal L_2( K, \mathbf H^{\mathbf s})}^2 \, ds \right)
  \to 0 \quad \text{as $\mu \to \infty$},
\end{equation}
using \eqref{Regularity-M} and the dominated convergence theorem.

Combining \eqref{Delta-1-bound}--\eqref{Delta-3-bound-3}, we conclude that
\[
  \norm{\mathbf u^\mu - \mathbf u}_{Z^{\mathbf s,b}(S,S')}
  \le
  \frac12
  \norm{\mathbf u^\mu - \mathbf u}_{Z^{\mathbf s,b}(S,S')}
  + o(1)
  \quad \text{as $\mu \to \infty$},
\]
for $\delta=S'-S$ small enough. Together with the induction hypothesis this implies \eqref{convergence-regularised-solution} on the interval $[0,S']$, and by induction we then obtain the convergence on the whole interval $[0,T]$. This completes the proof of Theorem \ref{Existence-Theorem-Truncated-regularised}.

\subsection{Proof of Theorem \ref{local_th}}\label{ConcreteCase}

For the convenience of the reader, we now show exactly how the abstract framework applies to prove our first main result, Theorem \ref{local_th}, except for the charge conservation, which is proved in the next section.

Taking $d=1$ and $n=3$, we can cast the stochastic DKG system \eqref{mild_psi}, \eqref{mild_phi} in the form \eqref{PDEsolution}, with
\begin{gather*}
  \mathbf u = \begin{pmatrix}
      \psi_+
      \\
      \psi_-
      \\
      \phi_+
  \end{pmatrix},
  \qquad
  \mathbf u_0 = \begin{pmatrix}
      f_+
      \\
      f_-
      \\
      g_+
  \end{pmatrix},
  \qquad
  \mathbf h(\xi) = \begin{pmatrix}
      +\xi
      \\
      -\xi
      \\
      +\japanese{\xi}
  \end{pmatrix},
  \qquad
  \mathbf S(t) \mathbf u_0 = \begin{pmatrix}
      S_{+\xi}(t) f_+
      \\
      S_{-\xi}(t) f_-
      \\
      S_{+\japanese{\xi}}(t) g_+
  \end{pmatrix},
  \\
  \mathbf N(\mathbf u) = \begin{pmatrix}
      i \phi \psi_{-}
      \\
      i\phi \psi_{+}
      \\
      i \japanese{D_x}^{-1} \re \left( \overline{\psi_+} \psi_- \right)
  \end{pmatrix},
  \quad \mathbf L(\mathbf u) = \begin{pmatrix}
      - iM\psi_{-} - M_{\mathfrak{K}_1}\psi_{+}
      \\
      - iM\psi_{+} - M_{\mathfrak{K}_1}\psi_{-}
      \\
      0
  \end{pmatrix},
  \quad
  \mathbf M(\mathbf u) = \begin{pmatrix}
      i \psi_{-} \mathfrak{K}_1
      \\
      i \psi_{+} \mathfrak{K}_1
      \\
     (i/2) \japanese{D_x}^{-1} \phi \mathfrak{K}_2
  \end{pmatrix},
\end{gather*}
where $\phi$ stands for $\phi_+ + \overline{\phi_+}$. Let $s,r \in \R$ and $0 < b < 1/2$ be as in Lemma \ref{DKG-bilinear-lemma}, and set $b'=-b$. Corresponding to $\mathbf s=(s,s,r)$ we then define the spaces $\mathbf H^{(s,s,r)}$ and $\mathbf X^{(s,s,r),b}$ as in \eqref{spaceHssr} and \eqref{spaceXssr}. Define the Wiener process $W(t)$ using $K=L^2(\R,\R)$ with an orthonormal basis $\{e_j\}_{j \in \N}$. Assume that the convolution kernels satisfy $\mathfrak k_1 \in H^{\Abs{s}}(\R,\R)$ and $\mathfrak k_2 \in H^{\max(0,r-1)}(\R,\R)$.

The boundedness \eqref{M1} of $\mathbf M$ is now a consequence of Lemma \ref{DKG-HS-lemma}. Concerning $\mathbf N$, the property \eqref{N0} is obvious, \eqref{N1}, \eqref{N2} follow from Lemma \ref{DKG-bilinear-lemma}, and \eqref{N3} holds by Lemma \ref{DKG-high-reg-lemma}, if we take $\mathbf s' = (s',s',r')$ with $s'=r' > \max (r,s,1/2)$. For $\mathbf L$, the bound in \eqref{L1} holds by Lemma \ref{DKG-linear-lemma}, and the bound in \eqref{N3} is trivial.

So with the above set-up, Theorem \ref{local_th}, with the exception of the charge conservation (considered below), follows from Theorems \ref{Existence-Theorem} and \ref{Uniqueness-theorem}. 

\section{Charge conservation}
\label{Charge_conservation}
\setcounter{equation}{0}

Let $s=0$ and $0 < r < 1/2$. We now prove the statement in Theorem \ref{local_th} about charge conservation of the local solution $(\psi_+,\psi_-,\phi)$, almost surely for $0 \le t < \tau$. Let $R > 0$. As explained in Section \ref{Abstract-WP}, the solution equals, up to the conditional stopping time $\tau_R$, the solution $(\psi_+^R,\psi_-^R,\phi^R)$ of the $R$-truncated problem, obtained in Theorem \ref{Existence-Theorem-Truncated} via the set-up in Section \ref{ConcreteCase}. Since $\tau_R \to \tau$ as $R \to \infty$, it clearly suffices to prove the charge conservation for $(\psi_+^R,\psi_-^R,\phi^R)$.

In the remainder of this section we fix $R > 0$, and to simplify the notation we drop the superscript $R$ on the solution. Thus, $(\psi_+,\psi_-,\phi)$ denotes the global solution of the truncated versions of \eqref{mild_psi}, \eqref{mild_phi}:
\begin{multline}
\label{mild_psi-truncated}
	\psi_{\pm}(t)
	=
	S_{\pm\xi}(t) f_\pm
	-
	iM \int_0^t S_{\pm\xi}(t-s) \psi_{\mp}(s) \, ds
	+
	i \int_0^t  S_{\pm\xi}(t-s) ( \Theta \phi \Theta \psi_{\mp})(s) \, ds
	\\
	+
	i \int_0^t S_{\pm\xi}(t-s) \psi_\mp(s) \mathfrak K_1 \, d W(s)
	-
	M_{\mathfrak K_1} \int_0^t S_{\pm\xi}(t-s) \psi_\pm(s) \, ds,
\end{multline}
and
\begin{multline}
\label{mild_phi-truncated}
	\phi_+(t)
	=
	S_{+\japanese{\xi}}(t) g_+
	+
	i \int_0^t S_{+\japanese{\xi}}(t-s) \japanese{D_x}^{-1}
    \re \left( \Theta\overline{\psi_+} \Theta\psi_- \right)(s) \, ds
	\\
	+
	\frac{i}{2} \int_0^t S_{+\japanese{\xi}}(t-s) \japanese{D_x}^{-1}
    \phi(s) \mathfrak K_2 \, dW(s)
.
\end{multline}
Here $\phi=2\re\phi_+$ and $\Theta(t)$ is defined as in \eqref{DKG-cutoff}, with $s=0$. We assume that $\theta$ is even, so that $P_\mu f$ is real-valued if $f$ is.

Set $\psi=(\psi_+,\psi_-)$ and $\psi_0=(f_+,f_-)$. We will prove that the charge is almost surely conserved:
\[
  \norm{\psi(t)}_{L^2}^2
  =
  \norm{\psi_0}_{L^2}^2 \quad \text{for $t \ge 0$.}  
\]
To this end, we want to apply It\^o's formula with the functional $\mathcal H \colon L^2(\R) \to \R$ given by
\[
  \mathcal H(u) = \int_{\R} \Abs{u(x)}^2 \, dx.
\]
However, as discussed in Section \ref{Regularisation}, it is necessary to first regularise the problem. So for $\mu \ge 1$ we consider the solution $(\psi_+^\mu,\psi_-^\mu,\phi^\mu)$, obtained in Theorem \ref{Existence-Theorem-Truncated-regularised}, of the frequency-truncated equations
\begin{multline}
\label{mild_psi-truncated-regularised}
	\psi_{\pm}^\mu(t)
	=
	S_{\pm\xi}(t) P_\mu f_\pm
	-
	iM \int_0^t S_{\pm\xi}(t-s) P_\mu^2 \psi_{\mp}^\mu(s) \, ds
	+
	i \int_0^t  S_{\pm\xi}(t-s) P_\mu \left( \Theta P_\mu \phi^\mu \cdot \Theta P_\mu \psi_{\mp}^\mu \right)(s) \, ds
	\\
	+
	i \int_0^t S_{\pm\xi}(t-s) P_\mu \left(P_\mu \psi_\mp^\mu(s)\right) \mathfrak K_1 \, d W(s)
	-
	M_{\mathfrak K_1} \int_0^t S_{\pm\xi}(t-s) P_\mu^2 \psi_\pm^\mu(s) \, ds,
\end{multline}
and
\begin{multline}
\label{mild_phi-truncated-regularise}
	\phi_+^\mu(t)
	=
	S_{+\japanese{\xi}}(t) P_\mu g_+
	+
	i \int_0^t S_{+\japanese{\xi}}(t-s) \japanese{D_x}^{-1}
    P_\mu \re \left( \Theta P_\mu\overline{\psi_+^\mu} \cdot \Theta P_\mu \psi_-^\mu \right)(s) \, ds
	\\
	+
	\frac i2 \int_0^t S_{+\japanese{\xi}}(t-s)  \japanese{D_x}^{-1}
    P_\mu (P_\mu\phi^\mu)(s) \mathfrak K_2 \, dW(s),
\end{multline}
where $\phi^\mu=2\re \phi_+^\mu$. Set $\psi^\mu= \left(\psi_+^\mu,\psi_-^\mu \right)$.
Then by Theorem \ref{Existence-Theorem-Truncated-regularised} we have
\begin{equation}\label{DKG-convergence}
  \mathbb E \left( \sup_{t \in [0,T]} \left( \norm{\psi^\mu(t) - \psi(t)}_{L^2}^2
  + \norm{\phi^\mu(t) - \phi(t)}_{H^r}^2
   \right) \right) \to 0 \quad \text{as $\mu \to \infty$}
\end{equation}
for any $T > 0$. Moreover, the spatial Fourier transform of $\left(\psi_+^\mu,\psi_-^\mu,\phi^\mu \right)$ is supported in $[-2\mu,2\mu]$.

Notice that
\(
    \mathcal H(\psi^\mu)
    =
    \mathcal H(\psi_+^\mu)
    +
    \mathcal H(\psi_-^\mu)
\)
and that the first and second derivatives of the functional $\mathcal H$ are given by the linear form
\(
    \mathcal H'(u)
    =
    2 \re \langle \cdot, u \rangle _{L^2}
\)
and the bilinear form
\(
    \mathcal H''(u)
    =
    2 \re \langle \cdot , \cdot \rangle _{L^2}
    .
\)

In terms of $X_\pm(t) = S_{\pm\xi}(-t) \psi_\pm^\mu(t)$ we can rewrite \eqref{mild_psi-truncated-regularised} as
\[
  X_\pm(t) = X_\pm(0) + \int_0^t \Psi(s) \, ds + \int_0^t \Phi(s) \, dW(s),
\]
where
\[
  \Psi(t) = -iM S_{\pm\xi}(-t) P_\mu^2 \psi_{\mp}^\mu(t) 
  +
  i S_{\pm\xi}(-t) P_\mu ( \Theta P_\mu \phi^\mu \cdot \Theta P_\mu \psi_{\mp}^\mu)(t)
  -
  M_{\mathfrak K_1} S_{\pm\xi}(-t) P_\mu^2 \psi_\pm^\mu(t)
\]
and
\[
  \Phi(t) = i S_{\pm\xi}(-t) P_\mu \left(P_\mu \psi_\mp^\mu(t)\right) \mathfrak K_1.
\]
Applying now It\^o's formula, as stated in \cite[Theorem 2.10]{Gawarecki_Mandrekar}, we get
\begin{multline*}
  \mathcal H(X_\pm(t)) - \mathcal H(X_\pm(0))
  \\
  =
  \int_0^t \mathcal H'(X_\pm(s)) \Psi(s) \, ds
  +
  \int_0^t \mathcal H'(X_\pm(s)) \Phi(s) \, dW(s)
  + \int_0^t \frac12 \trace \mathcal H''(X_\pm(s))\left(\Phi(s),\Phi(s)\right) \, ds.
\end{multline*}
Using the fact that the group $S_{\pm\xi}(t)$ is unitary on $L^2(\R)$, the above works out to be
\begin{multline*}
    \mathcal H(\psi_\pm^\mu(t))
    -
    \mathcal H(\psi_\pm^\mu(0))
    \\
  =
  \int_0^t 2\re \Innerprod{-iM P_\mu^2 \psi_{\mp}^\mu(s) 
  +
  i P_\mu ( \Theta P_\mu \phi^\mu \cdot \Theta P_\mu \psi_{\mp}^\mu)(s)
  -
  M_{\mathfrak K_1} P_\mu^2 \psi_\pm^\mu(s)}
  {\psi_\pm^\mu(s)}_{L^2} \, ds
  \\
  +
  \int_0^t 2\re \Innerprod{i P_\mu \left(P_\mu \psi_\mp^\mu(s)\right) \mathfrak K_1 \cdot}
  {\psi_\pm^\mu(s)}_{L^2} \, dW(s)
  \\
  + \int_0^t
    \sum _{j = 1}^{\infty}
    \re
    \left\langle
        i P_{\mu} ( P_{\mu} \psi_\mp^\mu(s) ) \mathfrak K_1 e_j
        ,
        i P_{\mu} ( P_{\mu} \psi_\mp^\mu(s) ) \mathfrak K_1 e_j
    \right\rangle
    _{L^2}
	ds
  \\
        =
    I_\pm + II_\pm + III_\pm.
\end{multline*}
Since $M, M_{\mathfrak K_1} \in \R$, $P_\mu \phi^\mu$ and $\Theta$ are real-valued, and $P_\mu$ is hermitian, it is clear that
\[
  I_+ + I_-
  =
  -
  M_{\mathfrak K_1} 2\re  \int_0^t \left(\Innerprod{ P_\mu \psi_+^\mu(s)}
  {P_\mu\psi_+^\mu(s)}_{L^2}
  +
  \Innerprod{ P_\mu\psi_-^\mu(s)}
  {P_\mu\psi_-^\mu(s)}_{L^2} \right)  \, ds.
\]
Thus
\[
    I_+ + I_-
    =
    -
    2 M_{\mathfrak K_1} \int_0^t
    \norm{ P_{\mu} \psi^\mu(s) }_{L^2}^2
    ds
    =
    - \int_0^t
    \norm{ \left( P_{\mu} \psi^\mu(s) \right) \mathfrak K_1}_{\mathcal L_2}^2
    ds,
\]
 where $\mathcal L_2 = \mathcal L_2(L^2,L^2)$ and the last equality holds by Remark \ref{Hilbert_Schmidt_remark}, since $2M_{\mathfrak K_1} = \norm{\mathfrak k_1}_{L^2}^2$.
Next we notice that
\[
    II_\pm
    =
    -2
    \sum _{j = 1}^{\infty}
    \int_0^t
    \im
    \left\langle
        ( P_{\mu} \psi_\mp^\mu(s) ) \mathfrak K_1 e_j
        ,
        P_{\mu} \psi_\pm^\mu(s)
    \right\rangle
    _{L^2}
	\, dB_j(s)
    .
\]
Thus
\(
    II_+ + II_- = 0
    ,
\)
since $\mathfrak K_1 e_j$ is a real-valued function.
Finally we notice that
\[
    III_\pm
    =
    \int_0^t
    \norm
    {
        P_{\mu} ( P_{\mu} \psi_\mp^\mu(s) ) \mathfrak K_1
    }
    _{\mathcal L_2}^2
	ds
    ,
\]
hence
\[
    III_+ + III_-
    =
    \int_0^t
    \norm
    {
        P_{\mu} ( P_{\mu} \psi^\mu(s) ) \mathfrak K_1
    }
    _{\mathcal L_2}^2
	ds
,
\]
where the operator inside the norm is regarded as a composition
of three operators (first apply $\mathfrak K_1$, then multiplication by $P_{\mu} \psi^\mu$ and finally $P_\mu$).

Summing the contributions, we arrive at
\begin{equation}
\label{charge_control}
    \mathcal H(\psi^\mu(t))
    -
    \mathcal H(\psi^\mu(0))
    =
    \int_0^t
    \left(
    \norm
    {
        P_{\mu} ( P_{\mu} \psi^\mu(s) ) \mathfrak K_1
    }
    _{\mathcal L_2}^2
	-
   \norm{ \left( P_{\mu} \psi^\mu(s) \right) \mathfrak K_1}_{\mathcal L_2}^2
    \right)
    \, ds,
\end{equation}
and letting $\mu \to \infty$ we get the charge conservation for $\psi$,
\[
  \mathcal H(\psi(t))
    -
    \mathcal H(\psi(0))
    =
    0 \quad \text{for all $t \ge 0$}.
\]
Indeed, let $T > 0$. By \eqref{DKG-convergence} we know that for some sequence $\mu_k \to \infty$ as $k \to \infty$, we have, almost surely,
\[
   \sup_{t \in [0,T]} \norm{\psi^{\mu_k}(t) - \psi(t)}_{L^2} \to 0
   \quad \text{as $k \to \infty$}.
\]
Thus, along $\mu=\mu_k$, the left hand side of \eqref{charge_control} converges to $\mathcal H(\psi(t)) - \mathcal H(\psi(0))$ for $0 \le t \le T$. Moreover, the right hand side converges to zero by the dominated convergence theorem. Indeed,
for $\mu$ large enough we have the bounds, uniformly in $s \in [0,T]$,
\[
  \norm
    {
        P_{\mu} ( P_{\mu} \psi^\mu(s) ) \mathfrak K_1
    }
    _{\mathcal L_2}
  \le
  \norm
    {
     ( P_{\mu} \psi^\mu(s) ) \mathfrak K_1
    }
    _{\mathcal L_2}
   \le
   \norm{\psi^\mu(s)}_{L^2} \norm{\mathfrak k_1}_{L^2}
   \le
    \left( \norm{\psi(s)}_{L^2} 
   + 1 \right)
   \norm{\mathfrak k_1}_{L^2},
\]
and
\[
    0 \le
    \norm{ (P_{\mu} \psi^\mu(s) ) \mathfrak K_1}_{\mathcal L_2}
    -
    \norm
    {
        P_{\mu} ( P_{\mu} \psi^\mu(s) ) \mathfrak K_1
    }
    _{\mathcal L_2}
    \\
    \le
    C \left( \norm{\psi(s)}_{L^2} + 1 \right)
    \norm{ (1-P_\mu) ( P_{\mu} \psi^\mu(s) ) \mathfrak K_1}_{\mathcal L_2}
    ,
\]
so it only remains to check that $\norm{ (1-P_\mu) ( P_{\mu} \psi^\mu(s) ) \mathfrak K_1}_{\mathcal L_2}$ tends to zero along $\mu=\mu_k$ as $k \to \infty$. But
\begin{multline*}
  \norm{ (1-P_\mu) ( P_{\mu} \psi^\mu(s) ) \mathfrak K_1}_{\mathcal L_2}
  \le
  \norm{ (1-P_\mu) \psi(s) \mathfrak K_1}_{\mathcal L_2}
  +
  \norm{ (1-P_\mu) ( P_{\mu} \psi^\mu(s) - \psi(s) ) \mathfrak K_1}_{\mathcal L_2}
  \\
  \le
  \norm{ (1-P_\mu) \psi(s) \mathfrak K_1}_{\mathcal L_2}
  +
  \norm{ (P_{\mu} \psi^\mu(s) -\psi(s) ) \mathfrak K_1}_{\mathcal L_2}
  ,
\end{multline*}
since $\norm{1 - P_{\mu}} \leq 1$.
Now the first term $\norm{ (1-P_\mu) \psi(s) \mathfrak K_1}_{\mathcal L_2} \to 0$ by \eqref{Regularity-M},
and the last term equals
\[
    \norm{ P_{\mu} \psi^\mu(s) -\psi(s) }_{L^2}
    \norm{\mathfrak k_1}_{L^2}
    \le
    \norm{ \psi^\mu(s) -\psi(s) }_{L^2}
    \norm{\mathfrak k_1}_{L^2}
    +
    \norm{ (P_{\mu} - 1)\psi(s) }_{L^2}
    \norm{\mathfrak k_1}_{L^2}
\]
that tends to zero along $\mu=\mu_k$ as $k \to \infty$.
This concludes the proof of charge conservation.

\section{Global existence}
\label{Global_existence}
\setcounter{equation}{0}

This section is devoted to the proof of Theorem \ref{global_th}.
So we suppose now
\begin{equation}\label{rb-cond}
  s=0, \qquad \frac14 < r < \frac12, \qquad \max(r,1-2r) < b < 1/2.
\end{equation}
Fix some Lebesgue exponent
\[ 
  p \ge \max\left( 4,  \frac{2b + 2r - 1}{b + 2r - 1} \right)
\]
and assume $f_\pm \in L^p \left( \Omega, L^2 \right)$, as well as $g_+ \in L^2(\Omega,H^r)$.

Consider the local solution $(\psi_+,\psi_-,\phi_+)$ from Theorem \ref{local_th}, existing up to the stopping time $\tau$. For $R \ge 1$ let $(\psi_+^R,\psi_-^R,\phi_+^R)$ be the solution of the truncated problem \eqref{mild_psi-truncated}, \eqref{mild_phi-truncated}, obtained in Theorem \ref{Existence-Theorem-Truncated}; it equals $(\psi_+,\psi_-,\phi_+)$ up to the stopping time $\tau_R$, which increases to $\tau$ as $R \to \infty$. As proved in the last section, we have almost surely the conservation of charge,
\begin{equation}\label{R-charge}
  \norm{\psi_+^R(t)}_{L^2}^2 + \norm{\psi_-^R(t)}_{L^2}^2 = \norm{\psi_0}_{L^2}^2
  \quad \text{for all $t \ge 0$}.
\end{equation}
And by \eqref{Truncated-growth} we have
\begin{equation}\label{R-blowup}
  \tau < \infty \implies
  \norm{\psi_+^R}_{X^{0,b}_{+\xi}(0,\tau_R)}
  +
  \norm{\psi_-^R}_{X^{0,b}_{-\xi}(0,\tau_R)}
  +
  \norm{\phi_+^R}_{X^{r,b}_{+\japanese{\xi}}(0,\tau_R)}
  \ge C \sqrt{R} \quad \text{for all $R$},
\end{equation}
where $C$ depends only on $b$.

We now claim that for all $R, T \ge 1$ we have bounds
\begin{equation}\label{phi-R-bound}
  \norm{\phi_+^R}_{L^2 \left( \Omega, X^{r,b}_{+\japanese{\xi}}(0,T) \right)}
  \le C(T) \left( \norm{g_+}_{L^2(\Omega,H^r)} + \norm{\psi_0}_{L^4(\Omega,L^2)}^2 \right)
\end{equation}
and
\begin{equation}\label{psi-R-bound}
  \norm{\psi_\pm^R}_{L^1\left(\Omega,X^{0,b}_{\pm\xi}(0,T)\right)}
  \le C\left(T, \norm{g_+}_{L^2(\Omega,H^r)},\norm{\psi_0}_{L^p(\Omega,L^2)} \right),
\end{equation}
which are \emph{uniform in $R$}. That is, the right hand sides do not depend on $R$. Note that by H\"older's inequality, \eqref{phi-R-bound} implies the $L^1(\Omega)$-bound
\begin{equation}\label{phi-R-bound-L1}
  \norm{\phi_+^R}_{L^1 \left(\Omega,X^{r,b}_{+\japanese{\xi}}(0,T)\right)}
  \le C(T) \left( \norm{g_+}_{L^2(\Omega,H^r)} + \norm{\psi_0}_{L^4(\Omega,L^2)}^2 \right).
\end{equation}

Granting the above claim for the moment, then we can almost surely exclude the scenario $\tau 
 < \infty$, as follows. Since $\{ \tau < \infty \} = \bigcup_{T=1}^\infty \{ \tau < T \}$, it suffices to check that $\mathbb P \left( \{ \tau < T \} \right) = 0$ for all $T \ge 0$. But \eqref{R-blowup} implies that
\[
  \norm{\psi_+^R}_{L^1\left(\Omega,X^{0,b}_{+\xi}(0,T)\right)}
  +
  \norm{\psi_-^R}_{L^1\left(\Omega,X^{0,b}_{-\xi}(0,T)\right)}
  +
  \norm{\phi_+^R}_{L^1\left(\Omega,X^{r,b}_{+\japanese{\xi}}(0,T)\right)}
  \ge C\sqrt{R} \mathbb P \left( \{ \tau < T \} \right)
\]
for all $R \ge 1$. So if $\mathbb P \left( \{ \tau < T \} \right) > 0$, then letting $R \to \infty$ we get a contradiction to \eqref{phi-R-bound-L1} and \eqref{psi-R-bound}.

It remains to prove the claim. Fix $R,T \ge 1$. We first prove \eqref{phi-R-bound}. Here we follow the proof of the bound \eqref{Bound-truncated-solution} in Theorem \ref{Existence-Theorem-Truncated}, but with one crucial difference: to ensure that the bounds are uniform in $R$, we have to avoid using the cutoff bound \eqref{Cutoff-fact1} at any point.

We cut $[0,T]$ into small subintervals $[0,\delta]$, $[\delta,2\delta]$ etc.
Fix now a subinterval $[S,S']$, $S'=S+\delta$. Recall that $\phi_+^R$ satisfies \eqref{mild_phi-truncated}, which we rewrite as
\[
  \phi_+^R(t)
	=
	S_{+\japanese{\xi}}(t-S) \phi_+^R(S)
	+ \Phi_S(t) + \Psi_S(t)
\]
where
\[
  \Phi_S(t) = i\int_S^t S_{+\japanese{\xi}}(t-s) \japanese{D_x}^{-1}
  \re \left( \Theta\overline{\psi_+^R} \Theta\psi_-^R \right)(s) \, ds,
\]
\[
  \Theta(t) = \theta_R\left(
  \norm{\psi_+^R}_{\widetilde X^{0,b}_{+\xi}(0,t)}^2
  +
  \norm{\psi_-^R}_{\widetilde X^{0,b}_{-\xi}(0,t)}^2
  +
  \norm{\phi_+^R}_{\widetilde X^{r,b}_{+\japanese{\xi}}(0,t)}^2
  \right)
\]
and
\[
  \Psi_S(t) = \frac{i}{2} \int_S^t S_{+\japanese{\xi}}(t-s) \japanese{D_x}^{-1}
    \phi^R(s) \mathfrak K_2 \, dW(s).
\]
Write
\[
  \norm{\phi_+^R}_{Z(S,S')}
  =
  \norm{\phi_+^R}_{L^2\left(\Omega,X^{r,b}_{+\japanese{\xi}}(S,S')\right)}
  +
  \norm{\phi_+^R}_{L^2(\Omega,C([S,S'],H^r))}
\]
As in the proof of Theorem \ref{Existence-Theorem-Truncated} we obtain (using Corollary \ref{X-cor})
\[
  \norm{\Psi_S}_{Z(S,S')} \le C \delta^b \norm{\phi_+^R}_{Z(S,S')}.
\]
So taking $\delta$ small enough, depending on $b$, we get, using also \eqref{Bourgain-1b} and \eqref{Bourgain-4},
\[
  \norm{\phi_+^R}_{Z(S,S')}
  \le C\left( \norm{\phi_+^R(S)}_{L^2(\Omega,H^r)} + \norm{\Phi_S}_{L^2(\Omega,X^{r,1}_{+\japanese{\xi}}(S,S'))} \right).
\]
By \eqref{Bourgain-2b}, and using $\abs{\Theta} \le 1$, we get, almost surely,
\[
  \norm{\Phi_S}_{X^{r,1}_{+\japanese{\xi}}(S,S')}
  \le
  C \norm{\overline{\psi_+^R}\psi_-^R}_{L_t^2((S,S'),H^{r-1})}
  \le
  C \sqrt{\delta} \norm{\psi_0}_{L^2}^2,
\]
where we used the Sobolev product law \eqref{Sob-product-1}, \eqref{Sob-product-2}, and the conservation of charge \eqref{R-charge}. So we conclude that on each subinterval $[S,S'] = [n\delta,(n+1)\delta]$ we have
\[
  \norm{\phi_+^R}_{Z(n\delta,(n+1)\delta)}
  \le C\left( \norm{\phi_+^R(n\delta)}_{L^2(\Omega,H^r)} +\sqrt{\delta} \norm{\psi_0}_{L^4(\Omega,L^2)}^2 \right).
\]
For $n=0$, $\phi_+^R(n\delta) = g_+$, while for $n \ge 1$,
\[
  \norm{\phi_+^R(n\delta)}_{L^2(\Omega,H^r)}
  \le \norm{\phi_+^R}_{Z((n-1)\delta,n\delta)}
\]
It follows that
\[
  \norm{\phi_+^R}_{Z(n\delta,(n+1)\delta)}
  \le C^{n+1} \norm{g_+}_{L^2(\Omega,H^r)} + \left( C + C^2 + \cdots + C^{n+1} \right) \sqrt{\delta} \norm{\psi_0}_{L^4(\Omega,L^2)}^2
  \quad \text{for $n=0,\dots,T/\delta$}.
\]
Summing over the subintervals now yields \eqref{phi-R-bound}.

With \eqref{phi-R-bound} in hand, we now prove \eqref{psi-R-bound}. The field $\psi_\pm^R$ satisfies \eqref{mild_psi-truncated}, which we write here as
\[
  \psi_{\pm}^R(t)
  =
  S_{\pm\xi}(t) f_\pm
  +
  \sum_{j=1}^4 \Psi_{j,\pm}(t)
\]
with
\begin{align*}
  \Psi_{1,\pm}(t)
  &=
  - iM \int_0^t S_{\pm\xi}(t-s)  \psi_{\mp}^R(s) \, ds,
  \\
  \Psi_{2,\pm}(t)
  &=
  - M_{\mathfrak K_1} \int_0^t S_{\pm\xi}(t-s) \psi_\pm^R(s) \, ds,
  \\
  \Psi_{3,\pm}(t)
  &=
  i \int_0^t S_{\pm\xi}(t-s) \psi_\mp^R(s) \mathfrak K_1 \, d W(s),
  \\
  \Psi_{4,\pm}(t)
  &=
  i \int_0^t  S_{\pm\xi}(t-s) \left( \Theta \phi^R \Theta \psi_{\mp}^R \right)(s) \, ds.
\end{align*}

In the estimates for these terms,
we use the bounds \eqref{Bourgain-1b} and \eqref{Bourgain-2b} on a time interval $[0,T]$ with $T \ge 1$,
hence the factors $\sqrt{T}$ and $T^{3/2}$ come up.
By \eqref{Bourgain-1b},
\begin{equation}\label{psi-R-bound-1}
  \norm{S_{\pm\xi}(t) f_\pm}_{L^1\left(\Omega,X^{0,b}_{\pm \xi}(0,T)\right)}
  \le C\sqrt{T} \norm{\psi_0}_{L^1(\Omega,L^2)}
  \le C \sqrt{T} \norm{\psi_0}_{L^2(\Omega,L^2)}.
\end{equation}
By \eqref{Bourgain-2b} and the charge conservation \eqref{R-charge},
\begin{equation}\label{psi-R-bound-2}
  \norm{\Psi_{1,\pm}}_{L^1\left(\Omega,X^{0,b}_{\pm \xi}(0,T)\right)}
  \le
  CM T^{3/2} \norm{\psi_{\mp}^R}_{L^1 \left(\Omega,X^{0,0}_{\pm \xi}(0,T) \right)}
  =
  CM T^{3/2} \norm{\psi_{\mp}^R}_{L^1 \left(\Omega,L^2((0,T) \times \R)\right)}
  \le
  CM T^2 \norm{\psi_0}_{L^2(\Omega,L^2)}.
\end{equation}
and similarly for $\Psi_{2,\pm}$.
Applying Lemma \ref{X-lemma} with the operator $M_1(f) = f\mathfrak K_1$, which by Remark \ref{Hilbert_Schmidt_remark} satisfies $\norm{M_1(f)}_{\mathcal L_2(L^2,L^2)} = \norm{f}_{L^2} \norm{\mathfrak k_1}_{L^2}$,
we get
\begin{equation}\label{psi-R-bound-3}
  \norm{\Psi_{3,\pm}}_{L^1 \left(\Omega,X^{0,b}_{\pm \xi}(0,T)\right)}
  \le
  \norm{\Psi_{3,\pm}}_{L^2 \left(\Omega,X^{0,b}_{\pm \xi}(0,T)\right)}
  \le
  C \norm{\psi_{\mp}^R}_{L^2(\Omega,L^2((0,T) \times \R))}
  \le
  C\sqrt{T} \norm{\psi_0}_{L^2(\Omega,L^2)}.
\end{equation}
By \eqref{Bourgain-2b} with the time regularity $1 - b > 1/2$ on the left hand side,
\[
  \norm{\Psi_{4,\pm}}_{X^{0,b}_{\pm \xi}(0,T)}
  \le
  CT^{3/2} \norm{\Theta \phi^R \Theta \psi_{\mp}^R}_{X^{0,-b}_{\pm \xi}(0,T)}.
\]
Note that \eqref{rb-cond} implies $1/2 - r < b/2 < b$.
So if one defines $\mu$ by
\[
  \frac12 - r = \mu b
  ,
\]
then $0 < \mu < 1/2$, and we find ourselves in the assumption of Corollary \ref{Null-form-corollary}.
Therefore, one can bound the norm on the right hand side as
\[
  \norm{\Theta \phi^R \Theta \psi_{\mp}^R}_{X^{0,-b}_{\pm \xi}(0,T)}
  \le
  C
  \norm{\Theta\phi_+^R}_{X^{r,b}_{+\japanese{\xi}}(0,T)}
  \norm{\Theta\psi_{\mp}^R}_{X^{0,b}_{\mp \xi}(0,T)}^{\mu}
  \norm{\Theta\psi_{\mp}^R}_{ L^2((0,T) \times \mathbb R) }^{ 1 - \mu }
  ,
\]
where the last norm
\[
  \norm{\Theta\psi_{\mp}^R}_{L^2((0,T) \times \R)}
  \le
  \sqrt{T} \norm{\psi_0}_{L^2}
\]
by charge conservation and the bound $\Abs{\Theta} \le 1$.
The first two norms are estimated as
\[
    \norm{\Theta\phi_+^R}_{X^{r,b}_{+\japanese{\xi}}(0,T)}
    \le
    C(T)
    \norm{\phi_+^R}_{X^{r,b}_{+\japanese{\xi}}(0,T)}
    , \quad
    \norm{\Theta\psi_{\mp}^R}_{X^{0,b}_{\mp \xi}(0,T)}
    \le
    C(T)
    \norm{\psi_{\mp}^R}_{X^{0,b}_{\mp \xi}(0,T)}
    ,
\]
where we used \eqref{Cutoff-fact2} (with $\mathbf v=\mathbf 0$)
to dispose of the the cutoff in front of $\phi_+^R$ and $\psi_+^R$
(this is why there appears a constant $C(T)$ depending on $T$).
Combining the above estimates, we have
\[
  \norm{\Psi_{4,\pm}}_{X^{0,b}_{\pm \xi}(0,T)}
  \le
  C(T)
  \norm{\phi_+^R}_{X^{r,b}_{+\japanese{\xi}}(0,T)}
  \left( \sqrt{T} \norm{\psi_0}_{L^2} \right)^{1-\mu}
  \norm{\psi_\mp^R}_{X^{0,b}_{\mp \xi}(0,T)}^{\mu}.
\]
Now write
\[
  1 = \frac12 + \frac{1-\mu}{q} + \mu,
  \quad \text{where} \quad
  q = \frac{1-\mu}{1/2-\mu}
  = \frac{2b + 2r - 1}{b + 2r - 1}.
\]
Then by H\"older's inequality,
\begin{equation}\label{psi-R-bound-4}
  \norm{\Psi_{4,\pm}}_{L^1 \left(\Omega,X^{0,b}_{\pm \xi}(0,T) \right)}
  \le
  C(T) T^{(1-\mu)/2}
  \norm{\phi_+^R}_{L^2 \left(\Omega,X^{r,b}_{+\japanese{\xi}}(0,T)\right)}
  \norm{\psi_0}_{L^q(\Omega,L^2)}^{1-\mu}
  \norm{\psi_{\mp}^R}_{L^1 \left(\Omega,X^{0,b}_{\mp \xi}(0,T) \right)}^{\mu}
  ,
\end{equation}
Note that $q \le p$, so that we can replace $L^q$ by $L^p$ in the last estimate.

So finally, combining \eqref{psi-R-bound-1}--\eqref{psi-R-bound-4}, making use of \eqref{phi-R-bound}, and setting
\[
    f(T)
    =
    \norm{\psi_+^R}_{L^1 \left( \Omega,X^{0,b}_{+\xi}(0,T) \right)}
    +
    \norm{\psi_-^R}_{L^1 \left( \Omega,X^{0,b}_{-\xi}(0,T) \right)}
    ,
\]
we deduce that for all $T \ge 1$ holds
\[
  f(T) \le CT \norm{\psi_0}_{L^2(\Omega,L^2)} + C(T)A(T) T^{(1-\mu)/2} \norm{\psi_0}_{L^p(\Omega,L^2)}^{1-\mu} \left[ f(T)\right]^\mu,
\]
where $A(T)$ stands for the expression on the right hand side of \eqref{phi-R-bound}. Recalling that $0 < \mu < 1/2$, we conclude that
\[
  f(T) \le C T \left( 1 + C(T)A(T) \right) \left( 1 + \norm{\psi_0}_{L^p(\Omega,L^2)} \right)
  \left( 1 +  \sqrt{f(T)} \right)
\]
for all $T \ge 1$. By the next lemma, we then get \eqref{psi-R-bound}, and this concludes the proof of global existence.

\begin{lemma}
If $a,b \in [0,\infty)$ satisfy
\begin{equation}\label{boot-strap}
    b \le a \left( 1 + \sqrt{b} \right),
\end{equation}
then
\[
  b < 1 + 4a^2.
\]
\end{lemma}

\begin{proof}
To get a contradiction, assume that $b \ge 1 + 4a^2$.
Then $b \ge 1$, so \eqref{boot-strap} implies $b \le a 2\sqrt{b}$, hence
$b \le 4a^2$, contradicting our assumption.
\end{proof}

\section{Bourgain isometry}
\label{Bourgain_isometry}
\setcounter{equation}{0}

This section is devoted to the proof of Lemma \ref{Bourgain_isometry_lemma}.
We start with a very general result related to Bochner spaces.
It may be difficult to find it in the existing literature,
so we provide here a complete proof. 

\begin{lemma}
	Let $\mu$ be a $\sigma$-finite complete measure on $Y$
	and $G$ be a separable Banach space.
	Consider a closed subspace $G_0$ in $G$ with the quotient projection
	\(
		\pi_0 : G \to G / G_0
		.
	\)
	The trivial case $G_0 = G$ is excluded, of course.
	Then there exists a unique linear operator $\Phi$ making the following diagram commutative:
	\begin{equation}
	\label{Bochner_commutative_diagram}
	\begin{tikzcd}
    	L^p(Y, \mu; G)
    	\arrow[r, "P"]
    	\arrow[d, " \pi : f \mapsto \bigl( y \mapsto \pi_0 ( f(y) ) \bigr) \;\; "']
    	&
    	L^p(Y, \mu; G) / L^p(Y, \mu; G_0)
		\arrow[dl, "\Psi = \Phi^{-1}", bend left=10, dashed]
    	\\
    	L^p(Y, \mu; G / G_0)
		\arrow[ur, "\Phi", dashed]
 	\end{tikzcd}
	\end{equation}
	where $P$ is the quotient projection and $p \in [1, \infty)$.
	Moreover, $\Phi$ is invertible and
	\(
		\norm \Phi
		=	
		\norm {\Phi^{-1}}
		=
		1
		.
	\)
\end{lemma}

\begin{proof}

We split the proof in several steps starting with the implied correctness of the diagram \eqref{Bochner_commutative_diagram}.
\begin{enumerate}

\item

Clearly,
\(
	L^p(Y, \mu; G_0)
\)
is a closed subspace in
\(
	L^p(Y, \mu; G)
	.
\)
Hence $P$ is well defined as the corresponding quotient projection,
in particular,
\(
	\norm P = 1
	.
\)

\item
\label{proof_step2}

For any
\(
	f \in L^p(Y, \mu; G)
\)
the value $\pi (f)$ is the composition
\(
	Y \xrightarrow{f} G \xrightarrow{\pi_0} G / G_0
	.
\)
Here $G$ and $G / G_0$ are endowed with the Borel $\sigma$-algebras,
whereas $f$ is measurable and $\pi_0$ is continuous.
Therefore $\pi (f)$ is measurable.
Moreover, $\pi$ is linear with $\norm \pi \leqslant 1$.
Indeed,
\[
	\norm{
		\pi f	
	}_{ L^p(Y, \mu; G / G_0) }^p
	=
	\int_Y
	\norm{
		\pi_0 f(y)	
	}_{ G / G_0 }^p
	d \mu(y)
	\leqslant
	\int_Y
	\norm{
		f(y)	
	}_{ G }^p
	d \mu(y)
	=
	\norm{
		f	
	}_{ L^p(Y, \mu; G) }^p
	.
\]

\item
\label{proof_step3}

An important claim implied in the statement is surjectivity of $\pi$.
Let
\(
	h \in L^p(Y, \mu; G / G_0)
	.
\)
We need to find an
\(
	f \in L^p(Y, \mu; G)
\)
such that $\pi f = h$.
A direct pointwise construction may not necessarily lead even to a measurable function $f$ on $Y$.
So we approximate $h$ by a sequence $h_n, n \in \mathbb N,$ of simple functions in
\(
	L^p(Y, \mu; G / G_0)
	.
\)
Then for any $y \in Y$ and $n \in \mathbb N$ there is $f_n(y) \in G$ such that
\(
	\pi_0( f_n(y) ) = h_n(y)
\)
and $f_n$ is simple.
Note that $f_n$ may not converge anywhere.
Let us choose a subsequence in $\{ h_n \}$, while keeping the same notation,
such that
\[
	\sum_{n = 1}^{\infty}
	\norm{
		h_{n + 1} - h_n
	}_{ L^p(Y, \mu; G / G_0) }
	<
	\infty
	.
\]
Making use of the fact that $\mu$ is $\sigma$-finite,
that is
\(
	Y = \bigcup \limits _{k = 1}^{\infty} Y_k
\)
with disjoint $Y_k$ having $\mu(Y_k) < \infty$,
we can introduce a measurable function
\(
	\chi : Y \to (0, \infty)
\)
taking at most countably many values
(in fact we need below that it is constant on each $Y_k$)
and normalized by
\(
	\int \chi d \mu = 1
	.
\)
Pointwisely, we have
\[
	\norm{
		h_{n + 1}(y) - h_n(y)
	}_{ G / G_0 }
	=
	\inf
	\left \{
		\norm F_G
		\colon
		F - ( f_{n + 1}(y) - f_n(y) ) \in G_0
	\right \}
	.
\]
Now let
\(
	\left \{
		Y_{l_n}^n
	\right \}
\)
be the partition associated with $h_n$.
Then for each $n \in \mathbb N$ we can approximate the infimum on the finest partition
\(
	\left \{
		Y_{l_{n + 1}}^{n + 1}
		\bigcap
		Y_{l_n}^n
		\bigcap
		Y_k
	\right \}
\)
of $Y$ as follows
\[
	\norm{
		h_{n + 1}(y) - h_n(y)
	}_{ G / G_0 }^p
	+
	\frac {\chi(y)}{ 2^{pn} }
	>
	\norm {F_n(y)}_G^p
	,
\]
where
\(
	F_n(y) - ( f_{n + 1}(y) - f_n(y) ) \in G_0
\)
and $F_n$ is constant on each
\(
	Y_{l_{n + 1}}^{n + 1}
	\bigcap
	Y_{l_n}^n
	\bigcap
	Y_k
	.
\)
One sets $F_0 = f_1$ and so $\pi F_0 = h_1$, in particular.
Integrating the above inequality one obtains the estimate
\[
	\norm{
		h_{n + 1} - h_n
	}_{ L^p(Y, \mu; G / G_0) }
	+
	\frac 1{ 2^n }
	\geqslant
	\norm {F_n}_{ L^p(Y, \mu; G) }
	,
\]
implying
\[
	\sum_{n = 1}^{\infty}
	\norm {F_n}_{ L^p(Y, \mu; G) }
	\leqslant
	\sum_{n = 1}^{\infty}
	\norm{
		h_{n + 1} - h_n
	}_{ L^p(Y, \mu; G / G_0) }
	+ 1 < \infty
	.
\]
Thus
\(
	\sum_{n = 0}^{\infty}
	F_n
\)
converges in
\(
	L^p(Y, \mu; G)
\)
to an $F$ satisfying $\pi F = h$.
Indeed,
\[
	\pi
	\left(
		\sum_{n = 0}^N
		F_n
	\right)
	=
	\sum_{n = 1}^N
	( h_{n + 1} - h_n ) + h_1
	=
	h_{N + 1}
	\to
	h
	\ \text{ as } \
	N \to \infty
	.
\]
Therefore $\pi$ is surjective.

\item
\label{proof_step4}

Finally, we can define the linear mappings $\Phi$ and $\Psi$ in the diagram \eqref{Bochner_commutative_diagram}.
Indeed, by the previous step for any
\(
	h \in L^p(Y, \mu; G / G_0)
\)
there is an
\(
	f_h \in L^p(Y, \mu; G)
\)
such that $\pi f_h = h$.
We set
\(
	\Phi(h) = P f_h
	.
\)
If we have two elements such that
\(
	\pi f_h^1
	=
	\pi f_h^2
	=
	h
\)
then
\[
	\int_Y
	\norm{
		\pi_0 f_h^1(y)
		-
		\pi_0 f_h^2(y)
	}_{ G / G_0 }^p
	d \mu(y)
	=
	\norm{
		\pi f_h^1 - \pi f_h^2	
	}_{ L^p(Y, \mu; G / G_0) }^p
	=
	0
	.
\]
Hence for a.e. $y$ the difference
\(
	f_h^1(y)
	-
	f_h^2(y)
	\in G_0
\)
and so
\(
	f_h^1
	-
	f_h^2
	\in L^p(Y, \mu; G_0)
\)
implying the equality
\(
	Pf_h^1
	=
	Pf_h^2
\)
in $L^p(Y, \mu; G) / L^p(Y, \mu; G_0)$.
Thus $\Phi$ is a well defined linear operator.

Similarly, for any
\(
	h \in L^p(Y, \mu; G) / L^p(Y, \mu; G_0)
\)
there is obviously an
\(
	f_h \in L^p(Y, \mu; G)
\)
such that $P f_h = h$.
We can set
\(
	\Psi(h) = \pi f_h
	.
\)
If we have two elements satisfying
\(
	P f_h^1
	=
	P f_h^2
	=
	h
\)
then their difference
\(
	f_h^1
	-
	f_h^2
	\in
	L^p(Y, \mu; G_0)
	.
\)
Hence for a.e. $y$ the difference
\(
	f_h^1(y)
	-
	f_h^2(y)
	\in G_0
\)
and so
\[
	\left(
		\pi f_h^1
		-
		\pi f_h^2
	\right)
	(y)	
	=
	\pi_0
	\left(
		f_h^1(y)
		-
		f_h^2(y)
	\right)
	=
	0
	\ \text{ in } \
	G / G_0
\]
implying the equality
\(
	\pi f_h^1
	=
	\pi f_h^2
\)
in $L^p(Y, \mu; G / G_0)$.
Thus $\Psi$ is well defined as well.

\item

Clearly, so defined $\Phi$ and $\Psi$ are the only linear operators
making the diagram \eqref{Bochner_commutative_diagram} commutative.
Moreover, it is straightforward to check that the compositions
$\Psi \Phi$ and $\Phi \Psi$ are identities in the spaces
\(
	L^p(Y, \mu; G / G_0)
\)
and
\(
	L^p(Y, \mu; G) / L^p(Y, \mu; G_0)
	,
\)
respectively.
Therefore $\Psi = \Phi^{-1}$.

\item
\label{proof_step6}

In this step we prove the bound $\norm \Psi \leqslant 1$
which in turn will automatically imply that $\Phi$ is bounded as well and $\norm \Phi \geqslant 1$.
Let
\(
	h \in L^p(Y, \mu; G) / L^p(Y, \mu; G_0)
\)
then by steps \ref{proof_step2}, \ref{proof_step4}
we have
\[
	\norm{
		\Psi h
	}_{ L^p(Y, \mu; G / G_0) }
	=
	\norm{
		\pi f_h
	}_{ L^p(Y, \mu; G / G_0) }
	\leqslant
	\norm{
		f_h
	}_{ L^p(Y, \mu; G) }
	.
\]
In other words,
\(
	\norm{
		\Psi h
	}_{ L^p(Y, \mu; G / G_0) }
	\leqslant
	\norm{
		f
	}_{ L^p(Y, \mu; G) }
\)
for every $f$ belonging to the preimage $P^{-1} \{ h \}$.
Therefore passing to infimum over this preimage one recovers
the quotient norm of $h$,
so that
\(
	\norm{
		\Psi h
	}_{ L^p(Y, \mu; G / G_0) }
	\leqslant
	\norm{
		h
	}_{ L^p(Y, \mu; G) / L^p(Y, \mu; G_0) }
	,
\)
which proves the claim $\norm \Psi \leqslant 1$ due to the arbitrary choice of $h$.

\item

In order to complete the proof of the lemma it is only left to get the bound $\norm \Phi \leqslant 1$.
Here we need to be careful again about measurability of pointwise quotient norm approximations.
So we will check that the bound
\(
	\norm{
		\Phi h
	}_{ L^p(Y, \mu; G) / L^p(Y, \mu; G_0) }
	\leqslant
	\norm{
		h
	}_{ L^p(Y, \mu; G / G_0) }
\)
holds for every simple $h$.
It is enough since $\Phi$ is already known to be bounded, by the previous step.

Similarly to the previous step, from the identity $\norm P = 1$ one deduces
\begin{equation}
\label{proof_step7_bound}
	\norm{
		\Phi h
	}_{ L^p(Y, \mu; G) / L^p(Y, \mu; G_0) }
	\leqslant
	\norm{
		f
	}_{ L^p(Y, \mu; G) }
	, \quad
	f \in \pi^{-1} \{ h \}
	.
\end{equation}
Let $\varepsilon > 0$ and function $\chi$ be as in step \ref{proof_step3}.
Then
\[
	\norm{
		h
	}_{ L^p(Y, \mu; G / G_0) }^p
	+
	\varepsilon
	=
	\int_Y
	\left(
		\inf_{ g \in \pi_0^{-1} \{ h(y) \} } \norm g _G^p
		+
		\varepsilon \chi(y)
	\right)
	d \mu(y)
	,
\]
and so there exists $g_{\varepsilon} : Y \to G$ such that
\[
	\inf_{ g \in \pi_0^{-1} \{ h(y) \} } \norm g _G^p
	+
	\varepsilon \chi(y)
	\geqslant
	\norm {g_{\varepsilon}(y)} _G^p
	, \quad
	\pi_0( g_{\varepsilon}(y) ) = h(y)
\]
and it takes at most countably many values,
its partition being the finest one for $h$ and $\chi$.
In particular,
\(
	g_{\varepsilon} \in L^p(Y, \mu; G)
\)
and
\(
	\pi ( g_{\varepsilon} )  = h
	,
\)
which because of \eqref{proof_step7_bound} implies
\begin{equation*}
	\norm{
		\Phi h
	}_{ L^p(Y, \mu; G) / L^p(Y, \mu; G_0) }^p
	\leqslant
	\norm{
		g_{\varepsilon}
	}_{ L^p(Y, \mu; G) }^p
	\leqslant
	\norm{
		h
	}_{ L^p(Y, \mu; G / G_0) }^p
	+
	\varepsilon
	.
\end{equation*}
Passing to the limit one obtains
\(
	\norm{
		\Phi h
	}_{ L^p(Y, \mu; G) / L^p(Y, \mu; G_0) }
	\leqslant
	\norm{
		h
	}_{ L^p(Y, \mu; G / G_0) },
\)
completing the proof.

\end{enumerate}

\end{proof}

We formulate the following simple lemma, that can be either found in \cite{Helemskii_lectures_2006}
or easily proved with the help of an argument similar to one used in step \ref{proof_step6} of the previous proof.

\begin{lemma}
	Consider nontrivial closed subspaces $E_0, F_0$ in normed spaces $E, F$, respectively.
	Let $\mathcal T : E \to F$ be a linear bounded operator satisfying $\mathcal T (E_0) \subset F_0$.
	Then there exists a unique linear $\widetilde \mathcal T$ making the following diagram commutative:
	\begin{equation}
	\label{quotient_commutative_diagram}
	\begin{tikzcd}
    	E
    	\arrow[r, "\mathcal T"]
    	\arrow[d, " P_1 \;\; "']
    	&
    	F
		\arrow[d, " \;\; P_2 "]
    	\\
    	E / E_0
		\arrow[r, "\widetilde \mathcal T", dashed]
		&
		F / F_0
 	\end{tikzcd}
	\end{equation}
	Moreover,
	\(
		\norm{ \widetilde \mathcal T }
		\leqslant
		\norm{ \mathcal T }
		.
	\)
\end{lemma}

We take
\(
	E = X_{ h(\xi) }^{s, b}
	,
	Y = \mathbb R^d
	,
	d \mu(\xi) = \langle \xi \rangle^{2s} d \xi
	,
	G = H^b( \mathbb R )
\)
and
\(
	F = L^2
	\left(
		\mathbb R^d
		,
		\langle \xi \rangle^{2s} d \xi
		;
		H^b( \mathbb R )
	\right)
	.
\)
We introduce $\mathcal T$ as the isometric extension of
\(
	e^{ ith(\xi) } \mathcal F_x
\)
defined on the Schwartz space
$\mathcal S \left( \mathbb R^{d + 1} \right)$
with the property \eqref{XNorm_alternative},
so that
\begin{equation}
\label{abstract_XNorm_alternative}
	\norm
	{
		u
	}
	_{ X_{h(\xi)}^{ s, b } }^2
	=
	\int
	\langle \xi \rangle ^{2s}
	\norm
    {
        \mathcal T u(\xi)
    } _{ H^b(\R) }^2
	d\xi
	.
\end{equation}
Now let us consider an interval $I = (S, T)$
and the closed subspaces
\(
	G_0
	=
	\left \{
		u \in H^b( \mathbb R )
		\colon
		u = 0 \text{ on } I
	\right \}
	,
\)
\(
	E_0
	=
	\left \{
		u \in X_{h(\xi)}^{ s, b }
		\colon
		u = 0 \text{ on } I \times \mathbb R^d
	\right \}
	.
\)
It is known that
\(
	H^b(I)
	=
	H^b( \mathbb R ) / G_0
\)
and
\(
	X_{h(\xi)}^{ s, b }(I)
	=
	X_{h(\xi)}^{ s, b } / E_0
\)
endowed with the quotient norms, of course.
We claim that there exists a unique invertible isometry $\widetilde \mathcal T$ making the following diagram commutative
\begin{equation}
\label{Bourgain_isometry_diagram}
\begin{tikzcd}
   	X_{h(\xi)}^{ s, b }
   	\arrow
    [
        r, " \mathcal T = e^{ ith(\xi) } \mathcal F_x "
        ,
        yshift = 0.7ex
    ]
   	\arrow
    [
        d, " P_1 \;\; "'
    ]
    &
    L^2
    \left(
		\mathbb R^d
		,
		\langle \xi \rangle^{2s} d \xi
		;
		H^b( \mathbb R )
	\right)
   	\arrow[
        l, " \mathcal T^{-1} = \mathcal F_{\xi}^{-1} e^{ -ith(\xi) } "
        ,
        yshift = -0.7ex
    ]
    \arrow[r, " \pi "]
    \arrow[d, " P_2 \;\; "']
    &
    L^2
    \left(
		\mathbb R^d
		,
		\langle \xi \rangle^{2s} d \xi
		;
		H^b( I )
	\right)
	\arrow
    [
        dl, "\Phi"
        ,
        bend left=10
        ,
        dashed
    ]
	\\
	X_{h(\xi)}^{ s, b }(I)	
	\arrow[
        r, "\widetilde \mathcal T", dashed
        ,
        yshift = 0.7ex
    ]
	&
    L^2
	\left(
		\mathbb R^d
		,
		\langle \xi \rangle^{2s} d \xi
		;
		H^b( \mathbb R )
	\right)
    /
    L^2
	\left(
		\mathbb R^d
		,
		\langle \xi \rangle^{2s} d \xi
		;
		G_0
	\right)
	\arrow[
        ur, "\Phi^{-1}"
        ,
        bend right = 2
        ,
        dashed
    ]
	\arrow[
        l, "\widetilde \mathcal T^{-1}", dashed
        ,
        yshift = -0.7ex
    ]
\end{tikzcd}
\end{equation}
This commutative diagram makes a complete sense of and proves the identity
\eqref{strong_XRestNorm_alternative}.
In other words, the term
\(
    e^{ it h(\xi) }
    \mathcal F_x
\)
staying in \eqref{strong_XRestNorm_alternative}
should be understood as the composition
\(
    \Phi^{-1} \widetilde \mathcal T
\)
having operator norm equal to unity,
namely,
\begin{equation}
\label{abstract_XRestNorm_alternative}
	\norm
	{
		u
	}
	_{ X_{h(\xi)}^{ s, b }(S,T) }
	=
	\left( \int
	\langle \xi \rangle ^{2s}
	\norm
    {
        \Phi^{-1} \widetilde \mathcal T
        u(\xi)
    } _{ H^b (S,T)}^2
	d\xi \right)^{1/2}
	.
\end{equation}
In order to appeal to the previous two lemmas, and to prove the claim,
one needs only to show the inclusions
\[
    \mathcal T(E_0)
    \subset
    L^2
    \left(
		\mathbb R^d
		,
		\langle \xi \rangle^{2s} d \xi
		;
		G_0
	\right)
    , \quad
    \mathcal T^{-1}
    \left(
        L^2
        \left(
            \mathbb R^d
            ,
            \langle \xi \rangle^{2s} d \xi
            ;
            G_0
        \right)
	\right)
    \subset
    E_0
    .
\]
Moreover, it is enough to check these on the smooth functions,
namely, that the following hold
\[
    \mathcal T
    \left(
        E_0
        \bigcap
        \left \{
            \phi \psi
            \colon
            \phi \in \mathcal S \left( \mathbb R^d \right)
            ,
            \psi \in \mathcal S ( \mathbb R )
        \right \}
	\right)
    \subset
    L^2
    \left(
		\mathbb R^d
		,
		\langle \xi \rangle^{2s} d \xi
		;
		G_0
	\right)
    ,
\]
\[
    \mathcal T^{-1}
    \left(
        \left \{
            \phi \psi
            \colon
            \phi \in \mathcal S \left( \mathbb R^d \right)
            ,
            \psi \in \mathcal S ( \mathbb R ) \bigcap G_0
        \right \}
	\right)
    \subset
    E_0
    .
\]
This is obvious and so \eqref{Bourgain_isometry_diagram}
is commutative.
Therefore
\eqref{abstract_XNorm_alternative}, \eqref{abstract_XRestNorm_alternative}
and
\eqref{XNorm_alternative}, \eqref{strong_XRestNorm_alternative}
are fully justified.

\section{A modified Bourgain norm}\label{Mod-B-norm}
\setcounter{equation}{0}

In preparation for the proof of the cutoff estimates, we now investigate more closely the modified Bourgain norm defined in \eqref{Bourgain_Slobodeckij_norm}, and derive some of its key properties.

Fix $s \in \R$, $b \in (0,1/2)$ and $h \in C(\R^d,\R)$. Write the norm \eqref{Bourgain-norm} as
\[
  \norm{u}_{X^{s,b}}
  = \left( \int_{\R^d} \norm{U(t,\xi)}_{H^b_t(\R)}^2 \, d\xi \right)^{1/2},
\]
where the transform $u(t,x) \mapsto U(t,\xi)$ is defined by
\[
  U(t,\xi) = \japanese{\xi}^{s} e^{ith(\xi)} \widehat u(t,\xi).
\]
By \eqref{HbEquiv}, the above norm is equivalent to, with constants depending only on $b$, the norm
\[
  \norm{u}_{\widehat X^{s,b}}
  := \left( \int_{\R^d} \left( \norm{U(t,\xi)}_{L^2_t(\R)}^2 
  + \norm{U(t,\xi)}_{S^b_t(\R)}^2 \right) \, d\xi \right)^{1/2}.
\]
Inserting here the characteristic function $\mathbb 1_{(S,T)}(t)$ of a time interval $(S,T)$, we compute
\begin{equation}
\label{CutoffNorm'}
  \norm{\mathbb 1_{(S,T)} u}_{\widehat X^{s,b}}
  =
  \left( \int_{\R^d} \left( \int_S^T \Abs{U(t,\xi)}^2 \left( 1 + \frac{1}{b(t-S)^{2b}} + \frac{1}{b(T-t)^{2b}} \right) \, dt
  + \norm{U(t,\xi)}_{S^b_t(S,T)}^2 \right) \, d\xi \right)^{1/2}.
\end{equation}
By \eqref{X_characteristic_norm}, \eqref{strong_XRestNorm_alternative}
the latter is equivalent to, with constants depending only on $b$,
the restriction norm
\begin{equation}\label{RestNorm}
  \norm{u}_{X^{s,b}(S,T)}
  = \left( \int_{\R^d} \norm{U(t,\xi)}_{H^b_t(S,T)}^2 \, d\xi \right)^{1/2}.
\end{equation}

The advantage of the norm \eqref{CutoffNorm'} is that it has an explicit expression; there are no restriction norms involved. It is still a bit tricky to work with, however. We will use instead the simpler, modified norm
\begin{equation}\label{CutoffNorm''}
  \norm{u}_{\widetilde X^{s,b}(S,T)} =
  \left( \int_{\R^d} 
  \left( \frac 1{ (T - S)^{2b} }
        \int_S^T \Abs{ U(t, \xi) }^2 \, dt
        +
        \int_S^T \int_S^T
        \frac
        { | U(t, \xi) - U(r, \xi) |^2 }
        { | t - r |^{1 + 2b} }
        \, dr \, dt \right) \, d\xi \right)^{1/2},
\end{equation}
which is the same as \eqref{Bourgain_Slobodeckij_norm}, and turns out to be equivalent to the two previous norms.

In fact, we have the following.

\begin{lemma}\label{NormLemma}
Let $T_0 > 0$ and $0 < b < 1/2$. Consider an interval $(S,T)$ with length at most $T_0$. Then the norms defined on $X^{s,b}(S,T)$ by \eqref{CutoffNorm'}, \eqref{RestNorm} and \eqref{CutoffNorm''} are pairwise equivalent, with constants depending only on $b$ and $T_0$.

Moreover, for all $T > 0$ we have the estimate
\begin{equation}\label{Large-T-estimate}
  \norm{u}_{X^{s,b}(0,T)} \ge C \norm{u}_{\widetilde X^{s,b}(0,T)},
\end{equation}
where $C > 0$ depends only on $b$.
\end{lemma}

\begin{proof} By translation invariance we may consider the interval $(0,T)$, where $0 < T \le T_0$. By \eqref{X_characteristic_norm} we already know that \eqref{CutoffNorm'} and \eqref{RestNorm} are equivalent, uniformly in $T$ (of any size). The equivalence of \eqref{CutoffNorm'} and \eqref{CutoffNorm''} reduces to proving the equivalence of the following norms on $H^b(0,T)$:
\[
  M_T(\phi) = \left( \int_0^T \Abs{\phi(t)}^2 \left( 1 + \frac{1}{bt^{2b}} + \frac{1}{b(T-t)^{2b}} \right) \, dt
  + \int_0^T \int_0^T \frac{\Abs{\phi(t)-\phi(r)}^2}{\Abs{ t - r }^{1+2b} } \, dr \, dt \right)^{1/2}
\]
and
\[
  N_T(\phi) = \left( T^{-2b} \int_0^T \Abs{\phi(t)}^2  \, dt
  + \int_0^T \int_0^T \frac{\Abs{\phi(t)-\phi(r)}^2}{\Abs{ t - r }^{1+2b} } \, dr \, dt \right)^{1/2}.
\]
First, since $0 < t < T$ implies $T^{-2b} < t^{-2b}$, it is clear that
\begin{equation} \label{MNineq}
  N_T(\phi)
  \le b^{1/2} M_T(\phi) \quad \text{for all $T > 0$}.
\end{equation}
It remains to show
\[
  M_T(\phi) \le C N_T(\phi).
\]
We claim that this holds for $T=T_0$. Granting this for the moment, it follows that the inequality holds also for $0 < T \le T_0$, since setting $g(s) = \phi\left(sT/T_0\right)$ and rescaling yields
\[
  M_T(\phi)^2
  \le \left(\frac{T}{T_0}\right)^{1-2b} M_{T_0}(g)^2 \le C^2 \left(\frac{T}{T_0}\right)^{1-2b} N_{T_0}(g)^2 = C^2 N_T(\phi)^2.
\]
It remains to prove the claim, namely $M_{T_0}(g) \le C N_{T_0}(g)$ for $g \in H^b(0,T_0)$, with $C$ depending on $T_0$ and $b$. But on the one hand, \eqref{HbEquiv} implies $M_{T_0}(g) = \norm{\mathbb 1_{(0,T_0)} g}_{H^b(\R)}$. On the other hand, $N_{T_0}(g) \sim_{b,T_0}  \norm{g}_{H^b(0,T_0)}$ by \eqref{HbEquivI}. By \eqref{H_characteristic_norm} it now follows that $M_{T_0}(g) \sim_{b,T_0} N_{T_0}(g)$.

Finally, \eqref{Large-T-estimate} is immediate from \eqref{MNineq}. This completes the proof of the lemma.
\end{proof}

Before proceeding with the proof of the cutoff estimates, we mention some properties of the modified Bourgain norm \eqref{CutoffNorm''}.

\begin{lemma}\label{Mod-rest-lemma}
Let $T_0 > 0$ and $0 < b < 1/2$. Then for all $0 < T \le T_0$ and all $u \in X^{s,b}(0,T)$ we have the bounds
\begin{equation}\label{increase2}
  \norm{u}_{\widetilde X^{s,b}(r,t)} \le C \norm{u}_{\widetilde X^{s,b}(0,T)} \quad \text{for} \quad 0 \le r < t \le T
\end{equation}
and
\begin{equation}\label{triangle}
  \norm{u}_{\widetilde X^{s,b}(0,t)}
  \le
  C \left( \norm{u}_{\widetilde X^{s,b}(0,r)} + \norm{u}_{\widetilde X^{s,b}(r,t)} \right) \quad \text{for} \quad 0 \le r < t \le T,
\end{equation}
where the constants only depend on $b$ and $T_0$.
\end{lemma}

\begin{proof} We apply Lemma \ref{NormLemma}. First, \eqref{increase2} follows via the equivalence of \eqref{CutoffNorm''} with the restriction norm \eqref{RestNorm}, and the fact that the latter is increasing with respect to the interval to which we restrict. Second, \eqref{triangle} follows via the equivalence of \eqref{CutoffNorm''} with the sharp cutoff norm \eqref{CutoffNorm'}, writing $\mathbb 1_{(0,t)} = \mathbb 1_{(0,r)} + \mathbb 1_{(r,t)}$ (a.e.) and applying the triangle inequality.
\end{proof}

\begin{lemma}\label{X-cont-lemma}
Let $0 < b < 1/2$. Assume that $u \in X^{s,b}(0,T)$ and set
\[
  f(t) = \norm{u}_{\widetilde X^{s,b}(0,t)}^2 \quad \text{for $0 < t \le T$}.
\]
Then $f$ is continuous on $(0,T]$. Moreover, if we additionally assume that $u \in C([0,T], H^s)$, then
\[
  \lim_{t \searrow 0} f(t) = 0,
\]
so $f$ extends to a continuous function on $[0,T]$.
\end{lemma}

\begin{proof}
For $0 < t \le T$,
\begin{equation}\label{f-def}
  f(t)
  =
  t^{-2b} \int_0^t \norm{u(r)}_{H^s}^2 \, dr 
  +
  \int_0^t\int_0^t \frac{\norm{U(r,\xi)-U(\sigma,\xi)}_{L^2_\xi}^2}{\Abs{r-\sigma}^{1+2b}} \, dr \, d\sigma.
\end{equation}
Continuity follows from the dominated convergence theorem. It remains to prove $\lim_{t \searrow 0} f(t) = 0$. For the second term in \eqref{f-def}, this is clear by dominated convergence, and the first term we bound by
\[
  t^{-2b}  \int_0^t \norm{u(r)}_{H^s}^2  \, dr
  \le
  t^{1-2b} \sup_{0 \le r \le t} \norm{u(r)}_{H^s}^2,
\]
which tends to zero as $t \searrow 0$ if $u \in C([0,T], H^s)$.
\end{proof}

\begin{lemma}\label{X-adapted-lemma}
Assume that $u \in X^{s,b}(0,T) \cap C([0,T], H^s)$ is an $H^s$-adapted random variable. Then the continuous function $f \colon [0,T] \to [0,\infty)$ from Lemma \ref{X-cont-lemma} is adapted.
\end{lemma}

\begin{proof}
For both terms in \eqref{f-def}, $\mathcal F_t$-measurability follows by Tonelli's theorem, since $u(t)$ is progressively measurable.
\end{proof}

\section{Cutoff estimates in $H^b$ and $X^{s,b}$}\label{CutoffSection}
\setcounter{equation}{0}

Here we prove Proposition \ref{MainCutoffLemma}, which we restate below for convenience. Let $0 < b < 1/2$. Fix a smooth, compactly supported function $\theta \colon \R \to \R$, and write $\theta_R(x) = \theta(x/R)$.

Via the transform $u(t,x) \mapsto U(t,\xi) = \japanese{\xi}^{s} e^{ith(\xi)} \widehat u(t,\xi)$, and recalling Lemma \ref{NormLemma}, we can identify $X^{s,b}_{h(\xi)}(S,T)$ with the space $L^2(\R^d,H^b(S,T))$ with norm (this corresponds to \eqref{CutoffNorm''})
\[
  \norm{U}_{(S,T)}
  =
  \norm{
  \left(
  \frac 1{ (T - S)^{2b} }
        \int_S^T \Abs{ U(t, \xi) }^2 \, dt
        +
        \int_S^T \int_S^T
        \frac
        { | U(t, \xi) - U(r, \xi) |^2 }
        { | t - r |^{1 + 2b} }
        \, dr \, dt
  \right)^{1/2}
  }_{L^2_\xi}.
\]
If $S=0$, we simply write $\norm{U}_T$. This norm is associated to the inner product
\[
  \innerprod{U}{V}_{T}
  =
  \int_{\R^d} \left( \frac{1}{T^{2b}} \int_0^T U(t,\xi) \overline{V(t,\xi)} \, dt
  +
  \int_0^T \int_0^T \frac{\left[U(t,\xi)-U(r,\xi)\right] \overline{\left[V(t,\xi)-V(r,\xi)\right]}}{\Abs{t-r}^{1+2b}} \, dr \, dt \right) \, d\xi.
\]
For a vector $\mathbf U=(U_1,\dots,U_n)$ we write
\[
  \norm{\mathbf U}_{T} = \left( \sum_{i=1}^n \norm{U_i}_{T}^2 \right)^{1/2}.
\]

With this notation, and taking into account Lemma \ref{NormLemma}, we can now restate Proposition \ref{MainCutoffLemma} as follows.

\begin{proposition}
\label{MainCutoffLemma-2}
   Let $T_0 > 0$. Then for all $T \in (0, T_0]$ and $R > 0$ we have the estimates
    \begin{equation}\label{MainCutoffLemma-2-1}
        \norm{
            \theta_R\left( \norm{\mathbf U}_{t}^2 \right) \mathbf U
        }_{T}
        \leqslant
        C \sqrt{R}
        ,
    \end{equation}
    \begin{equation}\label{MainCutoffLemma-2-2}
        \norm{
            \theta_R\left( \norm{\mathbf U}_{t}^2 \right) \mathbf U
            -
            \theta_R\left( \norm{\mathbf V}_{t}^2 \right) \mathbf V
        }_{T}
        \leqslant
        C \norm{\mathbf U - \mathbf V}_{T},
    \end{equation}
    where the constant $C$ depends only on $b$, $T_0$ and $\theta$.
    \end{proposition}

The remainder of this section is devoted to the proof of this result. For convenience, and without loss of generality, we assume that $\Abs{\theta} \le 1$ and $\theta$ is supported in $[-2,2]$. Throughout we assume that $T_0 > 0$, $b \in (0,1/2)$ and that functions $U,V$ etc.~ are in $L^2(\R^d,H^b(S,T))$, and similarly for vectors $\mathbf U,\mathbf V$ etc.

We note the bound, for all $x,y \in \R$, 
\begin{equation}\label{Trivial-theta-bound}
  \Abs{\theta_R(x) - \theta_R(y)} \le \frac{C}{R} \Abs{x-y}.
\end{equation}
This holds with $C=\norm{\theta'}_{L^\infty}$.

We first prove some preliminary lemmas.

\subsection{Preliminary estimates}

\begin{lemma}\label{LemmaKeyEst1} For all $T \in (0,T_0]$ and $R > 0$ we have the estimate
\begin{equation}\label{KeyEst1}
  \int_{\R^d} \int_0^T \Abs{U(t,\xi)}^2 \int_0^t \frac{\Abs{ \Innerprod{V}{W}_t - \innerprod{V}{W}_r  }^2}{(t-r)^{1+2b}} \, dr \, dt \, d\xi
  \le C \norm{U}_T^2 \norm{V}_T^2 \norm{W}_T^2,
\end{equation}
where $C$ depends only on $b$ and $T_0$.
\end{lemma}

\begin{proof}
We write
\[
  \innerprod{V}{W}_t = \mu(t) + \nu(t),
\]
where
\begin{align*}
  \mu(t) &= t^{-2b} \int_{\R^d} \int_0^t V(s,\zeta) \overline{W(s,\zeta)} \, ds \, d\zeta,
  \\
  \nu(t) &= 2 \int_{\R^d} \int_0^t \int_0^s \frac{\left[V(s,\zeta)-V(\sigma,\zeta)\right] \overline{\left[W(s,\zeta)-W(\sigma,\zeta)\right]}}{(s-\sigma)^{1+2b}} \, d\sigma \, ds \, d\zeta.
\end{align*}
The left side of \eqref{KeyEst1} is therefore bounded by $2(I+J)$, where
\begin{align*}
  I &= \int_{\R^d} \int_0^T \Abs{U(t,\xi)}^2 \int_0^t \frac{\Abs{\mu(t) - \mu(r)}^2}{(t-r)^{1+2b}} \, dr \, dt \, d\xi,
  \\
  J &= \int_{\R^d} \int_0^T \Abs{U(t,\xi)}^2 \int_0^t \frac{\Abs{\nu(t) - \nu(r)}^2}{(t-r)^{1+2b}} \, dr \, dt \, d\xi.
\end{align*}

To estimate $J$ we note that
\[
  \nu(t) - \nu(r) = 2 \int_{\R^d} \int_r^t \int_0^s \frac{\left[V(s,\zeta)-V(\sigma,\zeta)\right] \overline{\left[W(s,\zeta)-W(\sigma,\zeta)\right]}}{(s-\sigma)^{1+2b}}  \, d\sigma \, ds \, d\zeta,
\]
so by Cauchy-Schwarz,
\begin{multline*}
  \Abs{\nu(t) - \nu(r)}^2
  \le
  2 \int_{\R^d} \int_r^t \int_0^s \frac{\Abs{V(s,\zeta)-V(\sigma,\zeta)}^2}{(s-\sigma)^{1+2b}} \, d\sigma \, ds \, d\zeta
  \\
  \times 2 \int_{\R^d} \int_r^t \int_0^s \frac{\Abs{W(s,\zeta)-W(\sigma,\zeta)}^2}{(s-\sigma)^{1+2b}} \, d\sigma \, ds \, d\zeta,
\end{multline*}
where the last line is bounded by $\norm{W}_T^2$. Thus
\begin{multline}\label{I2bound}
  J \le \norm{W}_T^2
  2 \int_{\R^d} \int_0^T \Abs{U(t,\xi)}^2 \int_0^t \frac{1}{(t-r)^{1+2b}}
  \\
  \times \left( \int_{\R^d} \int_r^t \int_0^s \frac{\Abs{V(s,\zeta)-V(\sigma,\zeta)}^2}{(s-\sigma)^{1+2b}} \, d\sigma \, ds \, d\zeta \right) \, dr \, dt \, d\xi,
\end{multline}
and Lemma \ref{LemmaKeyEst2} below yields the desired estimate $J \le C \norm{U}_T^2 \norm{V}_T^2 \norm{W}_T^2$.

In $I$ we split the innermost integral as $\int_0^{t/2} + \int_{t/2}^t$ and write $I = I_1 + I_2$
accordingly. The term $I_1$ is easy to handle since $t-r \sim t$. Applying Cauchy-Schwarz and \eqref{increase2} to bound $\Abs{\mu(t)} \le \norm{V}_t \norm{W}_t \le C \norm{V}_T \norm{W}_T$, we then simply estimate
\[
  I_1
  \le
  C \norm{V}_T^2 \norm{W}_T^2
  \left( \int_{\R^d} \int_0^T \Abs{U(t,\xi)}^2 t^{-2b} \, dt \, d\xi \right)
  \le
  C' \norm{V}_T^2 \norm{W}_T^2 \norm{U}_T^2,
\]
where we used Lemma \ref{NormLemma} in the last step.

It remains to consider $I_2$. Writing
\[
  \mu(t) = t^{1-2b} f(t), \quad \text{where} \quad f(t) = \frac{1}{t} \int_0^t g(s) \, ds,
  \quad
  g(s) = \int_{\R^d} V(s,\zeta) \overline{W(s,\zeta)} \, d\zeta,
\]
we expand
\begin{align*}
  \mu(t) - \mu(r)
  &=
  (t^{1-2b} - r^{1-2b}) f(t) + r^{1-2b} [f(t)-f(r)]
  \\
  &=
  (t^{1-2b} - r^{1-2b}) f(t) + r^{1-2b} \left[ \left( \frac{1}{t} - \frac{1}{r} \right) \int_0^r g(s) \, ds
  + \frac{1}{t} \int_r^t g(s) \, ds \right]
  \\
  &=
  \frac{t^{1-2b} - r^{1-2b}}{t^{1-2b}} \mu(t) + \frac{r-t}{t} \mu(r)
  + \frac{r^{1-2b}}{t} \int_r^t g(s) \, ds
  \\
  &=: \Delta_1 + \Delta_2 + \Delta_3.
\end{align*}
Thus $I_2 \le 3( K_1 + K_2 + K_3)$, where
\[
  K_j = \int_{\R^d} \int_0^T \Abs{U(t,\xi)}^2
  \int_{t/2}^t \frac{\Delta_j^2}{(t-r)^{1+2b}} \, dr \, dt \, d\xi \qquad (j=1,2,3).
\]

For $K_2$ we use once more the bound $\Abs{\mu(r)} \le C \norm{V}_T \norm{W}_T$ and get
\begin{multline*}
  K_2 \le C \norm{V}_T^2 \norm{W}_T^2
  \int_{\R^d} \int_0^T \Abs{U(t,\xi)}^2 t^{-2}
  \left( \int_{t/2}^t (t-r)^{1-2b} \, dr \right) \, dt \, d\xi
  \\
  \le C' \norm{V}_T^2 \norm{W}_T^2
  \int_{\R^d} \int_0^T \Abs{U(t,\xi)}^2 t^{-2b}
  \, dt \, d\xi
  \le C'' \norm{V}_T^2 \norm{W}_T^2 \norm{U}_T^2.
\end{multline*}
The same bound is obtained for $K_1$, since there $t^{1-2b} - r^{1-2b} \sim t^{-2b} (t-r)$.

For $K_3$ we bound, by Cauchy-Schwarz,
\[
  (t-r)^{-2b} \int_r^t \Abs{g(s)} \, ds \le \norm{V}_{(r,t)} \norm{W}_{(r,t)}  \le C \norm{V}_T \norm{W}_T,
\]
where the last inequality follows from \eqref{increase2}. Thus
\begin{multline*}
  K_3
  \le
  C \norm{V}_T^2 \norm{W}_T^2
  \int_{\R^d} \int_0^T \Abs{U(t,\xi)}^2 t^{-4b}
  \left( \int_{t/2}^t (t-r)^{2b-1} \, dr \right) \, dt \, d\xi
  \\
  \le C' \norm{V}_T^2 \norm{W}_T^2
  \int_{\R^d} \int_0^T \Abs{U(t,\xi)}^2 t^{-2b}
  \, dt \, d\xi
  \le C'' \norm{V}_T^2 \norm{W}_T^2 \norm{U}_T^2.
\end{multline*}
This completes the proof of the lemma.
\end{proof}

The following lemma was used in the above proof.

\begin{lemma}\label{LemmaKeyEst2}
For all $T \in (0,T_0]$ and $R > 0$  we have the estimate
\begin{multline*}
  \int_{\R^d} \int_0^T \Abs{U(t,\xi)}^2 \int_0^t \frac{1}{(t-r)^{1+2b}}
  \left( \int_{\R^d} \int_r^t \int_0^s \frac{\Abs{V(s,\zeta)-V(\sigma,\zeta)}^2}{(s-\sigma)^{1+2b}} \, d\sigma \, ds \, d\zeta \right) \, dr \, dt \, d\xi
  \\
  \le C \norm{U}_T^2 \norm{V}_T^2,
\end{multline*}
with a constant $C$ depending only on $b$ and $T_0$.
\end{lemma}

\begin{proof} Here $0 < r < s < t < T$ and $0 < \sigma < s$. Rearranging the order of the integrations, we rewrite the integral as
\begin{multline*}
  \int_{\R^d} \int_0^T \int_0^s \left( \int_{\R^d} \int_s^T \Abs{U(t,\xi)}^2 \left( \int_0^s \frac{dr}{(t-r)^{1+2b}} \right) \, dt \, d\xi \right) 
  \frac{\Abs{V(s,\zeta)-V(\sigma,\zeta)}^2}{(s-\sigma)^{1+2b}} \, d\sigma \, ds \, d\zeta
  \\
  \le
  \int_{\R^d} \int_0^T \int_0^s \left( \int_{\R^d} \int_s^T \Abs{U(t,\xi)}^2 \frac{1}{2b(t-s)^{2b}} \, dt \, d\xi \right) \frac{\Abs{V(s,\zeta)-V(\sigma,\zeta)}^2}{(s-\sigma)^{1+2b}} \, d\sigma \, ds \, d\zeta.
\end{multline*}
But by Lemma \ref{NormLemma} and \eqref{increase2},
\[
  \int_{\R^d} \int_s^T \Abs{U(t,\xi)}^2 \frac{1}{2b(t-s)^{2b}} \, dt \, d\xi
  \le C \norm{U}_{(s,T)}^2 \le C \norm{U}_T^2,
\]
and the claimed inequality then follows.
\end{proof}

We will also need the following double mean value theorem.

\begin{lemma}\label{DiffLemma}
For all $x,y,X,Y \in \R$,
\begin{multline*}
  \Abs{\theta(x) - \theta(y) - \theta(X) + \theta(Y)}
  \\
  \le \norm{\theta''}_{L^\infty}
  \min\left(\Abs{x-y},\Abs{X-Y}\right)
  \max\left(\Abs{x-X},\Abs{y-Y}\right)
  +
  \norm{\theta'}_{L^\infty}\Abs{x-y-X+Y}.
\end{multline*}
\end{lemma}

\begin{proof}
Fix $x,y,X,Y$. By symmetry we may assume $\Abs{x-y} \le \Abs{X-Y}$. Defining $\kappa(t) = y + t(x-y)$ and $\rho(t) = Y + t(X-Y)$, we write
\[
  \theta(x) - \theta(y)
  = \left( \int_0^1 \theta'\left( \kappa(t) \right) \, dt \right) (x-y) =: I_1 (x-y)
\]
and
\[
  \theta(X) - \theta(Y)
  = \left( \int_0^1 \theta'\left( \rho(t) \right) \, dt \right) (X-Y) =: I_2 (X - Y).
\]
Then
\[
  \theta(x) - \theta(y) - \theta(X) + \theta(Y)
  = (I_1-I_2)(x-y) + I_2 (x-y-X+Y).
\]
Clearly, $\Abs{I_2} \le \norm{\theta'}_{L^\infty}$, so it only remains to check that
\[
  \Abs{I_1-I_2} \le \norm{\theta''}_{L^\infty}
  \frac12 \left( \Abs{x-X} + \Abs{y-Y} \right).
\]
But this is clear since
\[
  I_1-I_2
  =
  \int_0^1
  \left ( \int_0^1
  \theta''\left( \rho(t) + s \left[ \kappa(t) - \rho(t) \right] \right) \, ds
  \right)
  \left[ \kappa(t) - \rho(t) \right] \, dt
\]
and
\[
  \Abs{\kappa(t) - \rho(t)} = \Abs{y - Y + t(x-y-X+Y)} \le (1-t)\Abs{y-Y} + t\Abs{x-X}.
\]
\end{proof}

With these preliminary results in hand, we are now ready to start the proof of Proposition \ref{MainCutoffLemma-2}. We split the argument into several steps.

\subsection{Cutoff estimate, version I}\label{Cutoff-I}

\begin{lemma}\label{CutoffLemma} For all $T \in (0,T_0]$ and $R > 0$ we have the estimate
\begin{equation}\label{cutoff}
  \norm{\theta_R\left( \norm{\mathbf U}_{t}^2 \right) V}_{T} \le C \left( 1 + R^{-1} \norm{\mathbf U}_{T}^2 \right) \norm{V}_{T},
\end{equation}
with a constant $C$ depending only on $b$ and $T_0$.
\end{lemma}

\begin{proof}
Setting $\psi(t) = \theta_R(\norm{\mathbf U}_{t}^2)$, the square of the left side of \eqref{cutoff} equals $A+2B$, where
\begin{align*}
  A &= \int_{\R^d} T^{-2b} \int_0^T \Abs{\psi(t) V(t,\xi)}^2 \, dt \, d\xi,
  \\
  B &= \int_{\R^d} \int_0^T \int_0^t \frac{\Abs{\psi(t)V(t,\xi) - \psi(r)V(r,\xi)}^2}{(t-r)^{1+2b}} \, dr \, dt \, d\xi,
\end{align*}
and $\Abs{\psi(t)} \le 1$ implies $A \le \norm{V}_T^2$. Clearly, $B \le 2(B_1 + B_2)$,
where
\begin{align*}
  B_1 &= \int_{\R^d} \int_0^T \int_0^t \Abs{\psi(r)}^2 \frac{\Abs{V(t,\xi) - V(r,\xi)}^2}{(t-r)^{1+2b}} \, dr \, dt \, d\xi,
  \\
  B_2 &= \int_{\R^d} \int_0^T \Abs{V(t,\xi)}^2 \int_0^t \frac{\Abs{\psi(t) - \psi(r)}^2}{(t-r)^{1+2b}} \, dr \, dt \, d\xi,
\end{align*}
where $\Abs{\psi(r)} \le 1$ implies $B_1 \le \norm{v_j}_T^2$. In $B_2$ we estimate, using \eqref{Trivial-theta-bound},
\begin{equation}\label{Trivial-cutoff-estimate}
    \Abs{\psi(t) - \psi(r)}
  \le \frac{C}{R} \Abs{ \norm{\mathbf U}_t^2 - \norm{\mathbf U}_r^2 },
\end{equation}
and obtain $B_2 \le R^{-2} \norm{\mathbf U}_T^4 \norm{V}_T^2$ as a consequence of Lemma \ref{LemmaKeyEst1}.
\end{proof}

\begin{remark}\label{Cutoffremark}
The estimate \eqref{cutoff}, using the equivalent norm \eqref{CutoffNorm'} instead of \eqref{CutoffNorm''}, is claimed in \cite{Bouard_Debussche2007}, but there is a gap in the proof. To explain the problem, let us denote by $|\!|\!| \mathbf U |\!|\!|_T$ the norm used in \cite{Bouard_Debussche2007}, that is, the norm given by \eqref{CutoffNorm'}:
\[
  |\!|\!| \mathbf U |\!|\!|_T = \norm{\mathbb 1_{(0,T)} \mathbf u}_{\widehat X^{s,b}}, \quad \text{where $\mathbf U(t,\xi) = \japanese{\xi}^{s} e^{ith(\xi)} \widehat \mathbf u(t,\xi)$.}
\]
Then by the triangle inequality we have, for $0 < r < t$,
\begin{equation}\label{Too-rough}
    \Abs{ |\!|\!| \mathbf U |\!|\!|_t - |\!|\!| \mathbf U |\!|\!|_r } \le \norm{\mathbb 1_{(r,t)} \mathbf u}_{\widehat X^{s,b}}
    = |\!|\!| \mathbf U |\!|\!|_{(r,t)}.
\end{equation}
Combining \eqref{Too-rough} with the analogue of \eqref{Trivial-cutoff-estimate} for the norm $|\!|\!| \cdot |\!|\!|_t$ yields
\begin{equation}\label{Trivial-cutoff-estimate-2}
  \Abs{\psi(t) - \psi(r)}
  \le \frac{C}{R} |\!|\!| \mathbf U |\!|\!|_{(r,t)} |\!|\!| \mathbf U |\!|\!|_t,
\end{equation}
which is essentially what was used in \cite{Bouard_Debussche2007}, instead of \eqref{Trivial-cutoff-estimate}. But it is easy to see that \eqref{Trivial-cutoff-estimate-2} is not enough to prove the estimate for $B_2$. Indeed, take $\mathbf U$ and $V$ both to be the function $\mathbb 1_{(0,T)}(t) f(\xi)$, where $0 < T < 1$ and $\norm{f}_{L^2} = 1$. Then by \eqref{Basic-Sob-norm} we have $\norm{\mathbf U}_{(r,t)} \sim (t-r)^{1/2-b}$ and $\norm{\mathbf U}_{t} \sim t^{1/2-b}$, so if we estimate $B_2$ by using \eqref{Trivial-cutoff-estimate-2}, then we get
\[
  B_2 \le \frac{C}{R^2} \int_0^T t^{1-2b} \int_0^t (t-r)^{-4b} \, dr \, dt.
\]
But the right hand side equals $+\infty$ unless $b < 1/4$.
\end{remark}

\subsection{Cutoff estimate, version II}\label{Cutoff-II}

As a corollary to Lemma \ref{CutoffLemma}, we obtain the following variation on that result, proving the bound \eqref{MainCutoffLemma-2-1} in Proposition \ref{MainCutoffLemma-2}, as well as \eqref{MainCutoffLemma-2-2} in the case $\mathbf V=0$.

\begin{lemma}
\label{CutoffLemma2}
    For all $T \in (0, T_0]$ and $R > 0$ we have the estimates
    \begin{equation}
    \label{cutoff2}
        \norm{\theta_R\left( \norm{\mathbf U}_{t}^2\right) V}_{T} \le C \norm{V}_T,
    \end{equation}
    \begin{equation}
    \label{particular_cutoff2}
        \norm{\theta_R\left(\norm{\mathbf U}_{t}^2\right) \mathbf U}_{T}
        \le C \sqrt R
        ,
    \end{equation}
    with a constant $C$ depending only on $b$ and $T_0$.
\end{lemma}

\begin{proof} We first prove \eqref{cutoff2}. Recall the assumption $\supp\theta \subset [-2,2]$. So if $\norm{\mathbf U}_t^2 \ge 2R$ for all $t \in (0,T)$, then the left side of the inequality equals zero. It remains to consider the case where $\norm{\mathbf U}_t^2 < 2R$ for some $t \in (0,T)$. Define
\[
  T_{\mathbf U} = \sup \left\{ t \in (0,T) \colon \norm{\mathbf U}_t^2 < 2R \right\}.
\]
Then
\begin{equation}\label{Rbound}
  \norm{\mathbf U}_{T_{\mathbf U}}^2 \le 2R,
\end{equation}
by the continuity of $\norm{\mathbf U}_t$ with respect to $t > 0$ (see Lemma \ref{X-cont-lemma}). And if $T_{\mathbf U} < T$, then $\theta_R(\norm{\mathbf U}_{t}^2) = 0$ for $t \in [T_{\mathbf U},T]$. So \eqref{triangle} yields (if $T_{\mathbf U} = T$, this holds trivially)
\[
  \norm{\theta_R\left(\norm{\mathbf U}_{t}^2\right) V}_{T}
  \le C \norm{\theta_R\left(\norm{\mathbf U}_{t}^2\right) V}_{T_{\mathbf U}},
\]
and by Lemma \ref{CutoffLemma} the right hand side is dominated by
\[
  C\left( 1 + R^{-1} \norm{\mathbf U}_{T_{\mathbf U}}^2 \right) \norm{V}_{T_{\mathbf U}}
  \le 3C \norm{V}_{T_{\mathbf U}}
  ,
\]
where we used \eqref{Rbound}. By \eqref{increase2},
$\norm{V}_{T_{\mathbf U}} \le C\norm{V}_{T}$, which proves \eqref{cutoff2}. Taking now $V=U_j$ and using the bound \eqref{Rbound}, we get \eqref{particular_cutoff2}.
\end{proof}

\subsection{Difference estimate, version I}\label{Cutoff-III}

\begin{lemma}\label{DifferenceLemma} For all $T \in (0, T_0]$ and $R > 0$ we have the estimates
%
\begin{equation}
\label{difference}
  \norm{\theta_R\left(\norm{\mathbf U}_{t}^2\right) \mathbf U - \theta_R\left(\norm{\mathbf V}_{t}^2\right) \mathbf V}_{T}
  \le C \left( 1 + \frac{M^2}{R}
  + \frac{M^4}{R^2} \right)\norm{\mathbf U - \mathbf V}_T,
\end{equation}
where
\[
  M = \norm{\mathbf U}_T + \norm{\mathbf V}_T
\]
and the constant $C$ depends only on $b$ and $T_0$
\end{lemma}

\begin{proof}
Setting $\psi(t) = \theta_R(\norm{\mathbf U}_{t}^2)$ and $\chi(t) = \theta_R(\norm{\mathbf V}_{t}^2)$, we reduce to proving, for $1 \le j \le n$,
\[
  \norm{[\psi(t)-\chi(t)] U_j}_{T} \le \mathrm{r.h.s.}\eqref{difference}
\]
and
\[
  \norm{\chi(t) (U_j - V_j)}_{T} \le C\norm{U_j - V_j}_{T}.
\]
The latter holds by \eqref{cutoff2}, so we concentrate on the former, whose left hand side squared equals $A+2B$, where
\begin{align*}
  A &= \int_{\R^d} T^{-2b} \int_0^T \Abs{\psi(t)-\chi(t)}^2 \Abs{U_j(t,\xi)}^2 \, dt \, d\xi,
  \\
  B &= \int_{\R^d} \int_0^T \int_0^t \frac{\Abs{[\psi(t)-\chi(t)] U_j(t,\xi) - [\psi(r)-\chi(r)] U_j(r,\xi)}^2}{(t-r)^{1+2b}} \, dr \, dt \, d\xi.
\end{align*}

Using \eqref{Trivial-theta-bound} we estimate
\begin{equation}\label{alphabeta}
  \Abs{\psi(t)-\chi(t)}
  \le \frac{C}{R}
  \Abs{\norm{\mathbf U}_t^2-\norm{\mathbf V}_t^2}.
\end{equation}
Writing
\begin{equation}\label{uvexp}
  \norm{\mathbf U}_t^2-\norm{\mathbf V}_t^2 = \sum_{i=1}^n \left( \innerprod{U_i-V_i}{U_i}_t + \innerprod{V_i}{U_i-V_i}_t \right)
\end{equation}
and applying Cauchy-Schwarz yields
\begin{equation}\label{Diff1}
  \left( \norm{\mathbf U}_t^2-\norm{\mathbf V}_t^2 \right)^2
  \le
  C M^2 \norm{\mathbf U - \mathbf V}_{T}^2 \quad \text{for $0 \le t \le T$},
\end{equation}
implying $A \le C R^{-2} M^4 \norm{\mathbf U - \mathbf V}_{T}^2$, as desired.

It remains to bound the term $B$. Defining
\[
  \kappa = \psi-\chi
\]
we write $B \le 2(B_1 + B_2)$, where
\begin{align*}
  B_1 &= \int_{\R^d} \int_0^T  \int_0^t \Abs{\kappa(r)}^2 \frac{\Abs{U_j(t,\xi)-U_j(r,\xi)}^2}{(t-r)^{1+2b}} \, dr \, dt \, d\xi,
  \\
  B_2 &= \int_{\R^d} \int_0^T \Abs{U_j(t,\xi)}^2 \int_0^t \frac{\Abs{\kappa(t)-\kappa(r)}^2}{(t-r)^{1+2b}} \, dr \, dt \, d\xi.
\end{align*}

By \eqref{alphabeta} and \eqref{Diff1}, $\Abs{\kappa(r)}^2 \le C R^{-2} M^2 \norm{\mathbf U - \mathbf V}_{T}^2$,
so
\[
  B_1 \le C R^{-2} M^4 \norm{\mathbf U - \mathbf V}_{T}^2.
\]

In $B_2$ we write out
\[
  \kappa(t)-\kappa(r)
  = \theta_R(\norm{\mathbf U}_{t}^2) - \theta_R(\norm{\mathbf V}_{t}^2) - \left[\theta_R(\norm{\mathbf U}_{r}^2) - \theta_R(\norm{\mathbf V}_{r}^2) \right]
\]
and apply Lemma \ref{DiffLemma} to get
\[
  B_2 \le C \left( \frac{I}{R^4} + \frac{J}{R^2} \right),
\]
where
\begin{equation*}
  I = \int_{\R^d} \int_0^T \Abs{U_j(t,\xi)}^2 \left( \norm{\mathbf U}_{t}^2 - \norm{\mathbf V}_{t}^2 \right)^2
  \int_0^t \frac{\left( \norm{\mathbf U}_{t}^2 - \norm{\mathbf U}_{r}^2\right)^2
  + \left( \norm{\mathbf V}_{t}^2 - \norm{\mathbf V}_{r}^2 \right)^2}{(t-r)^{1+2b}} \, dr \, dt \, d\xi
\end{equation*}
and
\[
  J = \int_{\R^d} \int_0^T \Abs{U_j(t,\xi))}^2
  \int_0^t \frac{\left( \norm{\mathbf U}_{t}^2 - \norm{\mathbf V}_{t}^2 - \norm{\mathbf U}_{r}^2 + \norm{\mathbf V}_{r}^2 \right)^2}{(t-r)^{1+2b}} \, dr \, dt \, d\xi.
\]

By \eqref{Diff1} and Lemma \ref{LemmaKeyEst1},
\[
  I \le C_b M^8 \norm{\mathbf U - \mathbf V}_{T}^2,
\]
which is acceptable. Finally, using \eqref{uvexp} and Lemma \ref{LemmaKeyEst1},
\[
  J \le C_b M^4 \norm{\mathbf U - \mathbf V}_{T}^2,
\]
and this concludes the proof of the lemma.
\end{proof}

\subsection{Difference estimate, version II}\label{Cutoff-IV}

We are now in a position to finish the proof of Proposition \ref{MainCutoffLemma-2}, by proving the bound \eqref{MainCutoffLemma-2-2}.

\begin{lemma}\label{DifferenceLemma2} For all $T \in (0,T_0]$ and $R > 0$ we have the estimate
\begin{equation}\label{difference2}
  \norm{\theta_R\left(\norm{\mathbf U}_{t}^2\right) \mathbf U - \theta_R\left(\norm{\mathbf V}_{t}^2\right) \mathbf V}_{T}
  \le C \norm{\mathbf U-\mathbf V}_T,
\end{equation}
where the constant $C$ depends only on $b$ and $T_0$.
\end{lemma}

\begin{proof} First, if $\norm{\mathbf U}_{t}^2 \ge 2R$ and $\norm{\mathbf V}_{t}^2 \ge 2R$ for all $t \in [0,T]$, then the left side equals zero.

Second, consider the hybrid case where $\norm{\mathbf U}_{t}^2 \ge 2R$ for all $t \in [0,T]$, whereas $\norm{\mathbf V}_{t}^2 < 2R$ for some $t \in [0,T]$. Then $\theta_R(\norm{\mathbf U}_{t}^2) = 0$ for $t \in [0,T]$. Defining $T_{\mathbf V}$ as in the proof of Lemma \ref{CutoffLemma2}, we find
\[
  \mathrm{l.h.s.}\eqref{difference2}
  = \norm{\theta_R\left(\norm{\mathbf V}_{t}^2\right) \mathbf V}_{T}
  \le
  C \norm{\theta_R\left(\norm{\mathbf V}_{t}^2\right) \mathbf V}_{T_{\mathbf V}}
  \le
  C \sqrt{2R},
\]
by Lemma \ref{CutoffLemma2}. So if $\norm{\mathbf U-\mathbf V}_{T_{\mathbf V}} \ge \sqrt{R}$, we are done. If, on the other hand, $\norm{\mathbf U-\mathbf V}_{T_{\mathbf v}} \le \sqrt{R}$, then by the triangle inequality, $\norm{\mathbf U}_{T_{\mathbf V}} \le (1+\sqrt{2})\sqrt{R}$, since $\norm{\mathbf V}_{T_{\mathbf V}}^2 \le 2R$. Then \eqref{difference2} follows by applying Lemma \ref{DifferenceLemma} to
\[
  \norm{\theta_R(\norm{\mathbf V}_{t}^2) \mathbf V}_{T_{\mathbf v}} =
  \norm{\theta_R(\norm{\mathbf U}_{t}^2) \mathbf U - \theta_R(\norm{\mathbf V}_{t}^2) \mathbf V}_{T_{\mathbf V}}.
\]

The other hybrid case is symmetric, so we are only left with the case where $\norm{\mathbf U}_t^2 < 2R$ for some $t \in [0,T]$ and $\norm{\mathbf V}_s^2 < 2R$ for some $s \in [0,T]$. Define $T_{\mathbf U}$ and $T_{\mathbf V}$ as in the proof of Lemma \ref{CutoffLemma2}, so $\norm{\mathbf U}_{T_{\mathbf U}}^2 \le 2R$ and $\norm{\mathbf V}_{T_{\mathbf V}}^2 \le 2R$. By symmetry we may assume $T_{\mathbf U} \le T_{\mathbf V}$, and then \eqref{triangle} yields
\[
  \mathrm{l.h.s.}\eqref{difference2}
  \le
  C \norm{\theta_R(\norm{\mathbf U}_{t}^2) \mathbf U - \theta_R(\norm{\mathbf V}_{t}^2) \mathbf V}_{T_{\mathbf V}}.
\]
If $\norm{\mathbf U}_{T_{\mathbf V}}^2 \le 8R$, we now obtain \eqref{difference2} by Lemma \ref{DifferenceLemma}. If, on the other hand, $\norm{\mathbf U}_{T_{\mathbf V}}^2 > 8R$, then $\norm{\mathbf U - \mathbf V}_{T_{\mathbf V}} \ge \sqrt{2R}$, and using \eqref{triangle} we obtain
\[
  \mathrm{l.h.s.}\eqref{difference2}
  \le
  C \left( \norm{\theta_R(\norm{\mathbf U}_{t}^2) \mathbf U - \theta_R(\norm{\mathbf V}_{t}^2) \mathbf V}_{T_{\mathbf U}} +
  \norm{ \theta_R(\norm{\mathbf V}_{t}^2) \mathbf V}_{(T_{\mathbf U},T_{\mathbf V})} \right).
\]
The first term on the right can be handled by Lemma \ref{DifferenceLemma}, since $\norm{\mathbf U}_{T_{\mathbf U}}^2 \le 2R$ and  $\norm{\mathbf V}_{T_{\mathbf U}}^2 \le C\norm{\mathbf V}_{T_{\mathbf V}}^2 \le C2R$. For the second term, we get by \eqref{increase2} and Lemma \ref{CutoffLemma2},
\[
  \norm{ \theta_R(\norm{\mathbf V}_{t}^2) \mathbf V}_{(T_{\mathbf U},T_{\mathbf V})}
  \le
  C \norm{ \theta_R(\norm{\mathbf V}_{t}^2) \mathbf V}_{T_{\mathbf V}}
  \le
  C' \norm{ \mathbf V}_{T_{\mathbf V}}
  \le
  C' \norm{\mathbf U - \mathbf V}_{T_{\mathbf V}},
\]
where we used $\norm{\mathbf U - \mathbf V}_{T_{\mathbf V}} \ge \sqrt{2R} \ge \norm{\mathbf V}_{T_{\mathbf V}}$.
This concludes the proof of the lemma.
\end{proof}

\subsection{Sobolev–Slobodeckij norm on $H^b(0, T)$}
\label{only-time-cutoff}

All the properties discussed above in this section
are valid for usual Sobolev norms,
for functions depending only on the time variable $t$.
The proofs can be repeated directly without much of a difference.
However, we can get also those properties easily by
considering
\[
    u(t) = \phi(t) S_{h(\xi)}(t) f
\]
with  $\phi \in H^b(S,T)$ and $f \in H^s(\R^d)$.
Indeed, from \eqref{strong_XRestNorm_alternative} and \eqref{Bourgain_Slobodeckij_norm} it is clear that
\[
  \norm{u}_{X^{s,b}_{h(\xi)}(S,T)} = \norm{\phi}_{H^b(S,T)} \norm{f}_{H^s}
  \quad \text{and} \quad
  \norm{u}_{\widetilde X^{s,b}_{h(\xi)}(S,T)} = \norm{\phi}_{\widetilde H^b(S,T)} \norm{f}_{H^s},
\]
where
\[
  \norm{\phi}_{\widetilde H^b(S,T)} =
  \frac{1}{(T-S)^{2b}} \int_S^T \Abs{\phi(t)}^2 \, dt
  + \int_S^T \int_S^T \frac{\Abs{\phi(t)-\phi(r)}^2}{\Abs{t-r}^{1+2b}} \, dr \, dt.
\]
Normalising by taking $\norm{f}_{H^s} = 1$, we then get from Lemma \ref{NormLemma}, for $b \in (0,1/2)$,
\begin{equation}
    \label{Sob-ModEquivalence}
    C_{T_0,b}^{-1} \norm{\phi}_{H^b(S,T)}
    \le
    \norm{\phi}_{\widetilde H^b(S,T)}
    \le C_{T_0,b}
    \norm{\phi}_{H^b(S,T)} \qquad (0 \le S < T \le T_0),
\end{equation}
and Proposition \ref{MainCutoffLemma} has the following analogue.

\begin{proposition}
    Let $T_0 > 0$ and $b \in (0,1/2)$. Let $n \in \N$ and suppose that $\phi_i,\Phi_i \in H^b(0,T_0)$ for $1 \le i \le n$. Then for $T \in (0,T_0]$, $R > 0$ and $1 \le j \le n$ we have the estimates
    \[
        \norm{
            \theta_R\left( \sum_{i=1}^n \norm{\phi_i}_{\widetilde H^{b}(0,t)}^2 \right)
            \phi_j(t)
        }_{H^{b}(0,T)}
        \leqslant
        C \sqrt{R}
        ,
    \]
    \[
        \norm{
            \theta_R\left( \sum_{i=1}^n \norm{\phi_i}_{\widetilde H^{b}(0,t)}^2 \right)
            \phi_j(t)
            -
           \theta_R\left( \sum_{i=1}^n \norm{\Phi_i}_{\widetilde H^{b}(0,t)}^2 \right)
            \Phi_j(t)
        }_{H^{b}(0,T)}
        \leqslant
        C \sum_{i=1}^n \norm{\phi_i-\Phi_i}_{H^{b}(0,T)}
        ,
    \]
    where $C$ depends only on $b$, $T_0$ and $\theta$.
\end{proposition}



\vskip 0.05in
\noindent
{\bf Acknowledgments.}
{
	ED was partially supported by the ERC EU project 856408-STUOD
    and
    the Research Council of Norway through its Centres
    of Excellence scheme, Hylleraas-senteret, project number 262695.
    SS was partially supported by the Meltzer Research Fund.
    The authors would like to thank the anonymous referees,
    whose valuable comments helped us to improve the paper during the revision.
}

\bibliographystyle{acm}
\bibliography{bibliography}

\end{document}